\documentclass[final,leqno,onefignum,onetabnum]{siamltex}

\usepackage{amsmath,amssymb,euscript ,yfonts,psfrag,latexsym,dsfont,graphicx,bbm,color,amstext,wasysym,epsfig,subfig,parskip,textcomp}
\usepackage[normalem]{ulem}

\usepackage{bm,subfig,color}
\usepackage{algorithmic}
\usepackage{algorithm}
\usepackage{graphics}

\usepackage{amsfonts,url}

\usepackage[normalem]{ulem}
\usepackage{mathtools}

\usepackage{url}

\graphicspath{{./},{./arXiv/}}

\newtheorem{remark}[theorem]{Remark}
\newtheorem{problem}[theorem]{Problem}
\newtheorem{example}[theorem]{Example}

\def\R{\mathbb{R}}
\def\C{\mathbb{C}}

\def\D{\mathbb{D}}

\newcommand{\cS}{{\mathcal S}}
\newcommand{\cC}{{\mathcal C}}
\newcommand{\cD}{{\mathcal D}}
\newcommand{\cK}{{\mathcal K}}
\newcommand{\cE}{{\mathcal E}}

\newcommand{\cP}{{\mathcal P}}
\newcommand{\cT}{{\mathcal T}}

\def\exp{{\rm exp}\,}

\newcommand{\eq}{\begin{equation}}
\newcommand{\eeq}{\end{equation}}
\newcommand{\eqn}{\begin{eqnarray}}
\newcommand{\eeqn}{\end{eqnarray}}

\newcommand{\De}{\mathrm{d}}

\newcommand{\qed}{\hfill $\Box$ \vskip 2ex}

\def\bmat{\left[ \begin{array}}
\def\emat{\end{array} \right]}

\newcommand{\fM}{{\mathfrak M}}
\newcommand{\fP}{{{\mathfrak P}}}
\newcommand{\support}{{\it Supp}}

\newcommand{\bsea}{\begin{subeqnarray}}
\newcommand{\esea}{\end{subeqnarray}}

\def\bmat{\left[ \begin{array}}
\def\emat{\end{array} \right]}

\newcounter{acount}

\definecolor{Royalblue}{cmyk}{1,0.30,0.2,0.2}

\newfont{\BB}{msbm10 scaled\magstep1}
\def\E{\mbox{\BB E}}

\def\spacingset#1{\def\baselinestretch{#1}\small\normalsize}

\spacingset{1}

\title{Stochastic control liaisons:\\ Richard Sinkhorn meets Gaspard Monge\\ on  a Schr\"{o}dinger bridge  \thanks{Supported in part by the
NSF under grant 1807664, 1839441, 1901599, 1942523, the AFOSR under grants FA9550-17-1-0435,} and by the University of Padova Research Project CPDA 140897.}

\author{Yongxin Chen \thanks{School of Aerospace Engineering, Georgia Institute of Technology, Atlanta, GA 30332, USA. ({\tt yongchen@gatech.edu}).}
        \and Tryphon T. Georgiou \thanks{Department of Mechanical and Aerospace Engineering,
University of California, Irvine, CA 92697, USA. ({\tt  tryphon@uci.edu})}
         \and Michele Pavon \thanks{Dipartimento di Matematica ``Tullio Levi-Civita",
Universit\`a di Padova, via Trieste 63, 35121 Padova, Italy. ({\tt  pavon@math.unipd.it})}}

\begin{document}

\maketitle

\begin{abstract}
In 1931/32, Erwin Schr\"odinger studied a {\em hot gas Gedankenexperiment} (an instance of large deviations of the empirical distribution). Schr\"odinger's problem represents an early example of a fundamental {\em inference method}, the so-called maximum entropy method, having roots in Boltzmann's work and developed in subsequent years by Jaynes, Burg, Dempster and Csisz\'{a}r. The problem, known as the {\em Schr\"odinger bridge problem} (SBP) with ``uniform" prior, was more recently recognized as a regularization of the Monge-Kantorovich {\em Optimal Mass Transport} (OMT) problem, leading to effective computational schemes for the latter.   Specifically, OMT with quadratic cost may be viewed as a zero-temperature limit of the problem posed by Schr\"odinger in the early 1930s.
The latter amounts to minimization of the  {\em Helmholtz's free energy} over probability distributions that are constrained to possess two given marginals. The problem features a delicate compromise, mediated by a temperature parameter, between minimizing the internal energy and maximizing the entropy. These concepts are central to a rapidly expanding area of modern science dealing with
the so-called {\em Sinkhorn algorithm} which appear as a special case of an algorithm first studied in the more challenging continuous-space setting by the French analyst Robert Fortet in 1938/40 specifically for Schr\"{o}dinger bridges. Due to the constraint on end-point distributions, dynamic programming is not a suitable tool to attack these problems. Instead, Fortet's iterative algorithm and its discrete counterpart, the Sinkhorn iteration, permit computation of the optimal solution by iteratively solving the so-called {\em Schr\"{o}dinger system}. Convergence of the iteration is guaranteed by contraction along the steps in suitable metrics, such as Hilbert's projective metric. 

In both the continuous as well as the discrete-time and space settings, {\em stochastic control} provides a reformulation and a context for the dynamic versions of general Schr\"odinger bridges problems and of its zero-temperature limit, the OMT problem. These problems, in turn, naturally lead to {\em steering problems for flows of one-time marginals} which represent a new paradigm for {\em controlling uncertainty}. The zero-temperature problem in the continuous-time and space setting turns out to be  the celebrated Benamou-Brenier characterization of the {\em McCann displacement interpolation} flow in OMT.
The formalism and techniques behind these control problems on flows of probability distributions have attracted significant attention in recent years as they lead to a variety of new applications in spacecraft guidance, control of robot or biological swarms, sensing, active cooling, network routing as well as in computer and data science.
This multifacet and versatile framework, intertwining SBP and OMT, provides the substrate for a historical and technical overview of the field taken up in this paper. A key motivation has been to highlight links between the classical early work in both topics and the more recent stochastic control viewpoint, that naturally lends itself to efficient computational schemes and interesting generalizations.
\end{abstract}


\tableofcontents

\pagestyle{myheadings}
\thispagestyle{plain}
\markboth{YONGXIN CHEN, TRYPHON GEORGIOU and MICHELE PAVON}{RICHARD SINKHORN  MEETS GASPAR MONGE}

\section{Prelude: Sinkhorn's algorithm}

In 1962, Richard Sinkhorn completed his doctoral dissertation entitled ``On Two Problems Concerning Doubly Stochastic Matrices" and submitted \cite{Sin64} which would appear in {\em The Annals of Mathematical Statistics} in 1964. He showed there that the iterative process of alternatively normalizing the rows and columns of a matrix $A$ with striclty positive elements converges to a doubly stochastic matrix $D_1AD_2$. Here $D_1$ and $D_2$ are diagonal matrices with positive diagonal elements which are unique up to a scalar factor. By 1970, a survey by Fienberg had appeared \cite{Fien}, where many other contributions to this topic, including \cite{Sin67,SinKno67,IK}, were cited. Differently from Sinkhorn,  Ireland and Kullback \cite{IK}, who studied a more general problem than Sinkhorn,  credited Deming and Stephan \cite{DS1940} (1940) for first introducing the {\em iterative proportional fitting} (IPF)  procedure\footnote{Even earlier contributions are \cite{Yu,Ku}.}. Several other significant contributions on this problem followed including generalizations to multidimensional matrices  such as \cite{Bap,FL} and this line of research continues to this date, see e.g. \cite{Cuturi,Alt,TCDP,PC,BHDJ,DiMG}. This large body of literature often ignores two papers by Erwin Schr\"{o}dinger from the early 1930's \cite{S1,S2} as well as important contributions by Robert Fortet (1938/40) \cite{For0,For} and Arne Beurling (1960) \cite{Beu}. The same goes for some later contributions on the {\em Schr\"{o}dinger system} \cite{Jam2,F2}.
But why should this bulk of work on Sinkhorn algorithms  be related to Schr\"{o}dinger's quest for the most likely random evolution between two given marginals for a cloud of diffusive particles?


One of the main goals of this paper is to provide an exhaustive answer to this question.
To do that, we shall have to address several other questions such as: What is the relation between Schr\"{o}dinger's problem and the seventeen-hundred's Monge ``M\'emoire sur la th\'eorie des d\'eblais et des
remblais" (1781) \cite{Monge}? Was Schr\"{o}dinger interested in {\em regularizing} the Optimal Mass Transport (OMT) problem as some recent papers seem to hint \cite{BCCNP,CPSV,LYO}? And how does the latter regularization relate to von Helmholz's {\em free energy} of thermodynamics \cite{Hel}? Is there a relation to Boltzmann's maximum entropy problem (1877) \cite{BOL}? Is there a connection to multiplicative functional transformations of Markov processes \cite{IW}, or to Bernstein's reciprocal processes \cite{Ber, Jam1}? And what about connections to the Fisher information functional \cite{Vil} and to positive maps on cones contracting Hilbert's projective metric \cite{birkhoff1957extensions,bushell1973hilbert}? Lastly and surprisingly, but definitely not least, how is this scientific exploration intertwined with Democritus atomic hypothesis?

As it turns out, all of these topics are indeed tightly connected. Therefore, what we would like to discuss in this paper lies right at the crossroad of major areas of science, some still in rapid development. Clearly, given the overwhelming spectrum of ideas and concepts, attempting to sort this all out may result  in a fuzzy picture. Given our limited competence, how can we ever hope to give here at least a reasonable/interesting account of all these intersecting areas (sutor, ne ultra crepidam!)? Our choice is to discuss this science junction from an angle which is not the most common in the literature, namely {\em stochastic control}. We shall try to provide evidence that, once Schr\"odinger's problem has been converted to a stochastic control problem, it lends itself naturally to computational schemes as well as to interesting generalizations. It leads naturally to a {\em steering problem for probability distributions}, namely a relaxed version of a most central problem in deterministic optimal control \cite{FR}.

Our tale is apparently going to be a long and complex one with permanent danger of too much branching out and continuous flashbacks. The only way to avoid total chaos seems to us is starting with Erwin Schr\"{o}dinger's drama in the early 1930's.   
\section{Overture: Science dramas}
We briefly recall two famous dramas in the history of science.
\subsection{Schr\"{o}dinger's drama}
Edward Nelson
concludes his jewel 1967 book \cite{N1} with a 1926 quotation from Erwin Schr\"{o}dinger \cite[Paragraph 14]{S0}: ``... It has even been doubted whether what goes on in the atom could ever be described within the scheme of space and time. From the philosophical standpoint, I would consider a conclusive decision in this sense as equivalent to a complete surrender. For we cannot really alter our manner of thinking in space and time, and what we cannot comprehend within it we cannot understand at all."

Several
years later, in 1953, Schr\"{o}dinger wrote \cite{S3}: ``For it must have given to de Broglie the same shock and disappointment as it gave to me, when we learnt that a sort of transcendental, almost psychical interpretation of the wave phenomenon had been put forward, which was very soon hailed by the majority of leading theorists as the only one reconcilable with experiments, and which has now become the orthodox creed, accepted by almost everybody, with a few notable exceptions."\footnote{
According to Nelson,
 ``
 a realistic interpretation of quantum mechanics is, in my view, as unresolved as it was in the 1920s.", \cite[p.230]{N2}.}

Between these two dramatic statements, however, there is a time in the early thirties when Schr\"{o}dinger has hope again. It is when he introduces what we now call the {\em Schr\"{o}dinger bridges} in two remarkable papers \cite{S1,S2}. He states: ``Merkw\"{u}rdige Analogien zur Quantenmechanik, die mir sehr des Hindenkens wert erscheinen"\footnote{``Remarkable analogies to quantum mechanics which appear to me very worth of reflection".}. In this respect, the title of \cite{S1} is revealing: ``\"{U}ber die Umkehrung der Naturgesetze", namely ``On the reversal of natural laws".  A few years later in 1937, another scientific giant, Andrey Kolmogorov, publishes a paper \cite{Kol} with a very similar title ``Zur Umkehrbarkeit der statistischen Naturgesetze", namely ``On the reversibility of statistical natural laws". What had Schr\"{o}dinger glimpsed? Before we discuss this, let's go back some 120 years to another dramatic moment in the history of science.
 
\subsection{Fourier's drama} On December 21st, 1807, Joseph B.\ Fourier submits a manuscript to the Institute of France in Paris entitled ``Sur la propagation de la chaleur". He had been working on this memoir during his stay in Grenoble as Prefect of the Department of Is\'ere, a post to which he had been appointed by Napoleon himself. The surprising, but not fully justified, results provoke an animated discussion among the examiners. The committee consists of Lagrange, Laplace, Lacroix and  Monge. Lagrange and Laplace, who criticize Fourier's expansion  of functions as trigonometrical series, and Monge had all been teachers of Fourier at the \'Ecole Normale.  Here is how Fourier describes Monge \cite{Fourier0}: ``Having a loud voice and is active, ingenious and very learned." In 1785, while teaching at the \'Ecole Militaire in Paris, Monge, the father of {\em descriptive geometry}, see Figure \ref{DG},
\begin{figure}\begin{center}
{\includegraphics[height=.7\textwidth]{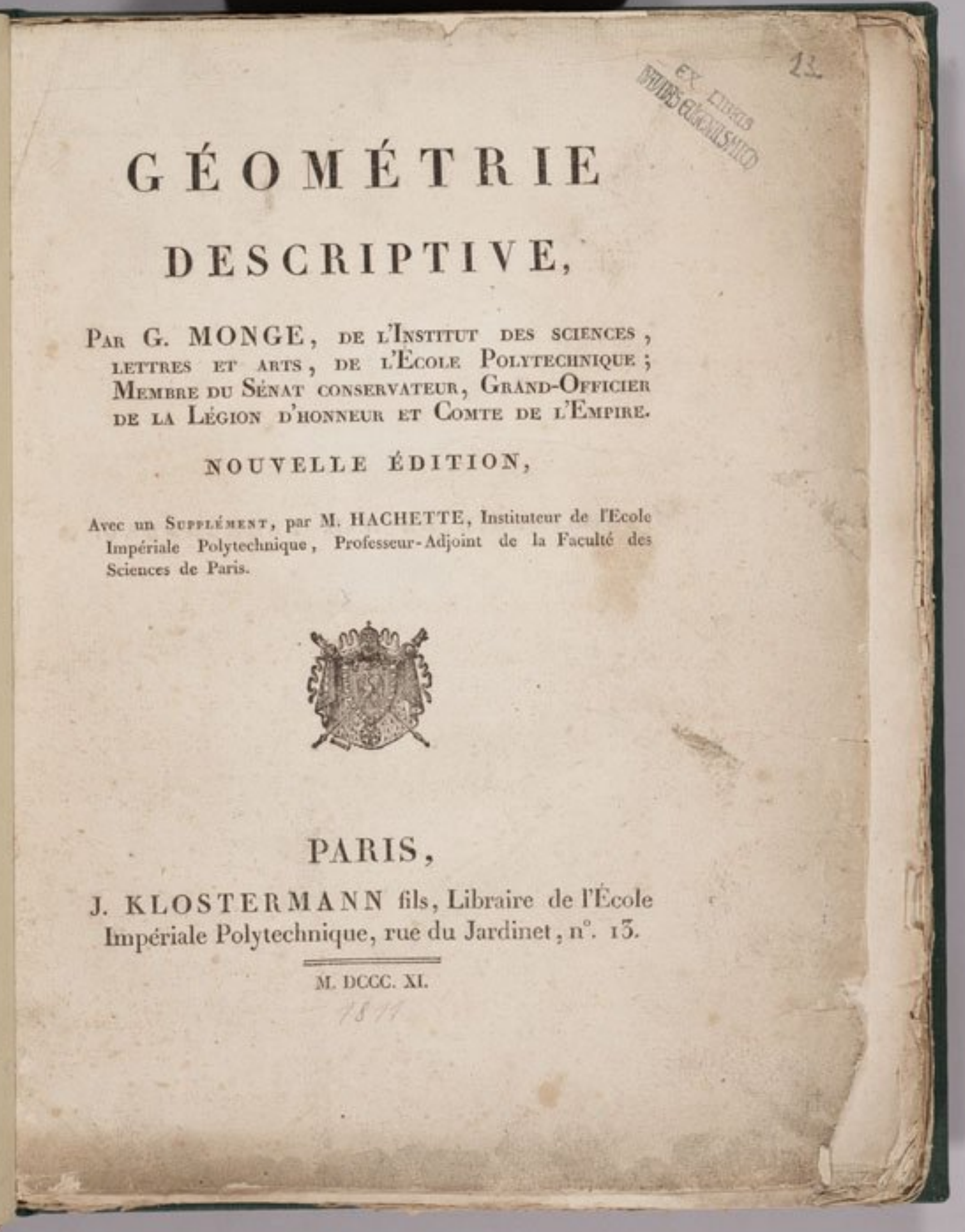}}
\caption{G. Monge, G\'eom\'etrie Descriptive}
\label{DG}
\end{center}\end{figure}
has among his students a sixteen years old Corsican gifted for mathematics by the name of Napoleone Bonaparte. The latter will be examined upon graduation by Pierre-Simon de Laplace. 
Of the group of outstanding mathematicians active in France in that time, Fourier and Monge are those who take on important political positions. Monge is Minister of the Marine during the French revolution.  In between two assignements in Italy, he directs the \'Ecole Polytechnique which he had co-established in 1794 with Lazare Carnot and Napoleon. Fourier and Monge had also been together in the expedition to Egypt in 1798 and members of the mathematics division of the Cairo Institute together with Malus and Napoleon himself. In spite of all of this,  the manuscript is eventually rejected (Fourier's ``Th\'eorie analytique de la chaleur" will appear only in 1822).\footnote{In 1826, Fourier announced a method for the solution of systems of
linear inequalities \cite{Fourier} which has elements in common with the simplex method of Linear Programming.}

Let us zoom in on Gaspard Monge, Comte de P\'eluse. A first version (unfortunately lost) of his ``M\'emoire sur la th\'eorie des d\'eblais et des remblais" is read at the Acad\'emie des sciences on January 27 and February 7, 1776. Although the memoir is proposed for publication by the secretary Condorcet, only on March 28, 1781 does Monge read a second version of his memoir, cf.\ Figure \ref{fig:Imagemarginals}.
\begin{figure}\begin{center}
{\includegraphics[height=.6 \textwidth,width=.6\textwidth]{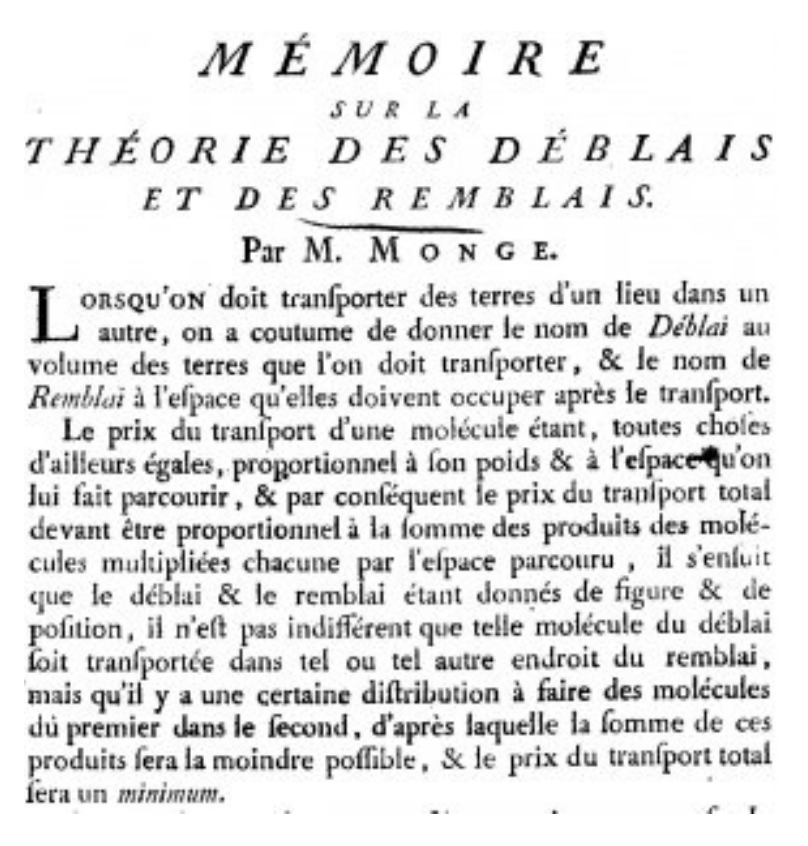}}\\
\includegraphics[height=0.4\textwidth,width=0.7\textwidth]{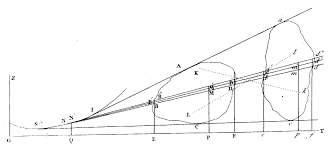}
\caption{Front page of Monge's memoire and Monge's drawing}
\label{fig:Imagemarginals}
\end{center}\end{figure}
A forty page publication follows in 1784. Monge, who is of rather humble origin, had worked for military institutions for a while taking advantage of his extraordinary drawing skills and geometric intuition, see below.
\begin{figure}\begin{center}
{\includegraphics[height=0.5 \textwidth,width=.7\textwidth]{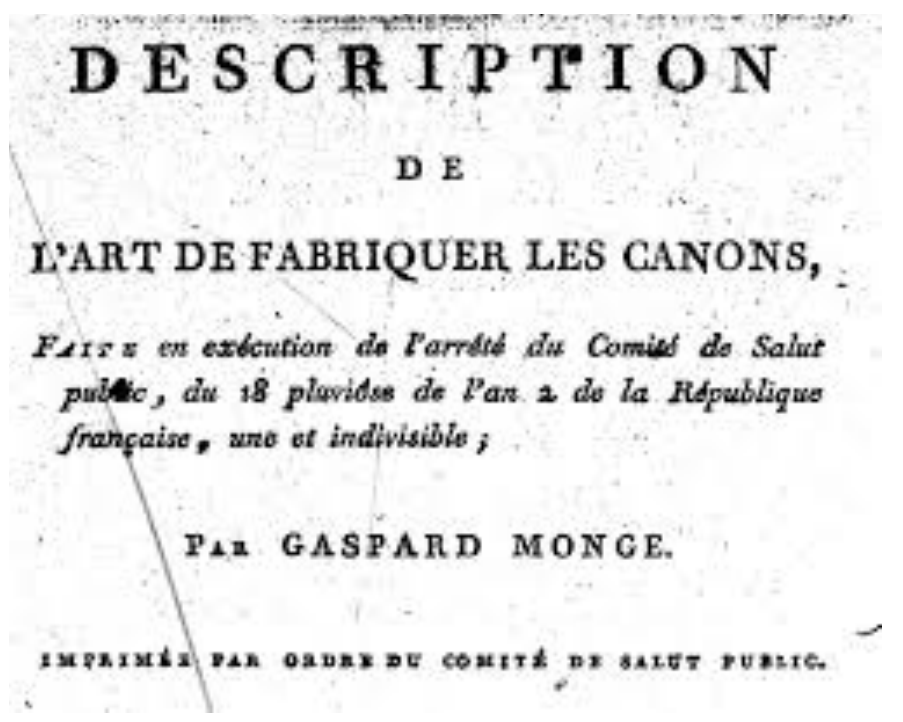}}
\caption{G. Monge, L'Art de fabriques les Canons}
\label{cannons}
\end{center}\end{figure}
Monge had applied descriptive geometry to such problem as designing cannons (Figure \ref{cannons}), cutting stones, planning city walls and drawing shadows.
But now 
he is interested in the following problem which has both civil and military applications: Suppose you need to build some embankments carrying debris from another location: How should the transport occur so that the average distance is minimized? Monge discusses two and three-dimensional problems showing a profound understanding of  the problem and of its challenging aspects. In particular, his intuition of normality of optimal transport paths to a certain one-parameter family of surfaces was 
proven to be correct more than one hundred years later in a 200-page memoire by Appel \cite{appel}. 
\subsection{Kantorovich}
Not much happens until the 1920's and 1930's when the transportation problem is first studied mathematically by A.N. Tolstoi \cite{Tol1,Tol2}. Then, in 1939,  the transport problem is briefly mentioned by the Soviet mathematician Leonid Kantorovich in the booklet  \cite{Kan1} where he lays the foundations of linear programming including duality theory and a variant of the simplex method.  In 1942, Kantorovich provided in \cite{Kan2} the generalization of linear programming to an abstract setting. He considers the following problem: Let $X$ be compact metric space with distance function $d$. Let $\nu_0$ and $\nu_1$ be probability measures on $X$ and let $\Pi(\nu_0,\nu_1)$ denote all the probability distributions on $X\times X$ having $\nu_0$ and $\nu_1$ as marginal distributions. Consider the problem
\begin{equation}\label{problemkantorovich}
K(\nu_0,\nu_1)=\inf_{\pi\in\Pi(\nu_0,\nu_1)}\int_{X\times X} d(x,y)d\pi(x,y).
\end{equation}
Kantorovich establishes the following fundamental duality theorem (see also Theorem \ref{duality} for the case of a non-distance cost function):
\begin{theorem}
\begin{equation}
K(\nu_0,\nu_1)=\sup_{\varphi}\left\{\int \varphi(x)d\left(\nu_0-\nu_1\right);\|\varphi\|_{{\rm Lip}}\le 1,\; \varphi\in L^1(|\nu_0-\nu_1|)\right\}.
\end{equation}
\end{theorem}

Lipschitz functions with constant 1 relative to the metric $d$ are called {\em potentials} by Kantorovich.
The transport plan $\pi$ is called a {\em potential plan} if
\[\varphi(x)-\varphi(y)= d(x,y),\;\; \pi \;{\rm a.s.}
\]
for an (optimal) potential function $\varphi$ (this theorem will be complemented in 1958 by the paper \cite{Kan4} joint with his student Rubinstein). In 1947 Kantorovich \cite{Kan3}, seeing the proceeding of a conference held in Leningrad (St. Petersburg) on the bicentennial of Monge's birth, realized that the surfaces of Monge were just the
level surfaces of the optimal potentials (dual functions) he had defined in the Doklady note \cite{Kan2}.
Thus, Kantorovich, a functional analyst motivated by economics applications,  provided a more manageable (relaxed) formulation of the transport problem and major advances opening the avenue for the impressive developments of the past twenty years \cite{RR,E,Vil,AGS,Vil2, OPV}, see \cite{Ver} for a full historical account of Kantorovich's contributions\footnote{It would have been possible to entitle this section ``Kantorovich's Drama". Indeed, to this day it is little known that he should be credited for linear programming, including the simplex method, and for duality theory even in an abstract setting. His metric is curiously called Wasserstein (Vasershtein) due to Dobrushin being aware of Vasershtein paper \cite{Vas} (Vasershtein worked in the laboratory he headed) and not being aware of Kantorovich's publications. Finally, Kantorovich's  ideas on
mathematical economics were long considered in official Soviet circles as
anti-Marxist. Consequently they suffered for many years a sort of ostracism. For all of this see \cite{Ver}. Fortunately, as a partial compensation, he was awarded in 1975 the Nobel price for economics.}.

\section{Elements of optimal mass transport theory}\label{OMT}

The literature on this problem is by now so vast and our degree of competence is such that we shall not even attempt here to give a reasonable and/or balanced introduction to the various fascinating aspects of this theory. Fortunately, there exist excellent monographs and survey papers on this topic, see \cite{RR,E,Vil,AGS,Vil2, OPV}, to which we refer the reader. The range of applications has also increased exponentially; we mention quality control, industrial manufacturing, vehicle path planning \cite{RT1,RT2},
image processing \cite{BCCNP}, computer graphics \cite{Sol1, Sol2, Sol4, Sol5, Sol6,Sol8}, machine learning \cite{Sol3,Arj17, Sol7}, econometrics \cite{AGal}, and so on. We shall only briefly review some concepts and results which are relevant for the topics of this paper.
\subsection{The Monge-Kantorovich static problem}

Let $\nu_0$ and $\nu_1$ be probability measures on the separable, complete metric spaces $X$ and $Y$, respectively. Let $c:X\times Y\rightarrow [0,+\infty)$ be a lower semicontinuous  map with $c(x,y)$ representing the cost of transporting a unit of mass from location $x$ to location $y$.
Let ${\cal T}_{\nu_0\nu_1}$ be the family of  measurable maps $T:X\rightarrow Y$ such that $T\#\nu_0=\nu_1$, namely such that $\nu_1$ is the {\em push-forward} of $\nu_0$ under $T$. Any $T\in {\cal T}_{\nu_0\nu_1}$ is called a {\em transport map}. Then, Monge's optimal mass transport problem (OMT) is
\begin{equation}\label{monge}\inf_{T\in {\cal T}_{\nu_0\nu_1}}\int_X c(x,T(x))d\nu_0(x)
\end{equation}
for the particular case where $c(x,y)=d(x,y)$, i.e., it is a metric on $X$ and $X=Y$.
This problem may be unfeasible and the family ${\cal T}_{\nu_0\nu_1}$ may be empty\footnote{This is the case when e.g., $\nu_0$ is a Dirac distribution and $\nu_1$ the sum of two Dirac distributions of half the magnitude. Since $\nu_0$ needs to be ``split'' so as to be transferred at two separate locations, a transference plan $\cal T$ does not exist as a map.}. This is never the case for the ``relaxed" version of the problem studied by Kantorovich in the 1940's
\begin{equation}\label{kantorovich}
\inf_{\pi\in\Pi(\nu_0,\nu_1)}\int_{X\times Y}  c(x,y)d\pi(x,y).
\end{equation}
Here $\Pi(\nu_0,\nu_1)$ are ``couplings" of $\nu_0$ and $\nu_1$, namely probability distributions on $X\times Y$ with marginals $\nu_0$ and $\nu_1$ called {\em transport plans}. Indeed, $\Pi(\nu_0,\nu_1)$ always contains the product measure $\nu_0\otimes\nu_1$. We have Kantorovich's duality theorem.
\begin{theorem}\label{duality} Suppose $c$ is lower semicontinuous, then  there exists a  solution to Problem {\em (\ref{kantorovich})}. Moreover 
\begin{eqnarray}\nonumber
\min_{\pi\in\Pi(\nu_0,\nu_1)}\int_{X\times Y}c(x,y)d\pi(x,y)=\sup_{(\varphi,\psi)\in \Phi_c}\left[\int_X\varphi d\nu_0+\int_Y\psi d\nu_1\right]\\\nonumber \Phi_c=\{(\varphi,\psi)| \varphi\in L^1(\nu_0), \psi\in L^1(\nu_1), \varphi(x)+\psi(y)\le c(x,y)\}.
\end{eqnarray}
\end{theorem}

Let us specialize the Monge-Kantorovich problem (\ref{kantorovich}) to the case $X=Y=\R^n$ and $c(x,y)=\|x-y\|^2$. Then, if $\nu_1$ does not give mass to sets of dimension $\le n-1$, by Brenier's theorem \cite[p.66]{Vil}, there exists a unique optimal transport plan $\pi$ (Kantorovich) induced by a $d\nu_0$ a.e.\ unique (Monge) map $T$, of the form $T=\nabla\varphi$ with $\varphi$ convex, so that 
\begin{equation}\label{optmap}\pi=(I\times\nabla\varphi)\#\nu_0, \quad \nabla\varphi\#\nu_0=\nu_1.
\end{equation}
Here $I$ denotes the identity map. Among the extensions of this result, we mention that to strictly convex, superlinear costs $c$, by Gangbo and McCann \cite{GM}.
The optimal transport problem may be used to introduce a useful distance between probability measures. Indeed, let $\mathcal P_2(\R^n)$ be the set of probability measures $\mu$ on $\R^n$ with finite second moment. For $\nu_0, \nu_1\in\mathcal P_2(\R^n)$, the Kantorovich-Wasserstein quadratic distance, is defined by
\begin{equation}\label{Wasserdist}
W_2(\nu_0,\nu_1)=\left(\inf_{\pi\in\Pi(\nu_0,\nu_1)}\int_{\R^n\times\R^n}\|x-y\|^2d\pi(x,y)\right)^{1/2}. 
\end{equation}
As is well known \cite[Theorem 7.3]{Vil}, $W_2$ is a {\em bona fide} distance. Moreover, it provides a most natural way to  ``metrize" weak convergence\footnote{Here, we say that, as $k\to\infty$, $\mu_k$ converges weakly to $\mu$ in $\mathcal P_2(\R^n)$ if $\int_{\R^n}fd\mu_k\rightarrow\int_{\R^n}fd\mu$ for any continuous function $f$ satisfying $f(x)\le c\left(1+d(x,x_0)^2\right)$ for every $x_0\in\R^n$, the latter condition guaranteeing tightness of the sequence $\{\mu_k\}$.}
in $\mathcal P_2(\R^n)$ \cite[Theorem 7.12]{Vil}, \cite[proposition 7.1.5]{AGS} (the same applies to the case $p\ge 1$ replacing $2$ with $p$ everywhere). The {\em Kantorovich-Wasserstein space} $\mathcal W_2$ is defined as the metric space $\left(\mathcal P_2(\R^n),W_2\right)$. It is a {\em Polish space}, namely a separable, complete metric space.

\subsection{The dynamic problem}

So far, we have dealt with {\em the static} optimal transport problem. Nevertheless, in \cite[p.378]{BB} it is observed that ``...a continuum mechanics formulation was already implicitly contained in the original problem addressed by Monge... Eliminating the time variable was just a clever way of reducing the dimension of the problem". Thus, a {\em dynamic} ({\em Eulerian}) version of the OMT problem was already {\em in fieri} in Gaspar Monge's 1781 {\em ``M\'emoire sur la th\'eorie des d\'eblais et des remblais"}\,! It was elegantly accomplished by Benamou and Brenier in \cite{BB} by showing that 
\begin{subequations}\label{eq:BB}
\begin{eqnarray}\label{BB1}&&W_2^2(\nu_0,\nu_1)=\inf_{(\mu,v)}\int_{0}^{1}\int_{\R^n}\|v(t,x)\|^2\mu_t(dx)dt,\\&&\frac{\partial \mu}{\partial t}+\nabla\cdot(v\mu)=0,\label{BB2}\\&& \mu_0=\nu_0, \quad \mu_1=\nu_1.\label{Bboundary}
\end{eqnarray}\end{subequations}
Here the flow $\{\mu_t; 0\le t\le 1\}$ varies over continuous maps from $[0,1]$ to $\mathcal P_2(\R^n)$ and $v$ over smooth fields.  Benamou and Brenier were motivated by computational considerations, a topic which had not received much attention in OMT, see \cite{AHT}. In \cite{Vil2}, Villani states at the beginning of Chapter $7$ that two main motivations for the time-dependent version of OMT are
\begin{itemize}
\item a time-dependent model gives a more complete description of the
transport;
\item the richer mathematical structure will be useful later on.
\end{itemize}
We can add three further reasons:
\begin{enumerate}
\item it allows to view the optimal transport problem as an (atypical) {\em optimal control} problem, see Section \ref{SCOMT} below and  \cite{CGP1}-\cite{CGP10};
\item it opens the way to establish a connection with the {\em Schr\"{o}dinger bridge} problem, where the latter appears as a regularization of the former \cite{Mik, mt, MT,leo,leo2,CGP3,CGP4,CDPS};
\item In some applications such as computer graphics \cite{Sol4, Sol8}, interpolation of images \cite{CGP5}, spectral morphing \cite{JLG}, machine learning \cite{Sol7} and network routing \cite{CGPT2,CGPT2,CGPT3}, the interpolating flow is essential!
\end{enumerate}
Let $\{\mu^*_t; 0\le t\le 1\}$ and $\{v^*(t,x); (t,x)\in[0,1]\times\R^n\}$ be optimal for (\ref{eq:BB}). Then 
$$\mu^*_t=\left[(1-t)I+t\nabla\varphi\right]\#\nu_0, $$ with $T=\nabla\varphi$ solving Monge's problem,
provides McCann's {\em displacement interpolation} between $\nu_0$ and $\nu_1$ \cite{McCann}.
In fact,  $\{\mu^*_t; 0\le t\le 1\}$ may be seen as a constant-speed geodesic joining $\nu_0$ and $\nu_1$ in $\mathcal W_2$ and, moreover, as realized by Otto in \cite{O},
$\mathcal W_2$ can be endowed a Riemannian-like  structure\footnote{More precisely, a {\em weak Riemannian structure}, see \cite[2.3.2]{AG},
a limitation being
 that the tangent space about singular distributions is not ``rich enough,'' leading towards all other nearby distributions.}  which is consistent with  $W_2$.

McCann discovered \cite{McCann} that certain functionals are {\em displacement convex}, namely convex along Wasserstein geodesics. This has led to a variety of applications. Following one of Otto's main discoveries \cite{JKO,O}, it turns out that a large class of PDE's may be viewed as {\em gradient flows}  on the Wasserstein space ${\cal W}_2$. This interpretation, because of the displacement convexity of the functionals, is well suited to establish uniqueness and to study energy dissipation and convergence to equilibrium. A rigorous setting in which to make sense of the Otto calculus has been developed by Ambrosio, Gigli and Savar\'e \cite{AGS} for a suitable class of functionals. Convexity along geodesics in ${\cal W}_2$ also leads to new proofs of various  geometric and functional inequalities \cite{McCann}, \cite[Chapter 9]{Vil}, \cite{CVR}. Finally, we mention that, when the space is not flat, qualitative properties of optimal transport can be quantified in terms of how bounds on the Ricci-Curbastro curvature affect the displacement convexity of certain specific functionals \cite[Part II]{Vil2}.

In passing and for completeness, we note that
the {\em tangent space} of $\mathcal P_2(\R^n)$ at a probability measure $\mu$, denoted by $T_{\mu}\mathcal P_2(\R^n)$ \cite{AGS} may be identified with the closure in  $L^2_{\mu}$ of the span of vector fields 
$\{\nabla\varphi:\varphi\in C^\infty_c\}$, where $C^\infty_c$ is the family of smooth functions with compact support. It is naturally equipped with the inner product of $L^2_{\mu}$. Several recent papers have contributed to the development of
second-order calculus in Wasserstein's space \cite{vonRen,GC,CP1,CP2} building on \cite{Lott,Gigli,AG}.

\subsection{Optimal mass transport
as a stochastic control problem}\label{SCOMT}
{Optimal control, deeply rooted in the classical Calculus of Variations, seeks to modify the natural ({\em free}) evolution of a system so as to minimize a suitable cost. This field received a boost in the days of the space race motivated by aeronautical and astronautical problems such as navigation and the soft moon landing problem \cite[p.\ 21]{FR}. Foundational contributions were provided in the late fifties and in the early sixties by Pontryagin and his school in the Soviet Union, and by Bellman, Kalman, and others in the United States.  When the system starts from random initial conditions and/or is subject to random disturbances, the problem becomes one of a stochastic nature where the cost is now the expectation of a suitable random functional.  A crucial aspect of the problem is the information which is available to design the control action. In this paper, we shall mainly discuss the case where the state variables constitute a fully observable (vector) Markov process with values in a Euclidean space or, in the discrete-time setting, in a finite alphabet set. In the continuous-time setting, standard references are Lee-Markus and  Fleming-Rishel \cite{LM,FR}. For Markov Decision Processes,  standard references are \cite{Put,Berts}. As we shall see, both the OMT problem and  its regularized version (Schr\"odinger Bridge) can be viewed as stochastic optimal control problems: It is possible to derive suitable Hamilton-Jacobi-Bellman (HJB) equations. The difficulty lies with selecting the appropiate solution of the HJB as the boundary term is missing. Our first step in introducing the optimal control viewpoint on the topics of this paper consists in re-deriving the Benamou-Brenier formulation of OMT using elementary control considerations.

Let us start by observing  that the square of the Euclidean distance can be expressed as the infimum of an action integral, namely,
\begin{equation}\label{calvar}
\frac{1}{2}\|x-y\|^2=\min_{x\in\mathcal X_{xy}}\int_{0}^{1}\frac{1}{2}\|\dot{x}\|^2dt
\end{equation}
where $\mathcal X_{xy}$ is the family of $C^1([0,1];\R^n)$ paths with $x(0)=x$ and $x(1)=y$. The minimum in  (\ref{calvar}) is evidently achieved by
$$x^*(t)= (1-t)x+ty,
$$
namely, the straight line joining $x$ and $y$. Since $x^*(t)$ is a Euclidean geodesic, any probabilistic average of the lengths of $C^1$ trajectories starting at $x$ at time $0$ and ending in $y$ at time $1$ gives necessarily a higher value. Thus, the probability measure on $C^1([0,1];\R^n)$ concentrated on the path $\{x^*(t);0\le t\le 1\}$ solves the following problem
\begin{equation}\label{stoch}
\inf_{P_{xy}\in\D^1(\delta_{x},\delta_{y})}\E_{P_{xy}}\left\{\int_{0}^{1}\frac{1}{2}\|\dot{x}\|^2dt\right\},
\end{equation}
where $\D^1(\delta_{x},\delta_{y})$ are the probability measures on $C^1([0,1];\R^n)$ whose initial and final one-time marginals are Dirac's deltas concentrated at $x$ and $y$, respectively. Since (\ref{stoch}) provides us with yet another representation for $\frac{1}{2}\|x-y\|^2$, in view of (\ref{Wasserdist}), we also get that
\begin{equation*}
\inf_{\pi\in\Pi(\nu_0,\nu_1)}\int \frac{1}{2}\|x-y\|^2d\pi(x,y)=
\inf_{\pi\in\Pi(\nu_0,\nu_1)}\int \inf_{P_{xy}\in\D^1(\delta_{x},\delta_{y})}\E_{P_{xy}}\left\{\int_{0}^{1}\frac{1}{2}\|\dot{x}\|^2dt\right\}d\pi.
\end{equation*}

Now observe that if $P_{xy}\in\D^1(\delta_{x},\delta_{y})$ and $\pi\in\Pi(\nu_0,\nu_1)$, then 
$$P=\int_{\R^n\times\R^n}P_{xy}d\pi(x,y)$$
is a probability measure in $\D^1(\nu_0,\nu_1)$, namely a measure on $C^1([0,1];\R^n)$ with initial and final marginals $\nu_0$ and $\nu_1$, respectively. On the other hand, the disintegration of any measure $P\in\D^1(\nu_0,\nu_1)$ with respect to the initial and final positions\footnote{Disintegration can be viewed here as the opposite process to the construction of a product measure.} yields $P_{xy}\in\D^1(\delta_{x},\delta_{y})$ and $\pi\in\Pi(\nu_0,\nu_1)$. Thus, we get that the original optimal transport problem is equivalent to
\begin{equation}\label{stoch2}
\inf_{P\in\D^1(\nu_0,\nu_1)}\E_P\left\{\int_{0}^{1}\frac{1}{2}\|\dot{x}\|^2dt\right\}.
\end{equation}

So far, we have followed \cite[pp.\ 2-3]{leo}.
Instead of the ``particle'' picture, we can also consider the hydrodynamic version of (\ref{calvar}), namely the optimal control problem
\begin{eqnarray}\label{calvarhydro1}
\frac{1}{2}\|x-y\|^2=\inf_{v\in\mathcal V_{y}}\int_{0}^{1}\frac{1}{2} \|v(t,x^v(t))\|^2 dt\\
\dot{x}^v(t)=v(t,x^v(t)),\quad x(0)=x,\nonumber
\end{eqnarray}
where the admissible feedback control laws $v(\cdot,\cdot)$ in $\mathcal V_y$ are continuous and such that $x^v(1)=y$.
Following the same steps as before, we get that the optimal transport problem is equivalent to the following stochastic control problem with atypical boundary constraints
\begin{subequations}\label{eq:stochcontr}
\begin{eqnarray}\label{stochcontr1}&&\inf_{v\in\mathcal V}\E\left\{\int_{0}^{1}\frac{1}{2}\|v(t,x^v(t))\|^2dt\right\}\\&& \dot{x}^v(t)=v(t,x^v(t)),\quad {\rm a.s.},\quad x(0)\sim\nu_0,\quad x(1)\sim\nu_1.\label{stochcontr2}
\end{eqnarray}
\end{subequations}

Finally suppose $d\nu_0(x)=\rho_0(x)dx$, $d\nu_1(y)=\rho_1(y)dy$ and $x^v(t)\sim\rho(t,x)dx$. Then, necessarily, $\rho$ satisfies (weakly) the continuity equation
\begin{equation}\label{continuity}
\frac{\partial \rho}{\partial t}+\nabla\cdot(v\rho)=0
\end{equation}
expressing the conservation of probability mass. Moreover,
$$\E\left\{\int_{0}^{1}\frac{1}{2}\|v(t,x^v(t))\|^2dt\right\}=\int_{\R^n}\int_{0}^{1}\frac{1}{2}\|v(t,x)\|^2\rho(t,x)dtdx.
$$
Hence (\ref{eq:stochcontr}) turns into the Benamou-Brenier problem (\ref{eq:BB}): 
\begin{subequations}\label{eq:bb}
\begin{eqnarray}\label{bb1}&&\inf_{(\rho,v)}\int_{\R^n}\int_{0}^{1}\frac{1}{2}\|v(t,x)\|^2\rho(t,x)dtdx,\\&&\frac{\partial \rho}{\partial t}+\nabla\cdot(v\rho)=0,\label{bb2}\\&& \rho(0,x)=\rho_0(x), \quad \rho(1,y)=\rho_1(y).\label{boundary}
\end{eqnarray}\end{subequations}

The variational analysis for (\ref{eq:stochcontr}) or, equivalently, for (\ref{eq:bb})  can be carried out in many different ways. For instance, let $\mathcal P_{\rho_0\rho_1}$ be the family of flows of probability densities $\rho=\{\rho(t,\cdot); 0\le t\le 1\}$ satisfying (\ref{boundary}) and let $\mathcal V$ be the family of continuous feedback control laws $v(\cdot,\cdot)$. Consider the unconstrained minimization of the Lagrangian over $\mathcal P_{\rho_0\rho_1}\times\mathcal V$,
\begin{equation}\label{lagrangian}
\mathcal L(\rho,v)=\int_{\R^n}\int_{0}^{1}\left[\frac{1}{2}\|v(t,x)\|^2\rho(t,x)+\lambda(t,x)\left(\frac{\partial \rho}{\partial t}+\nabla\cdot(v\rho)\right)\right]dtdx,
\end{equation}
where $\lambda$ is a $C^1$ Lagrange multiplier. Integrating by parts, assuming that limits for $\|x\|\rightarrow\infty$ are zero, we get
\begin{align}\label{lagrangian2}
&\int_{\R^n}\int_{0}^{1}\left[\frac{1}{2}\|v(t,x)\|^2+\left(-\frac{\partial \lambda}{\partial t}-\nabla\lambda\cdot v\right)\right]\rho(t,x)dtdx\\
&\hspace*{90pt}\nonumber+\int_{\R^n}\left[\lambda(1,x)\rho_1(x)-\lambda(0,x)\rho_0(x)\right]dx.
\end{align}
The last integral is constant over $\mathcal P_{\rho_0\rho_1}$ for a fixed $\lambda$ and can therefore be discarded. We are left to minimize
\begin{equation}\label{lagrangian3}\int_{\R^n}\int_{0}^{1}\left[\frac{1}{2}\|v(t,x)\|^2+\left(-\frac{\partial \lambda}{\partial t}-\nabla\lambda\cdot v\right)\right]\rho(t,x)dtdx
\end{equation}
over $\mathcal P_{\rho_0\rho_1}\times\mathcal V$.
We consider doing this in two stages, starting from minimization with respect to $v$ for a fixed flow of probability densities $\rho=\{\rho(t,\cdot); 0\le t\le 1\}$ in $\mathcal P_{\rho_0\rho_1}$. Pointwise minimization of the integrand at each time $t\in[0,1]$
gives that
\begin{equation}\label{optcond}
v^*_\rho(t,x)=\nabla\lambda(t,x)
\end{equation}
which is continuous. Substituting this expression for the optimal control into (\ref{lagrangian3}), we obtain 
\begin{equation}
J(\rho)=-\int_{\R^n}\int_{0}^{1}\left[\frac{\partial \lambda}{\partial t}+\frac{1}{2}\|\nabla\lambda\|^2\right]\rho(t,x)dtdx.
\end{equation}
In view of this, if $\lambda$ satisfies the Hamilton-Jacobi equation
\begin{equation}\label{HJ}
\frac{\partial \lambda}{\partial t}+\frac{1}{2}\|\nabla\lambda\|^2=0,
\end{equation}
then $J(\rho)$ is identically zero over $\mathcal P_{\rho_0\rho_1}$ and any $\rho\in\mathcal P_{\rho_0\rho_1}$ minimizes the Lagrangian (\ref{lagrangian}) together with the feedback control (\ref{optcond}).
We have therefore established the following \cite{BB}:
\begin{proposition}Let $\rho^*(t,x)$ with $t\in[0,1]$ and $x\in {\mathbb R}^n$, satisfy
\begin{equation}\label{optev}
\frac{\partial \rho^*}{\partial t}+\nabla\cdot(\rho^*\nabla\lambda)=0, \quad \rho^*(0,x)=\rho_0(x),
\end{equation}
where $\lambda$ is a solution of the Hamilton-Jacobi equation
\begin{equation}\label{HJclass}
\frac{\partial \lambda}{\partial t}+\frac{1}{2}\|\nabla\lambda\|^2=0
\end{equation}
for some boundary condition $\lambda(1,x)=\lambda_1(x)$.
If $\rho^*(1,x)=\rho_1(x)$, then the pair $\left(\rho^*,v^*\right)$ with $v^*(t,x)=\nabla\lambda(t,x)$ is a solution of {\em (\ref{eq:BB})}.
\end{proposition}

The stochastic nature of the Benamou-Brenier formulation \eqref{eq:bb} stems from the fact that initial and final densities are specified. Accordingly, the above requires solving a two-point boundary value problem and the resulting control dictates the local velocity field. In general, one cannot expect to have a classical solution of (\ref{HJclass}) and has to be content with a viscosity solution \cite{FS}. See \cite{TT} for a recent contribution in the case when only samples of $\rho_0$ and $\rho_1$ are known.}

\subsection{Optimal mass transport with a ``Prior"}\label{OMTprior}


The stochastic control formulation \eqref{eq:stochcontr} of OMT casts this as a problem to steer the dynamical system $\dot x = u$, with $u$ being the control input, between specified marginal distributions for the state. 
The generalization to non-trivial underlying dynamics of the form
$\dot x = f(t,x) + u$ leads in a similar manner to 
\begin{subequations}\label{eq:stochcontrgeneral}
\begin{eqnarray}&&\inf_{u\in\mathcal V}\E\left\{\int_{0}^{1}\frac{1}{2}\|u(t,x^u(t))\|^2dt\right\}\\&& \dot{x}^u(t)=f(t,x^u(t))+u(t,x^u(t)),\quad {\rm a.s.},\quad x(0)\sim\nu_0,\quad x(1)\sim\nu_1.
\end{eqnarray}
\end{subequations}

Once again, this is a non-standard minimum energy control-effort problem due to the constraint on the final state distribution.	
A fluid dynamic reformulation proceeds as follows. Suppose we are given two end-point marginal probability densities $\rho_0$ and $\rho_1$, and suppose we are also given a model
\begin{equation}\label{continuityeq}
\frac{\partial \rho}{\partial t}+\nabla\cdot(f\rho)=0
\end{equation}
for the flow of probability densities $\{\rho(t,x); 0\le t\le 1\}$,
for a  continuous vector field $f(\cdot,\cdot)$, which however is not consistent with the given end-point marginals. Then, (\ref{continuityeq}) represents a ``prior" evolution to serve as a reference when seeking an update in the vector field to minimize the quadratic cost
\begin{subequations}\label{eq:priorBB}
\begin{eqnarray}\label{priorBB1}
&&\inf_{(\rho,v)}\int_{\R^n}\int_{0}^{1}\frac{1}{2}\|v(t,x)-f(t,x)\|^2\rho(t,x)dtdx,
\\&&\frac{\partial \rho}{\partial t}+\nabla\cdot(v\rho)=0,\label{priorBB2}
\\&&\rho(0,x)=\rho_0(x), \quad \rho(1,y)=\rho_1(y).\label{priorboundary}
\end{eqnarray}
\end{subequations}
Clearly, if the prior flow satisfies $\rho(0,x)=\rho_0(x)$ and $\rho(1,y)=\rho_1(y)$, then it solves the problem and $v^*=f$. Moreover, the standard OMT problem is recovered when the prior evolution is constant, i.e. $f\equiv 0$.


\begin{remark}
{\em From a different angle, problem \eqref{eq:priorBB} can be motivated as follows. It seeks a correction $v$ to a transportation plan $f$ that has already been computed from obsolete data (e.g., marginal distributions for transportantion of resources in $\rho_{0,\rm old}$ to meet demands in $\rho_{1,\rm old}$) as this data is being updated to a new set of marginals $\rho_0$ and $\rho_1$, respectively.}\hfill$\Box$
\end{remark}

{The particle version of \eqref{eq:priorBB} takes the form of a more familiar OMT problem, namely
\begin{equation}\label{GenOptTrans}
\inf_{\pi\in\Pi(\nu_0,\nu_1)}\int_{\R^n\times\R^n}c(x,y)d\pi(x,y),
\end{equation}
where  $d\nu_0(x)=\rho_0(x)dx$, $d\nu_1(y)=\rho_1(y)dy$ and
\begin{equation}\label{cost}
c(x,y)=\inf_{x\in\mathcal X_{xy}}\int_{0}^{1}L(t,x(t),\dot{x}(t))dt,\quad L(t,x,\dot{x})=\|\dot{x}-f(t,x)\|^2.
\end{equation}
The explicit calculation of the function $c(x,y)$ when $f\not\equiv 0$ is nontrivial. Moreover, the zero-noise limit results of \cite[Section 3]{leo2}, based on a Large Deviations Principle \cite{DZ}, although very general in other ways, seem to cover here only the case where $c(x,y)=c(x-y)$ is strictly convex originating from a Lagrangian $L(t,x,\dot{x})=c(\dot{x})$. 
We mention that  \cite{CGP5} deals with OMT problems where the Lagrangian is {\em not strictly convex} with respect to $\dot{x}$. Finally, we feel that the present formulation is a most natural one in which to study zero-noise limits of Schr\"odinger bridges with a general Markovian prior evolution. References \cite{CGP4,CGP5} discuss this same problem in the  case of a Gaussian prior, and show directly the convergence of the solution to the Hamilton-Jacobi-Bellman equation to solution of a Hamilton-Jacobi equation.}

The variational analysis for \eqref{eq:priorBB} can be carried out as in Section \ref{SCOMT} obtaining the following result:

\begin{proposition}If $\lambda$ satisfies the Hamilton-Jacobi equation
\begin{equation}\label{eq:hamiltonjacobi}
\frac{\partial \lambda}{\partial t}+f\cdot\nabla\lambda+\frac{1}{2}\|\nabla\lambda\|^2=0,
\end{equation}
and is such that the solution
$\rho^*$ to
\begin{equation}\label{prioroptev}
\frac{\partial \rho^*}{\partial t}+\nabla\cdot[(f+\nabla\lambda)\rho^*]=0, \quad \rho^*(0,x)=\rho_0(x),
\end{equation}
satisfies the end-point condition $\rho^*(1,x)=\rho_1(x)$ as well, then the pair 
\[
\left(\rho^*(t,x), v^*(t,x)=f(t,x)+\nabla\lambda(t,x)\right)
\]
solves \eqref{eq:priorBB}, provided $\lambda(t,x)\rho^*(t,x)$ vanishes as $\|x\|\rightarrow \infty$ for each fixed $t$.
\end{proposition}

\section{Schr\"{o}dinger's bridges}
Two excellent surveys on this topic are \cite{Wak,leo}. See also \cite{CGP15} for an exposition of the Schr\"{o}dinger bridge problems in the context of control engineering.
\subsection{The hot gas Gedankenexperiment}
In 1931/32, Erwin Schr\"{o}dinger considered the following Gedankenexperiment \cite{S1,S2}: We have a cloud of $N$ independent Brownian particles evolving in time\footnote{To put Schr\"odinger's 1931 work in perspective, one has to recall  what science had accomplished on the atomic hypothesis and physical Brownian motion by that time, besides Boltzmann's work \cite{BOL}. For this, see \cite{N1} and Section \ref{AHMSO} below.}.
This cloud of particles has been observed having
at the initial time $t=0$ an empirical distribution approximatively  equal to $\rho_0(x)dx$. At time $t=1$, an empirical distribution approximatively equal to  $\rho_1(x)dx$ is observed which considerably differs 
from what it should be according to the law of large numbers ($N$ is large, say of the order of Avogadro's number), namely
$$\rho_1(y)\neq \int_{\R^n}p(0,x,1,y)\rho_0(x)dx,
$$
where
\begin{equation}\label{transitiondensity}
p(s,x,t,y)=\left[2\pi(t-s)\right]
^{-\frac{n}{2}}\exp\left[-\frac{\|x-y\|^2} {2(t-s)}\right],\quad s<t
\end{equation}
is the transition density of the Wiener process (heat kernel). It is apparent that the particles have been transported in an unlikely way. But of the many unlikely ways in which this could have happened, which one is
the most likely? In modern probabilistic terms, this is a problem of {\em large deviations of the empirical distribution} as observed by F\"ollmer \cite{F2}.
Thus, at the outset, Schr\"odinger's motivation and the context of his question had no connection to OMT that we just discussed--the confluence of the two that we highlight shortly may be seen as a deep and lucky coincidence.

\subsection{Large deviations and maximum entropy}
\label{margedevmaxent}
The area of large deviations is concerned with the probabilities of rare events.
In light of
Sanov's theorem \cite{SANOV}, Schr\"odinger's problem can be 
seen as
a large deviations (maximum entropy) problem for distributions on trajectories, as we proceed to explain.

Let $\Omega=C([0,1];\R^n)$ be the space of $\R^n$-valued continuous functions, $\cD$ be the space of probability measures on $\Omega$, and  $W_x\in\cD$ denote Wiener measure starting at $x$ at $t=0$. If, instead assuming a Dirac marginal concentrated at $x$, we
enlarge $W_x$ to subsume
the volume measure at $t=0$,
we obtain
$$W:=\int W_x\,dx.$$
This is an unbounded nonnegative measure on the path space $\Omega$, called {\em stationary Wiener measure} (or, sometimes, {\em reversible Brownian motion}); $W$ has marginals at each point in time that coincide with the Lebesgue measure and therefore, while still nonnegative, it is not a probability measure. In fact it serves as a convenient analogue of the Lebesgue measure on paths that ``symmetrizes'' and ``uniformizes'' the Wiener measure with respect to the time arrow.

Alternatively,  we can enlarge $W_x$ to
$$W_\rho:=\int W_x\,\rho(x) dx$$
so that instead the Dirac marginal at $t=0$ it now has  a  marginal probability measure $\rho dx$. Clearly $W_\rho$ is a probability measure on $\Omega$ which is absolutely continuous with respect to $W$.
Indeed, if
\[
\D(P\| Q) = \begin{cases} \E_{P}\left\{ \log \frac{\De P}{\De Q} \right\}, \quad & \mbox{if $ P \ll Q$}\\+ \infty \quad & \mbox{otherwise} \end{cases}.
\]
denotes the relative entropy functional (divergence, Kullback-Leibler index) between nonnegative measures, it can be seen that
\[
\D(W_\rho\|W)=\int \rho \log(\rho) dx\, =-S(\rho),
\]
where $S(\rho)$ is the differential entropy of the measure $\rho(x)dx$.
Moreover,
\[
\D(W_\rho\| W_{\rho_0})= \int \rho \log(\frac{\rho}{\rho_0})dx\; =:\,\D(\rho\|\rho_0).
\]

Now let us consider $X^1, X^2,\ldots$ be i.i.d.\ Wiener evolutions on $[0,1]$ with values in $\R^n$ and starting value $x^i$ distributed according to $\rho$. The {\em empirical distribution} $\mu_N$ associated to $X^1,X^2,\ldots X^N$ is defined by
\begin{equation}
\label{emp}\mu_N(X^1,X^2,\ldots, X^N):=\frac{1}{N}\sum_{i=1}^N\delta_{X^i}.
\end{equation}
The expression in (\ref{emp}) defines a map from $\Omega^N$ to the space ${\cal D}$ of probability distributions on $C([0,1];\R^n)$. Hence, for  $E\subset {\cal D}$, we may consider the probability of
\[
\{(\omega^1,\ldots,\omega^N)|\mu_N(\cdot)\in E\},
\]
in the product measure $W^N_\rho$ on $\Omega^N$. By the ergodic theorem, see e.g. \cite[Theorem A.9.3.]{ellis}, as $N$ tends to infinity, the sequence of distributions $\mu_N$ converges weakly
to $W_\rho$. Hence, if $W_\rho\not\in E$, it follows that 
\[
W_\rho^N(\{(\omega^1,\ldots,\omega^N)| \mu_N(\cdot)\in E\})\searrow 0.
\]
In this, large deviation theory provides us with a much finer result, that the decay is {\em exponential} and the exponent may be characterized solving a {\em maximum entropy problem} \cite{DZ}.

Specifically, in our setting we let $E={\cal D}(\rho_0,\rho_1)$, namely the set of distributions on $C([0,1];\R^n)$ having marginal densities $\rho_0$ and $\rho_1$ at times $t=0$ and $t=1$, respectively. Then, Sanov's theorem \cite{SANOV}, \cite[Theorem 6.2.10]{DZ} asserts that, assuming the ``prior" $W_{\rho}$ does not have the required marginals, the probability of observing an empirical distribution $\mu_N$ in a neighborhood of ${\cal D}(\rho_0,\rho_1)$ in the weak topology, decays as
\[
e^{-N \inf\left\{\D(P\|W_\rho) \mid  P\in{\cal D}(\rho_0,\rho_1)\right\}}
\]
as $N\to\infty$.
This large-deviations statement can be turned around in the spirit of Gibbs conditioning principle (cf.\ \cite[Section 7.3]{DZ}), to deduce that, given as data the marginal distributions $\rho_0,\rho_1$, the most likely distribution is the closest to $W_\rho$ in the sense of relative entropy. But, in our setting, $\rho$ is unspecified, and in light of the fact that
\[
\D(P\| W_\rho)=\D(P\| W)-\D(\rho_0\| \rho),
\]
the most likely random evolution between two given marginals is in fact the solution of the Schr\"odinger Bridge Problem:

\begin{problem}\label{bridge}
\begin{equation}P_{\rm SBP}:={\rm argmin}\{ \D(P\|W) \mid P\in{\cal D}(\rho_0,\rho_1)\}.
\end{equation} 
\end{problem}

The optimal solution is referred to as the \emph{Schr\"odinger bridge} between $\rho_0$ and $\rho_1$ over $W$, since its marginal flow $\{\rho(t,\cdot); 0\le t\le 1\}$, which is the \emph{entropic interpolation} between $\rho_0$ and $\rho_1$, is seen as a ``bridge'' between the two marginals.

Next we discuss the structure of the solution and the reduction of the above to a static problem.

Let $P\in\cal D$ be a finite-energy diffusion \cite{Foe}. That is, under $P$, the canonical coordinate process $X_t(\omega)=\omega(t)$ has a (forward) It\^o differential
\begin{equation}\label{fordiff}dX_t=\beta_tdt+dW_t
\end{equation}
 where $\beta_t$ is adapted to $\{{\cal F}^-_t\}$ (${\cal F}^-_t$ is the $\sigma$-algebra of events up to time $t$) and
 \begin{equation}\label{finenergy}\E_P\left[\int_0^1\|\beta_t\|^2dt\right]<\infty.
 \end{equation}
Conditioning the process on starting at $X_0=x$ and ending at $X_1=y$ gives
 \[P_{xy}=P\left[\,\cdot\mid X_0=x,X_1=y\right],\quad W_{xy}=W\left[\,\cdot\mid X_0=x,X_1=y\right].
\]
These laws are referred to as the disintegrations of $P$ and $W$ with respect to the initial and final positions \cite{ChangPollard}.
Let also $\rho_{01}^P$ and  $\rho_{01}^{W}$
be the joint initial-final time distributions under $P$ and $W$, respectively. Then, we have the following decomposition of relative entropy \cite{F2}
\begin{eqnarray}
\D(P\|W)&=&\E_P\left[\log\frac{dP}{dW}\right]=\int\int\left[\log\frac{\rho^P_{01}(x,y)}{\rho_{01}^{W}(x,y)}\right]\rho^P_{01}(x,y)dxdy\nonumber\\
&&\hspace*{.5cm}
+\int\int\int\left(\log\frac{dP^{y}_{x}}{dW^{y}_{x}}\right)dP^y_x \rho^P_{01}(x,y)dxdy.\label{decomposition}
\end{eqnarray}
Clearly, since $\rho^P_{01}(x,y)$ and $dP_x^y$ can be independently chosen, the choice $P^{y}_{x}=W^{y}_{x}$ which actually makes the second summand equal to zero is optimal.  Thus, Problem \ref{bridge} reduces to the following
``static'' one:

\begin{problem}\label{static}
Minimize
\begin{equation}\label{staticindex}
\D(\rho_{01}\|\rho_{01}^{W})=\int\int\left[\log\frac{\rho_{01}(x,y)}{\rho_{01}^{W}(x,y)}\right]\rho_{01}(x,y)dxdy
\end{equation}
over the set of densities
\[
\Pi(\rho_0,\rho_1):=\left\{\rho_{01} \mbox{ on }\R^n\times\R^n\mid 
\int \rho_{01}(x,y)dy=\rho_0(x),\quad \int \rho_{01}(x,y)dx=\rho_1(y)\right\}.
\]
\end{problem}

One should note that the conditions defining $\Pi$ are linear; the elements of  $\Pi(\rho_0,\rho_1)$ are referred to as {\em couplings} between $\rho_0$ and $\rho_1$.
If $\rho^*_{01}$ is the solution to Problem \ref{static}, i.e., the optimal coupling, then, evidently,
\begin{equation}\label{representation}P^*(\cdot)=\int_{\R^n\times\R^n} W_{xy}(\cdot)\rho^*_{01}(dxdy),
\end{equation}
solves\footnote{Note that decomposition (\ref{decomposition}) and the resulting argument remain valid even when the prior measure is induced by  any, possibly non-Markovian, finite-energy diffusion $\bar{P}$, see (\ref{representation2}) below.} Problem \ref{bridge}.
The structure of the problem further endows $P^*$ with the same three-points transition density \[p^*(s,x,t,y,u,z)=\frac{p(s,x,t,y)p(t,y,u,z)}{p(s,x,u,z)}, s<t<u,
\]
as the prior -- a property that is referred to by saying that it belongs to the same {\em reciprocal class} as the prior measure \cite{Ber, Jam1}.

It is natural to consider the case when the prior is a Wiener measure with  variance $\epsilon$, denoted by
$W_\epsilon$ and having transition kernel
\[p_\epsilon(0,x,1,y)=\left[2\pi\epsilon\right]
^{-\frac{n}{2}}\exp\left[-\frac{\|x-y\|^2} {2\epsilon}\right],
\]
and to contemplate the limiting process when $\epsilon\to 0$.
Indeed, using $\rho_{01}^{W_\epsilon}(x,y)=\rho^{W_\epsilon}_0(x)p_\epsilon(0,x;1,y)$ and the fact that
\[\int\int\left[\log\rho_{0}^{W_\epsilon}(x)\right]\rho_{01}(x,y)dxdy=\int\left[\log\rho_{0}^{W_\epsilon}(x)\right]\rho_{0}(x)dx
\]
is independent of $\rho_{01}$, we obtain
\begin{align}\D(\rho_{01}\|\rho_{01}^{{W_\epsilon}})&=-\int\int\left[\log\rho_{01}^{W_\epsilon}(x,y)\right]\rho_{01}(x,y)dxdy+\int\int\left[\log\rho_{01}(x,y)\right]\rho_{01}(x,y)dxdy\\\nonumber
&=\int\int\frac{\|x-y\|^2}{2\epsilon}\rho_{01}(x,y)dxdy-S(\rho_{01})+\text{constant},
\end{align}
where $S$ is the differential entropy. Thus, Problem \ref{static} of minimizing $\D(\rho_{01}\|\rho_{01}^{{W_\epsilon}})$ over the couplings $\Pi(\rho_0,\rho_1)$
is equivalent to
\begin{equation}\label{eq:regularizedOT}
 \min_{\rho_{01} \in \Pi(\rho_0,\rho_1)}  \int  \frac{\|x-y\|^2}{2} \rho_{01}(x,y) \De x\De y +  \epsilon \int \rho_{01}(x,y) \log \rho_{01}(x,y) \De x\De y.
\end{equation}
It is seen that in the limit, as $\epsilon\to 0$, the cost reduces to $\frac12\|x-y\|^2$ which is in the form of the Kantorovich functional in \ref{problemkantorovich}. Thus, {\em the Schr\"odinger Bridge Problem represents a regularization of Optimal Mass Transport (OMT) obtained by subtracting from the cost functional a term proportional to the entropy}.

We have already seen that the Schr\"odinger's bridge problem can be motivated in three different ways. Firstly, via the original statistical mechanical thought-experiment of E.\ Schr\"odinger (a large-deviations problem).  
Secondly, via Sanov's theorem and Gibbs conditioning principle, as a maximum entropy problem.
It is an early and important instance of an inference method that prescribes how to choose a posterior distribution making the fewest number of assumptions beyond the available information. This approach has been  noticeably developed over the years by  Jaynes, Burg, Dempster, and Csisz\'{a}r \cite{Jaynes57,Jaynes82,BURG1,BLW,Dem,csiszar0,csiszar1,csiszar2}. Both of these two forms of the problem have their roots in Boltzmann's work \cite{BOL}.  The more recent third motivation comes from regularized OMT \cite{Mik,mt,MT,leo2,leo,CDPS} which mitigates its computational challenges \cite{AHT,BB,PC}. 

\subsection{Derivation of the Schr\"{o}dinger system}\label{derivation}
The Lagrangian function of Problem \ref{static} has the form
\begin{eqnarray}\nonumber
&&{\cal L}(\rho_{01},\lambda, \mu)= \int\int\left[\log\frac{\rho_{01}(x,y)}{\rho_{01}^{W}(x,y)}\right]\rho_{01}(x,y)dxdy\\&&+\int \lambda(x)\left[\int \rho_{01}(x,y)dy-\rho_0(x)\right]+\int \mu(y)\left[\int \rho_{01}(x,y)-\rho_1(y)\right].\nonumber
\end{eqnarray}
Setting the first variation equal to zero, we get the (sufficient) optimality condition
$$1+\log \rho^*_{01}(x,y)-\log p(0,x, 1,y)-\log \rho^{W}_0(x)+\lambda(x)+\mu(y)=0,
$$
where we have used the expression $\rho_{01}^{W}(x,y)=\rho^{W}_0(x)p(0,x,1,y)$ with $p$ as in (\ref{transitiondensity}). We get
\begin{eqnarray}\nonumber
\frac{\rho^*_{01}(x,y)}{p(0,x, 1,y)}&=&\exp\left[\log \rho^{W}_0(x)-1-\lambda(x)-\mu(y)\right]\\&=&\exp\left[\log \rho^{W}_0(x)-1-\lambda(x)\right]\exp\left[-\mu(y)\right].\nonumber
\end{eqnarray}
Thus,  the ratio $\rho^*_{01}(x,y)/p(0,x,1,y)$ factors into a function of $x$ times a function of $y$. Let us denote them by $\hat{\varphi}(x)$ and $\varphi(y)$, respectively. The optimal $\rho^*_{01}(\cdot,\cdot)$ has then the form  the form 
\begin{equation}\label{optimaljoint} 
\rho^*_{01}(x,y)=\hat{\varphi}(x) p(0,x,1,y)\varphi(y),
\end{equation} 
with $\varphi$ and $\hat{\varphi}$  satisfying
\begin{eqnarray}\hat{\varphi}(x)\int p(0,x,1,y)\varphi(y)dy&=&\rho_0(x),\label{opt1}\\
\varphi(y)\int p(0,x,1,y)\hat{\varphi}(x)dx&=&\rho_1(y).\label{opt2}
\end{eqnarray}
Let $\hat{\varphi}(0,x)=\hat{\varphi}(x)$, $\varphi(1,y)=\varphi(y)$ and  
$$\hat{\varphi}(1,y):=\int p(0,x,1,y)\hat{\varphi}(0,x)dx,\quad \varphi (0,x):=\int p(0,x,1,y)\varphi(1,y).
$$
Then, (\ref{opt1})-(\ref{opt2}) is equivalent to  the system
\begin{subequations}
\begin{eqnarray}\label{Schonestep1}
\hat{\varphi}(1,y)=\int p(0,x,1,y)\hat{\varphi}(0,x)dx,
\\\label{Schonestep2}\quad \varphi (0,x)=\int p(0,x,1,y)\varphi(1,y)dy
\end{eqnarray}
with the boundary conditions
\begin{equation}\label{BConestep}
\varphi(0,x)\cdot\hat{\varphi}(0,x)=\rho_0(x),\quad \varphi(1,y)\cdot\hat{\varphi}(1,y)=\rho_1(y).
\end{equation}
\end{subequations}
The question of existence and uniqueness of  positive functions $\hat{\varphi}$, $\varphi$ satisfying (\ref{Schonestep1})-(\ref{Schonestep2})-(\ref{BConestep}), left open by Schr\"{o}dinger, is a highly nontrivial one and has been settled in various degrees of generality by Fortet, Beurling, Jamison and F\"{o}llmer \cite{For,Beu,Jam2,F2},
see also\cite{leo,CLMW}; note that both Fortet and Beurling predate Sinkhorn's work \cite{Sin64,Sin67} on a problem in statistics that turned out to be closely related.

There are basically two ways to deal with the existence of solutions to the Schr\"odinger system of equations (\ref{Schonestep1}--\ref{BConestep}). One is to prove existence for the dual problem of the original convex optimization problem. This was first accomplished by Beurling \cite{Beu} leading to \cite{Jam2}, see also the recent paper \cite{CLMW}.  Alternatively, one can try to prove convergence of a suitable successive approximation scheme. This was first accomplished by Fortet \cite{For}, see also \cite{CGP9}. We outline Fortet's results in Section \ref{contIPF}. We remark that, in the special case where the marginals $\rho_0, \rho_1$ are Gaussian, the Schr\"odinger system has a closed-form solution. This was only recently discovered in \cite{CGP1}. The discrete counterpart of Fortet's algorithm is the so-called IPF-Sinkhorn algorithm which is discussed in Section \ref{ITERATIVEALGORITHM}. Also notice that in the recent paper \cite{EP} (see also \cite{leo3}), the bulk of Fortet's paper has been rewritten filling in all the missing steps and explaining the rationale behind his complex approximation scheme. 

The pair $(\varphi,\hat{\varphi})$ satisfying  (\ref{Schonestep1}--\ref{BConestep}) is unique up to multiplication of $\varphi$ by a positive constant $c$ and division of $\hat{\varphi}$ by the same constant.
At each time $t$, the marginal $\rho(t,\cdot)$ factors as
\begin{equation}\label{factorization}\rho(t,x)=\varphi(t,x)\cdot\hat{\varphi}(t,x).
\end{equation}
 The factorization (\ref{factorization}) resembles Born's relation (in quantum theory)
\[\rho(t,x)=\psi(t,x)\cdot\bar{\psi}(t,x)
\]
with $\psi$ and $\bar{\psi}$ satisfying two adjoint equations like $\varphi$ and $\hat{\varphi}$. Moreover, the solution of Problem \ref{bridge} exhibits the following remarkable {\em reversibility property}: Swapping the two marginal densities $\rho_0$ and $\rho_1$, the new solution is simply the time reversal of the previous one, cf.\ the title ``On the reversal of natural laws" of \cite{S1}. These are the remarkable analogies to quantum mechanics which appeared to Schr\"odinger very worth of reflection. But, wait a minute: 
When Schr\"{o}dinger poses his question, the very foundations of probability theory are still missing, the notion of stochastic process has not been introduced yet. Although Wiener had given a rigorous construction of Wiener measure in 1923 \cite{Wiener}, there was hardly any theory of continuous parameter stochastic processes in early 1930s. Many relevant results such as Sanov's theorem  and multiplicative functionals transformations of Markov processes \cite{IW} were not available. How could Schr\"odinger formulate and, to a large extent, solve such an abstract problem in the early 1930's? Schr\"{o}dinger, in his countryman Boltzmann's style, discretizes space to be able to compute by the De Moivre-Stirling formula the most likely joint initial-final distribution very much like Boltzmann had done in 1877 \cite{BOL}. 

In \cite{PTT}, the Schr\"odinger bridge problem was considered in the case when only samples of the boundary marginals are available. A numerical method has been developed which has potential to work in high-dimensional settings employing constrained maximum likelihood in place of the nonlinear boundary coupling and importance sampling to propagate $\varphi$ and $\hat{\varphi}$. Another paper dealing with a similar problem is \cite{BHDJ}.

\subsection{Stochastic control formulation}
In 1975 Jamison \cite{Jam2} showed that the solution of the Schr\"{o}dinger bridge problem is an h-path process in the sense of Doob \cite{Doob}, \cite[p. 566]{Doob2}. Indeed, dividing both sides of (\ref{optimaljoint}) by $\rho_0(x)$ (assumed everywhere positive), we get
\begin{equation}\label{multfunct}
p^*(0,x,1,y)=\frac{1}{\varphi(x)}p(0,x,1,y)\varphi(y),
\end{equation}
where $\varphi$, in Doob's language, is {\em space time harmonic} satisfying
\begin{equation}\label{space-time}
\frac{\partial\varphi}{\partial t}+\frac{1}{2}\Delta\varphi=0.
\end{equation}
The solution is namely obtained from the prior distribution via a multiplicative functional transformation of Markov processes \cite{IW}. It is worthwhile to mention that Francesco Guerra has connected such  function  to the {\em importance function} of neutron transport theory \cite{Guerra87}.

The connection between Schr\"odinger bridges and stochastic control, clearly established in \cite{DP}, was prepared by work in the latter field on the so-called logarithmic transformation of parabolic differential equations \cite{Fle1,Fle2,Fle3,Hol, Kar,Mit,Bl} as well as by work in mathematical physics \cite{Guerra87,Zam}.
The Schr\"{o}dinger bridge problem can be turned, thanks to {\em Girsanov's theorem},  into a stochastic calculus of variations problem \cite{Nag,Wak0,DGW,AN} which, in turn, can be reformulated in the language of stochastic control \cite{DP,DPP,PW,FHS}. Let $P\in\cal D$ be a finite-energy diffusion with forward differential (\ref{fordiff}). Then, by Girsanov's theorem \cite{KS},
\begin{equation}\label{girsanov}
\begin{split}
\log\frac{dP}{dW}&=\log\frac{\rho^P_0(X_0)}{\rho^{W}_0(X_0)}+\int_0^{1}\beta_t dX_t-\int_0^1\frac{1}{2}\|\beta_t\|^2dt,\quad P \; {\rm a.s.}, \\
&=\log\frac{\rho^P_0(X_0)}{\rho^{W}_0(X_0)}+\int_{0}^{1}\beta_t dW_t+\int_{0}^{1}\frac{1}{2}\|\beta_t\|^2dt,\quad P \; {\rm a.s.}.
\end{split}
\end{equation}
By the finite energy condition (\ref{finenergy}), 
\[Y_t:=\int_{0}^{t}\beta_{\tau} dW_{\tau}
\]
is a martingale and has therefore constant expectation. Since $Y_0=0$, we have
\[\E_P\left[\int_{0}^{1}\beta_t dW_t\right]=0.
\]
Then (\ref{girsanov}) yields
\begin{equation}\D(P\|W)=\E_P\left[\log\frac{dP}{dW}\right]=\D(\rho_{0}\|\rho_{0}^{W})+\E_P\left[\int_{0}^{1}\frac{1}{2}\|\beta_t\|^2dt\right].\label{girsrepres}
\end{equation}
Note that $\D(\rho_{0}\|\rho_{0}^{W})$ is constant over ${\cal D}(\rho_0,\rho_1)$. We then get a stochastic control formulation. Problem \ref{bridge} (when the prior has variance $\epsilon$) is equivalent to
\begin{problem}\label{stochcontr}
\begin{equation}
\begin{split}
&{\rm Minimize}_{u\in{\cal U}}\; J(u)=\E\left[\int_0^1\frac{1}{2\epsilon}\|u_t\|^2dt\right],\\
{\rm subject\; to}\; &dX_t=u_tdt+\sqrt{\epsilon} dW_t, \quad X_0\sim\rho_0(x), \quad X_1\sim\rho_1(y),
\end{split}
\end{equation}
where the family ${\cal U}$ consists of adapted, finite-energy control functions.
\end{problem}

The optimal control is of the feedback type
\begin{equation}\label{optcontr}
u^*(t,x)=\epsilon\nabla\log\varphi(t,x),
\end{equation}

where $(\varphi,\hat{\varphi})$ solve the Schr\"odinger system
\begin{subequations}\label{eq:Schrsystem}
\begin{eqnarray}\label{eq:SchrsystemA}
\frac{\partial\varphi}{\partial t}+\frac{\epsilon}{2}\Delta\varphi&=&0,\\
\frac{\partial\hat{\varphi}}{\partial t}-\frac{\epsilon}{2}\Delta\hat{\varphi}&=&0 \label{eq:SchrsystemB},\\\varphi(0,x)\cdot\hat{\varphi}(0,x)&=&\rho_0(x),\label{eq:SchrsystemC}\\ \varphi(1,y)\cdot\hat{\varphi}(1,y)&=&\rho_1(y).\label{eq:SchrsystemD}
\end{eqnarray}
\end{subequations}

\subsection{Fluid-dynamic formulation}

Problem \ref{stochcontr} leads immediately to the following fluid dynamic problem
\begin{problem}\label{fluid}
\begin{subequations}\label{eq:BBS}
\begin{eqnarray}\label{BBS1}&&\inf_{(\rho,u)}\int_{\R^n}\int_{0}^{1}\frac{1}{2}\|u(t,x)\|^2\rho(t,x)dtdx,\\&&\frac{\partial \rho}{\partial t}+\nabla\cdot(u\rho)-\frac{\epsilon}{2}\Delta\rho=0,\label{BBS2}\\&& \rho(0,x)=\rho_0(x), \quad \rho(1,y)=\rho_1(y).\label{Boundary}
\end{eqnarray}\end{subequations}
where $u(\cdot,\cdot)$ varies over continuous functions on $[0,1]\times \R^n$.
\end{problem}

\begin{remark}\label{deterministicflow}{\em Contrary to what is often stated in the literature (see e.g. \cite{LegLi})\footnote{We contributed ourselves to this confusion in \cite[Section 5]{CGP4}.}, Problem \ref{fluid} is not equivalent to Problems \ref{bridge}, \ref{static} and \ref{stochcontr} in that it only reproduces the optimal entropic interpolating flow $\{\rho^*(t,\cdot); 0\le t\le 1\}$. Information about correlations at different times and smoothness of the trajectories in the support of the measure $P^*$ is here lost. Indeed, let $(\rho^*,u^*)$ be optimal for Problem \ref{fluid} and define the {\em current velocity field} \cite{N1}
\begin{equation}\label{current}
v^*(t,x):=u^*(t,x)-\frac{\epsilon}{2}\nabla\log\rho^*(t,x).
\end{equation}
Assume that $v^*$ guarantees existence and uniqueness of the following initial value problem on $[0,1]$ for any deterministic initial condition
\[
\dot{X}_t=v^*(t,X_t), \quad X_0\sim
 \rho_0dx.
\]
Then the probability density $\rho(t,x)$ of $X_t$ satisfies the continuity equation 
\[\frac{\partial \rho}{\partial t}+\nabla\cdot(v^*\rho)=0
\]
as well as (\ref{BBS2}) with the same initial condition and therefore coincides with $\rho^*(t,x)$. \hfill $\Box$ }
\end{remark}

It appears that, as $\epsilon\searrow 0$, the solution to this problem converges to the solution of the Benamou-Brenier Optimal Mass Transport problem \cite{BB}. This is indeed the case  \cite{Mik, mt, MT,leo,leo2}. The  analysis in Remark \ref{deterministicflow}, however, suggests that an alternative fluid-dynamic problem characterization of the entropic interpolation flow $\{\rho^*(t,\cdot); 0\le t\le 1\}$ may be possible.
Indeed, such alternative time-symmetric problem was derived in \cite{CGP4}, see also \cite{GLR}: 
\begin{problem}\label{fluidsymm}
\begin{subequations}\label{eq:BBSS}
\begin{eqnarray}\label{BBSS1}&&\inf_{(\rho,v)}\int_{\R^n}\int_{0}^{1}\left[\frac{1}{2}\|v(t,x)\|^2+ \frac{\epsilon^2}{8}\|\nabla\log\rho\|^2\right]\rho(t,x)dtdx,\\&&\frac{\partial \rho}{\partial t}+\nabla\cdot(v\rho)=0,\label{BBSS2}\\&& \rho(0,x)=\rho_0(x), \quad \rho(1,y)=\rho_1(y).\label{BBSS3}
\end{eqnarray}\end{subequations}
\end{problem}
The two criteria differ by a (scaled) Fisher information functional
\[{\mathcal I}(\rho)=\int\|\nabla\log\rho(t,x)\|^2\rho(t,x)dx
\]
while the Fokker-Planck equation has been replaced by the continuity equation. Both Problems \ref{fluid} and \ref{fluidsymm} can be thought of as regularizations of the Benamou-Brenier problem \cite{BB} and as dynamic counterparts of (\ref{eq:regularizedOT}) \cite{Mik, mt, MT,leo,leo2,CGP3,CGP4}. 
\begin{remark}\label{carlenfisher}{\em The problem of minimizing the Yasue \cite{Y} action (\ref{BBSS1}) under the continuity equation (\ref{BBSS2}) and boundary conditions (\ref{BBSS3}) was formulated by Carlen in \cite[p.131]{Carlen}. He was investigating there possible connections between OMT and Nelson's
stochastic mechanics. He wrote: ``... the Euler-Lagrange equations for it are not easy to understand". The solution to Carlen's problem was provided in  \cite[Section 5]{CGP4} through the fluid-dynamic problem associated to the Schr\"odinger bridge. Carlen's statement has caused several authors (see e.g. \cite{Leg}) to believe that Yasue in \cite{Y} had already discussed this problem. This is not true as what Yasue developed there, using Nelson's integration by parts formula for semimartingales \cite[Theorem 11.12]{N1}, was a stochastic calculus of variations which may be viewed as the ``particle" form  of Hamilton principle in stochastic mechanics. The fluid dynamic counterpart of this principle was developed in \cite{GuMo}, see also \cite{N}, but with a different action (there is a minus in front of the Fisher information functional)!   For a clarification of the connection between Schr\"odinger bridges, Nelson's stochastic mechanics and Bohm's stochastic mechanics \cite{Bo,BV,BH} through calculus of variations in the Wasserstein space see \cite[Section VI]{CP2}.}\hfill$\Box$
\end{remark}

\subsection{General prior}
All of what we have seen in this section can be generalized to a general Markovian prior measure. Indeed, suppose that $\bar{P}$ is a Markov finite energy diffusion \cite{Foe}. Under $\bar{P}$, the canonical coordinate
process has a forward It\^o differential

\[
dX_t=f(t,X_t)dt+\sqrt{\epsilon}dW_t.
\]
Let $\bar{p}(s,x,t,y)$ be the corresponding transition density. Consider the Schr\"{o}dinger bridge problem with prior $\bar{P}$, namely
\begin{problem}\label{generalbridge}
\begin{equation}{\rm Minimize}\quad \D(P\|\bar{P}) \quad {\rm over} \quad P\in{\cal D}(\rho_0,\rho_1).
\end{equation} 
\end{problem}

Then, a decomposition like (\ref{decomposition}) turns the problem into Problem \ref{static}. In particular, (\ref{representation}) is replaced by
\begin{equation}\label{representation2}
P^*(\cdot)=\int_{\R^n\times\R^n} \bar{P}_{xy}(\cdot)\rho^*_{01}(dxdy).
\end{equation}
 The various dynamic formulations get modified accordingly. Problem \ref{stochcontr} becomes
\begin{problem}\label{genpriorstochcontr}
\begin{equation}
\begin{split}
&\min_{u\in{\cal U}}\; J(u)=\E\left[\int_0^1\frac{1}{2 \epsilon}\|u_t\|^2dt\right],\\
 &dX_t=\left[f(t,X_t)+u_t\right]dt+ \sqrt{\epsilon} dW_t, \quad X_0\sim\rho_0(x)dx, \quad X_1\sim\rho_1(y)dy,
\end{split}
\end{equation}
where the family ${\cal U}$ consists of adapted, finite-energy control functions.
\end{problem}

Then (\ref{optcontr}) is replaced by
\begin{equation}\label{generealoptcontr}
u^*_t=\epsilon\nabla\log\varphi(t,X_t),
\end{equation}
where $(\varphi,\hat{\varphi})$ now solve the Schr\"odinger system
\begin{subequations}\label{eq:genSchrsystem}
\begin{eqnarray}\label{eq:genSchrsystemA}
&&\frac{\partial\varphi}{\partial t}+f\cdot\nabla\varphi+\frac{{ \epsilon}}{2}\Delta\varphi=0,\\
&&\frac{\partial\hat{\varphi}}{\partial t}+\nabla\cdot(f\hat{\varphi})-\frac{{ \epsilon}}{2}\Delta\hat{\varphi}=0 \label{eq:genSchrsystemB},\end{eqnarray}
\end{subequations}
with the same boundary conditions. The fluid dynamic formulations also change accordingly. Let
\[\bar{v}(t,x):=f(t,x)-\frac{{\epsilon}}{2}\nabla\log\bar{\rho}(t,x)
\]
be the current velocity of the prior process. Then
Problem \ref{fluidsymm} becomes
\begin{subequations}\label{FDproblem}
\begin{align}\label{SBB1}
&\inf_{(\rho,v)}\int_{\R^n}\int_{0}^{1}\left[\frac{1}{2}\|v(t,x)-\bar{v}(t,x)\|^2+\frac{{\epsilon^2}}{8}\left\|\nabla\log\frac{\rho(t,x)}{\bar{\rho}(t,(x)}\right\|^2\right]\rho(t,x)dx dt,\\
&{\frac{\partial \rho}{\partial t}+\nabla\cdot(v\rho)=0},\label{SBB2}\\ &\rho_0=\nu_0, \quad \rho_1=\nu_1,\label{SBB3}
\end{align}
\end{subequations}
where the two criteria differ by a {\em relative Fisher information} term. 
%

The theory can be further extended to the case of a general diffusion coefficient (anisotropic diffusion)  and to the situation where the prior Markovian measure features creation and/or killing, so that the probability mass is not preserved at each time \cite{DGW,AN,Wak,CGP2}. Both cases are treated in Section \ref{anisotropic}.

\begin{remark}\label{Cdynvsstatic} {\em We like to stress here one aspect of (\ref{representation}), namely that this is just a {\em representation} of the solution $P^*$ even if we have been able to solve Problem \ref{static}. This representation needs to be supplemented with the multiplicative functional relation between transition densities
\begin{equation}\label{multfunct2}
p^*(s,x,t,y)=\frac{1}{\varphi(s,x)}\bar{p}(s,x,t,y)\varphi(t,y), \quad 0\le s<t\le 1,
\end{equation}
and the associated relation between drifts (\ref{generealoptcontr}). We shall come back to this point when discussing the discrete state space case.
}\hfill$\Box$
\end{remark}

\section{Stochastic control and 
general bridge problems}\label{anisotropic}

The classical probabilistic literature on the Schr\"odinger bridge problem only considers the case in which the noise and control matrices are constant, diagonal and nonsingular (often the identity matrix), see Problem \ref{genpriorstochcontr}. This permits, in particular, to employ the Girsanov transformation (\ref{girsanov}). In many applications, however, such a requirement represents a serious limitation. Noise may not affect all components of the state vector like in the controlled oscillator (\ref{OUC}) and/or the control channel may be dictated by technological constraints such as in some engineering applications. Central in such applications are problems for controlled Gauss-Markov models\footnote{In a twin paper \cite{CGP14}, we shall review the by now vast literature \cite{HS1,HS2,SkeIke, GS,ZGS,CGP1,CGP3,CGP5,CGP7,CGP8,CGP11,HW,Bak1, Bak2, Bak3, Bak4, Bak5,GT1,RT1,RT2,OGT,OT,ROT,OT2,ACGP,CCGP}  on optimal steering of probability distributions for Gauss-Markov models in continuous and discrete time, over a finite or infinite time horizon, with or without state and/or control constraints and applications.}: Find an adapted control $u$ minimizing
\[
J(u)=\E\left\{\int_0^1\|u(t)\|^2 \,dt\right\},
\]
among those which achieve the transfer
\begin{eqnarray}\nonumber
dx(t)=A(t)x(t)dt+B(t)u(t)dt+B_1(t)dW_t,\\ x(0)\sim\mathcal N(m_0,\Sigma_0),\quad x(1)\sim\mathcal N(m_1,\Sigma_1).\nonumber
\end{eqnarray}
It is therefore important to pose the minimum energy steering problem for  probability distributions in a more general setting where the original large deviations/maximum entropy motivation might be lost. This will make salient the advantage of the control formulation which always makes  perfect sense and in which the free (uncontrolled) evolution plays the role of the prior model. We proceed in this section to see how much of the fluid-dynamic formulation  can be extended to the general case \cite{CGP2}.

We consider a cloud of particles with density $\rho(t,x)$, $x\in\R^n$, which evolves according to the transport-diffusion equation
\begin{equation}\label{Fokker-Planck}
\frac{\partial \rho}{\partial t}+\nabla\cdot(f(t,x)\rho)+V(t,x)\rho=\frac{1}{2}\sum_{i,j=1}^n\frac{\partial^2(a_{ij}(t,x)\rho)}{\partial x_i\partial x_j}, 
\end{equation}
with  $\rho(0,\cdot)=\rho_0(\cdot)$  a probability density. Differently from previous literature on connections to Feynman-Kac \cite{Wak,PW, DPP} and following our desire to be able to model inertial particles as in the previous section, we assume that the matrix $a(t,x)=[a_{ij}(t,x)]_{i,j=1}^{m}$ is only positive semidefinite of constant rank on $[0,1]\times\R^n$ with
\[
a_{ij}(t,x)=\sum_k\sigma_{ik}(t,x)\sigma_{kj}(t,x)
\]
for a matrix $\sigma(t,x)=[\sigma_{ik}(t,x)]\in\R^{n\times m}$ of constant rank $m\leq n$. Notice that the presence of $V(t,x)\ge 0$ allows for the possibility of loss of mass, so that the integral of $\rho(t,x)$ over $\R^n$ is not necessarily constant. This flexibility permits to model particles satisfying
\begin{equation}\label{eq:nonlinearSDE}
dX_t=f(t,X_t)dt+\sigma(t,X_t)dW_t,
\end{equation}
which are absorbed at some rate by the medium in which they travel or, if the sign of $V$ is negative, created out of this same medium \cite[p. 272]{KT}. We assume that $f$ and $\sigma$ are smooth and that the operator
$$L=\sum_{i,j=1}^na_{ij}(t,x)\partial_{x_i}\partial_{x_j}+\sum_{j=1}^nf_j(t,x)\partial_{x_j}-\partial_t
$$
is  {\em hypoelliptic} satisfying H\"{o}rmander's condition \cite{H,OR}. Hypoelliptic diffusions model important processes in many branches of science: Ornstein-Uhlenbeck stochastic oscillators, Nyquist-Johnson circuits with noisy resistors, in image reconstruction based on Petitot's model of neurogeometry of vision \cite{BCGR}, etc. The ``reweighing" of the original measure of the Markov process (\ref{eq:nonlinearSDE}) when $V$ is unbounded is a delicate issue and can be accomplished via the Nagasawa transformation, see \cite[Section 8B]{Wak} for details.

Let us now suppose that (\ref{Fokker-Planck}) represents a {\em prior} evolution and that at time $t=1$ we measure an empirical probability density $\rho_1(\cdot)\neq \rho(1,\cdot)$ as dictated by \eqref{Fokker-Planck}.
Thus, the model (\ref{Fokker-Planck}) is thought not to be consistent with the estimated end-point empirical distribution. Yet, suppose that one has reasons to believe\footnote{Alternative reasoning based on Gibbs conditioning principle, as before, may be possible.} that the actual evolution was close to the nominal one and that only the actual drift field is different, perturbed by an additive term $\sigma(t,x) u(t,x)$, i.e., 
\[
 v(t,x)=f(t,x)+\sigma(t,x) u(t,x).
 \]
Notice that the control variables, which may be fewer than $n$, act through the same channels of the diffusive part. The assumption that stochastic excitation and control enter through the same ``channels'' is natural in certain applications as explained and treated in \cite{CGP1} for linear diffusions. The case were these channels may differ is considered in~\cite{CGP3}.

Taking (\ref{Fokker-Planck}) as a reference evolution and given the terminal probability density $\rho_1$, we are led to consider the following generalization of Problem \ref{fluid}:
\begin{subequations}\label{FDProblem}
\begin{eqnarray}
\label{FD1}
&&\inf_{(\rho,u)}\int_{\R^n}\int_{0}^{1}\left[\frac{1}{2}\|u\|^2+V(t,x)\right]\rho(t,x)dtdx,\\
&&\frac{\partial \rho}{\partial t}+\nabla\cdot((f+\sigma u)\rho)=\frac{1}{2}\sum_{i,j=1}^n\frac{\partial^2\left(a_{ij}\rho\right)}{\partial x_i\partial x_j},\label{FD2}\\
&&\rho(0,x)=\rho_0(x), \quad \rho(1,x)=\rho_1(x).\label{FD3}
\end{eqnarray}
\end{subequations}

When $[a_{ij}]$ does depend on $x$, the connection to a relative entropy problem on path space is apparently available only under rather restrictive assumptions such as uniform boundness of $a$ \cite[Section 5]{Fischer}. Problem (\ref{FDProblem}) thus appears as a generalization of Problem \ref{fluid} which is not necessarily connected to  a large deviations problem. In  \cite{CH}, a special case of Problem \ref{FDProblem} has been considered ($V\equiv 0$, $\sigma(t,x)=B(t)$). For certain classes of drifts $f$, a Wasserstein proximal algorithm has been used to numerically solve a pair of Initial Value Problems equivalent to \eqref{generalizedschrodinger} below. This result has been generalized to the situation where there are hard state constraints (the diffusion paths have to remain in a bounded domain) in \cite{CH2}. We mention here that chance state constraints were considered for the covariance control problem in \cite{OGT}. Moreover, input constraints for discrete and continuous time models have been considered in \cite{Bak2,OT2}. In \cite{ROT2,YCTC}, a nonlinear covariance control problem was studied  by iteratively solving an approximate linearized problem and by differential dynamic programming, respectively.

We provide below in some detail the variational analysis for (\ref{FDProblem}) which can be viewed as a generalization of that of  Section \ref{SCOMT}. Let $\cP_{\rho_0\rho_1}$ be the family of flows of probability densities
satisfying (\ref{FD3}). Let $\mathcal U$ be a family of continuous  feedback control laws $u(\cdot,\cdot)$. Consider the unconstrained minimization of the Lagrangian over $\cP_{\rho_0\rho_T}\times\mathcal U$
\begin{eqnarray}\nonumber
\mathcal L(\rho,u,\lambda)=\int_{\R^n}\int_{0}^{1}\left[\left(\frac{1}{2}\|u(t,x)\|^2+V(t,x)\right)\rho
+\lambda(t,x)\left(\frac{\partial \rho}{\partial t}+\nabla\cdot((f+\sigma u)\rho)\right.\right.\\\left.\left.-\frac{1}{2}\sum_{i,j=1}^n\frac{\partial^2}{\partial x_i\partial x_j}\left(a_{ij}(t,x)\rho\right)\right)\right]dtdx,\nonumber
\end{eqnarray}
where $\lambda$ is a $C^1$ Lagrange multiplier. After integration by parts, assuming that limits for $x\rightarrow\infty$ are zero, and observing that the boundary values are constant over $\cP_{\rho_0\rho_1}$, we get the problem
\begin{equation}\label{Lagrangian2}
\inf_{(\rho,u)\in\mathcal X_{\rho_0\rho_1}\times\mathcal U}\int_{\R^n}
\int_{0}^{1}\left[\frac{1}{2}\|u\|^2+V-\left(\frac{\partial \lambda}{\partial t}+(f+\sigma u)\cdot\nabla\lambda+\frac{1}{2}\sum_{i,j=1}^n a_{ij}\frac{\partial^2\lambda}{\partial x_i\partial x_j}\right)\right]\rho dtdx
\end{equation}
Pointwise minimization of the integrand with respect to $u$ for each fixed flow of probability densities $\rho$ gives
\begin{equation}\label{prioroptcond}
u^*_{\rho}(t,x)=\sigma'\nabla\lambda(t,x).
\end{equation}
Plugging this form of the optimal control into (\ref{Lagrangian2}), we get the functional of $\rho\in\cP_{\rho_0\rho_1}$
\begin{equation}\label{priorlagrangian3}
J(\rho,\lambda)=\int_{\R^n}\int_{0}^{1}\left[\frac{\partial \lambda}{\partial t}+f\cdot\nabla\lambda+\frac{1}{2}\nabla\lambda\cdot a\nabla\lambda-V+\frac{1}{2}\sum_{i,j=1}^n a_{ij}\frac{\partial^2\lambda}{\partial x_i\partial x_j}\right]\rho dtdx.
\end{equation}
We then have the following result: 
\begin{proposition}If $\rho^*$ satisfies
\begin{equation}\label{Optev}
\frac{\partial \rho}{\partial t}+\nabla\cdot((f+a\nabla\lambda)\rho)=\frac{1}{2}\sum_{i,j=1}^n\frac{\partial^2\left(a_{ij}\rho\right)}{\partial x_i\partial x_j},
\end{equation}
with $\lambda$ a solution of the HJB-like equation
\begin{equation}\label{HJB}
\frac{\partial \lambda}{\partial t}+f\cdot\nabla\lambda+\frac{1}{2}\sum_{i,j=1}^n a_{ij}(t,x)\frac{\partial^2\lambda}{\partial x_i\partial x_j}+\frac{1}{2}\nabla\lambda\cdot a\nabla\lambda-V=0,
\end{equation}
and $\rho^*(1,\cdot)=\rho_1(\cdot)$, then the pair $\left(\rho^*,u^*\right)$ with $u^*=\sigma'\nabla\lambda$ is a solution of  (\ref{FDProblem}).
\end{proposition}

Of course, the difficulty lies with the nonlinear equation (\ref{HJB}) for which no boundary value is available. Together, $\rho(t,x)$ and $\lambda(t,x)$ satisfy the coupled equations (\ref{Optev})-(\ref{HJB}) and the split boundary conditions for $\rho(t,x)$ in \eqref{FD3}.
Let us, however, define 
$$\varphi(t,x)=\exp[\lambda(t,x)], \quad (t,x)\in [0,1]\times\R^n. 
$$
If $\lambda$ satisfies (\ref{HJB}), we get that $\varphi$ satisfies the {\em linear} equation
\begin{equation}\label{harmonic}
\frac{\partial \varphi}{\partial t}+f\cdot\nabla\varphi+\frac{1}{2}\sum_{i,j=1}^n a_{ij}(t,x)\frac{\partial^2\varphi}{\partial x_i\partial x_j}=V\varphi.
\end{equation}
Moreover, for $\rho$ satisfying (\ref{Optev}) and $\varphi$ satisfying (\ref{harmonic}), let us define
$$\hat{\varphi}(t,x)=\frac{\rho(t,x)}{\varphi(t,x)},\quad (t,x)\in[0,1]\times\R^n. 
$$
Then a long but straightforward calculation  shows that $\hat{\varphi}$ satisfies the original  equation (\ref{Fokker-Planck}). Thus, we have the system of linear PDE's
\begin{subequations}\label{generalizedschrodinger}\begin{eqnarray}\label{shroedinger1}
\hspace*{-.3in}\frac{\partial \varphi}{\partial t}+f\cdot\nabla\varphi+\frac{1}{2}\sum_{i,j=1}^n a_{ij}\frac{\partial^2\varphi}{\partial x_i\partial x_j}&=&V\varphi,\\
\hspace*{-.3in}\frac{\partial \hat{\varphi}}{\partial t}+\nabla\cdot(f\hat{\varphi})-\frac{1}{2}\sum_{i,j=1}^n\frac{\partial^2\left(a_{ij}\hat{\varphi}\right)}{\partial x_i\partial x_j}&=&-V\hat{\varphi},
\label{schroedinger2}\end{eqnarray}
nonlinearly coupled through their boundary values as
\begin{equation}\label{BND}
\varphi(0,\cdot)\hat{\varphi}(0,\cdot)=\rho_0(\cdot),\quad \varphi(1,\cdot)\hat{\varphi}(1,\cdot)=\rho_1(\cdot).
\end{equation}
\end{subequations}
Equations (\ref{shroedinger1}-\ref{BND}) constitute a {\em generalized Schr\"{o}dinger system}. We have therefore established the following result.
\begin{theorem}Let $(\varphi(t,x),\hat{\varphi}(t,x))$ be nonnegative functions satisfying (\ref{shroedinger1})-(\ref{BND}) for $(t,x)\in\left([0,1]\times\R^n\right)$. Suppose $\varphi$ is everywhere positive. Then the pair $\left(\rho^*,u^*\right)$ with 
\begin{subequations}
\begin{eqnarray}\label{optlaw}
\hspace*{-.3in}u^*(t,x)&=&\sigma'\nabla\log\varphi(t,x),\\
\hspace*{-.3in}\frac{\partial \rho}{\partial t}+\nabla\cdot((f+a\nabla\log\varphi)\rho)&=&\frac{1}{2}\sum_{i,j=1}^n\frac{\partial^2\left(a_{ij}\rho\right)}{\partial x_i\partial x_j},
\label{optevolution}\end{eqnarray} 
\end{subequations}
is a solution of  (\ref{FDProblem}).
\end{theorem}

Establishing existence and uniqueness (up to multiplication/division of the two functions by a positive constant) of the solution of the Schr\"{o}dinger system is extremely challenging even when the diffusion coefficient matrix $a$ is constant and nonsingular. Particular care is required in the case when $V$ is unbounded or singular \cite{AN}. Nevertheless, if the fundamental solution $p$ of (\ref{Fokker-Planck}) is everywhere positive on $\left([0,1]\times\R^n\right)$, which is to be expected in the hypoelliptic case \cite{H,Kli}, existence and uniqueness follows from a deep result of Beurling \cite{Beu} suitably extended by Jamison \cite[Theorem 3.2]{Jam1}, \cite[Section 10]{Wak}.

\begin{remark} {\em It is interesting to note that although \eqref{FDProblem} is not convex in $(\rho,u)$, it can be turned into a convex problem in a new set of coordinates $(\rho, m)$ where $m=\rho  u$, in which case it becomes
\begin{subequations}\label{FDproblem_convex}
\begin{eqnarray}
\label{FD1_convex}
&&\hspace*{-.3in}\inf_{(\rho,m)}\int_{\R^n}\int_{0}^{1}\left[\frac{1}{2}\frac{\|m\|^2}{\rho(t,x)}+V(t,x)\rho(t,x)\right]dtdx,\\
&&\hspace*{-.3in}\frac{\partial \rho}{\partial t}+\nabla\cdot(f\rho+\sigma m)=\frac{1}{2}\sum_{i,j=1}^n\frac{\partial^2\left(a_{ij}\rho\right)}{\partial x_i\partial x_j},\label{FD2_convex}\\
&&\hspace*{-.3in}\rho(0,\cdot)=\rho_0(\cdot), \quad \rho(1,\cdot)=\rho_1(\cdot).\label{FD3_convex}
\end{eqnarray}
\end{subequations}
This type of coordinate transformation has been effectively used in \cite{BB} in the context of optimal mass transport.}\hfill$\Box$
\end{remark}

\section{The atomic hypothesis, microscopes and stochastic oscillators}\label{AHMSO}

Richard Feynman, in his first lecture of the famous Caltech series, stated: ``If, in some cataclysm, all of scientific knowledge were to be destroyed, and only one sentence passed on to the next generation of creatures, what statement would contain the most information in the fewest words? I believe it is the atomic hypothesis that all things are made of atoms - little particles that move around in perpetual motion, attracting each other when they are a little distance apart, but repelling upon being squeezed into one another."

 
 The atomic hypothesis apparently originates with Democritus of Abdera (a colony of Miletus in nowadays Greek Thrace) and his mentor Leucippus. Democritus, according to the Greek historian Diogenes La\"{e}rtius, (Democritus, Vol. IX, 44) states:
 \[
 \alpha\rho\chi\acute{\alpha}\varsigma\; 
 \epsilon\iota\nu\alpha\iota\; 
 \tau\omega\nu\;
 \acute o\lambda\omega\nu\;
 {\alpha}\tau\acute{o}\mu o\upsilon\varsigma\;
 \kappa\alpha\acute{\iota}\; 
 \kappa\epsilon\nu\acute{o}\nu,\;
 \tau{\alpha}\; 
 \delta'\acute{\alpha}\lambda\lambda\alpha\;
 \pi\acute{\alpha}\nu\tau\alpha\;
 \nu\epsilon\nu o\mu\acute{\iota}\sigma\theta\alpha\iota
 \]
which (Robert Drew Hicks (1925)) can be translated as:
``The first principles of the universe are atoms and empty space; everything else is merely thought to exist. "

Democritus had (correctly) imagined that the wearing down of a wheel and the drying of  clothes would be due to small particles of wood and water, respectively, flying out of them. Then, he had the following philosophical argument  (according to Aristotle's report): If matter were infinitely divisible, only points with no extension would remain. But putting together an arbitrary number of them, we would still get things without extension\footnote{This kind of subtle argument has its roots in the Elean philosophical school of Parmenides and Zeno (the one of the turtle-Achilles paradox). According to Diogenes La\"{e}rtius, Leucippus was a pupil of Zeno. Elea, nowadays Velia, located approximatively $90$ miles south-east of Naples, was a Greek colony flourishing in the Vth century BCE.}.  Democritus was largely ignored in ancient Athens. He is said to have been disliked so much by Plato that the latter wished all of his books burned. It should also be stressed that, differently from the subsequent Plato and Aristoteles, the atomists Leucippus and Democritus, following the scientific rationalist philosophy associated with the Miletus school, wanted to investigate in the Vth century BCE the {\em causes} of natural phenomena rather than their {\em significance}$\,$! 

The long history of the atomic hypothesis intersects twice with the history of the microscope. The {\em first intersection} simply occurs  because the invention of this instrument at the beginning of the seventeen century made it possible to observe the very irregular motion of particles immersed in a fluid. These observations of ``animated" or ``irritable" particles were made, among others, by  van Leeuwenhoek, Buffon, Spallanzani long before the British botanist Robert Brown\footnote{In \cite[Chapter 2]{N1}, Edward Nelson writes about Robert Brown: ``His contribution was to establish Brownian motion as an important
phenomenon, to demonstrate clearly its presence in inorganic as well as
organic matter, and to refute by experiment facile mechanical explanations
of the phenomenon".}.  Many other important contributions to the atomistic theory had come before Brown, among others, from Dalton and Avogadro.  By 1877 the kinetic theory asserting that Brownian motion of  particles
is caused by bombardment by the molecules of the fluid was rather well established. 

In 1877 \cite{BOL}, Boltzmann poses and solves the first large deviation and relative maximum entropy problem in history where his ``loaded dice" are actually molecules! Nevertheless, the theory is not open to experimental verification as the velocity of Brownian particles cannot be measured accurately. At the beginning of the twentieth century, there are still prominent scientists such as Ostwald and Mach who are not convinced of the existence of atoms and molecules due to their ``positivistic philosophical attitude" (Albert Einstein \cite[p.49]{Schilpp}). In 1905 Einstein (and , independently, in 1906 Smoluchowski), proposes a p.d.e.\ model which is open to the experimental check of measuring the diffusion coefficient, thereby circumventing the need to measure the velocity of Brownian particles. This is accomplished in 1908 by Perrin \cite[Section 3]{perrin} some $2,350$ years after Democritus' argument!  Meanwhile, in 1908, a fundamental step in the direction of a large body of modern science is taken by Paul Langevin in \cite{Lan}. He argues that the equation of motion has the form
\begin{equation}\label{langevin}
m\frac{d^2x}{dt^2}=-\gamma \frac{dx}{dt}+ F, \quad \gamma>0,
\end{equation}
where $F$ is a  complementary force which maintains the agitation of the particle which the viscous resistance would stop without it. This is the first stochastic differential equation in history which is written down long before the relative probabilistic foundations and concepts (Wiener process) are introduced! After important contributions by a number of theoretical physicists and engineers such as Fokker and Planck, the Nyquist-Johnson model for RLC networks with noisy resistors in 1928 \cite{Johnson,Nyq}, we come in 1930/31 to the accepted  model for physical Brownian motion in a conservative force field \cite{UO,Kap,Chand}\footnote{Ornstein and Uhlenbeck only considered the case of a quadratic potential leading to a Gauss-Markov model in phase space.}, \cite[Chapter 10]{N1} given by the stochastic oscillator
\begin{subequations}\label{OU34}
\begin{eqnarray}\label{OU3}
dx(t)&=&\phantom{-}v(t)\,dt,\\\label{OU4}
dv(t)&=&-\beta v(t)\,dt-\frac{1}{m}\nabla V(x(t))dt+\sigma dW_t,
\end{eqnarray}
\end{subequations}
with $x(t_0)=x_0$ and $v(t_0)=v_0$ a.s., where $w(t)$ is a standard $3$-dimensional Wiener process and Einstein's {\em fluctuation-dissipation relation} holds
\begin{equation}\label{EFD}
\sigma^2=2kT\beta.
\end{equation}
Here $k$ is Boltzmann's constant and $T$ is the absolute temperature of the fluid. The original Einstein-Smoluchowski theory is the high-friction limit of this model \cite[Theorem 10.1]{N1}\footnote{Like the Aristotelian $F=mv$, often observed in nature, is the high friction limit of the Newtonian $F=ma$ .}. Condition (\ref{EFD}) guarantees the existence and the Boltzmann-Gibbs nature of an invariant measure for (\ref{OU34}) with density
\begin{equation}\label{BG}\rho_{BG}(x,v)=Z^{-1}\exp\left[-\frac{H(x,v)}{kT}\right], \mbox{ for } H(x,v)=\frac{1}{2}m\|v\|^2+V(x),
\end{equation}
and $Z$ is a suitable normalizing constant (partition function),
see \cite{HP1,CGPcooling} for a generalization of this result.

These models have since played a central role in many areas of science besides microphysics such as electric circuits \cite{Nyq}, astronomy \cite{Chand}, mathematical finance since \cite{Bachelier}, biology, chemistry, etc. In more recent times, stochastic oscillators   play a central role in {\em cold damping feedback} where (\ref{OU34}) is replaced by
\begin{subequations}\label{OUC}
\begin{eqnarray}\label{OUC1}
dx(t)&=&\phantom{-}v(t)\,dt,\\\label{OUC2}
dv(t)&=&-\beta v(t)\,dt+u(x(t),v(t))dt-\frac{1}{m}\nabla V(x(t))dt+\sigma dW_t,
\end{eqnarray}
\end{subequations}
The purpose of this feedback control action is to reduce the effect of thermal noise on the motion of an oscillator by applying a viscous-like force, which is  the very first feedback control action  mathematically analyzed \cite{Max}.  James Clerk Maxwell  writes there: ``In one class of regulators of machinery, which we may call {\em moderators}, the resistance is increased by a quantity depending on the velocity". The first implementation on electrometers \cite{MZV} dates back to the fifties. Since then, it has been successfully employed in a variety of areas such as  atomic force microscopy (AFM) \cite{LMC} ({\em second intersection}!), see\footnote{Notice that the experimental apparatus here is partially inspired by that of Kappler \cite{Kap} with a light being shined onto a small mirror and the angle being measured through the position of the reflected spot a large distance away.} Figure \ref{SurfTop},  polymer dynamics \cite{DE,BBP} and nano to
meter-sized resonators, see \cite{FK,MG,SR,Vin,PV}. For (\ref{OUC}), the feedback control action $u(t)=-\alpha v(t)$, $\alpha>0$, asymptotically steers the phase space distribution to the steady state
\begin{equation}\label{StSt}
\bar{\rho}(x,v)=\bar Z^{-1}\exp\left[-\frac{H(x,v)}{kT_{{\rm eff}}}\right],
\end{equation}
where the {\em effective temperature} $T_{{\rm eff}}$ satisfies
\[T_{{\rm eff}}=\frac{\beta}{\beta+\alpha}T<T.
\]

\begin{figure}\begin{center}
\includegraphics[width=12cm]{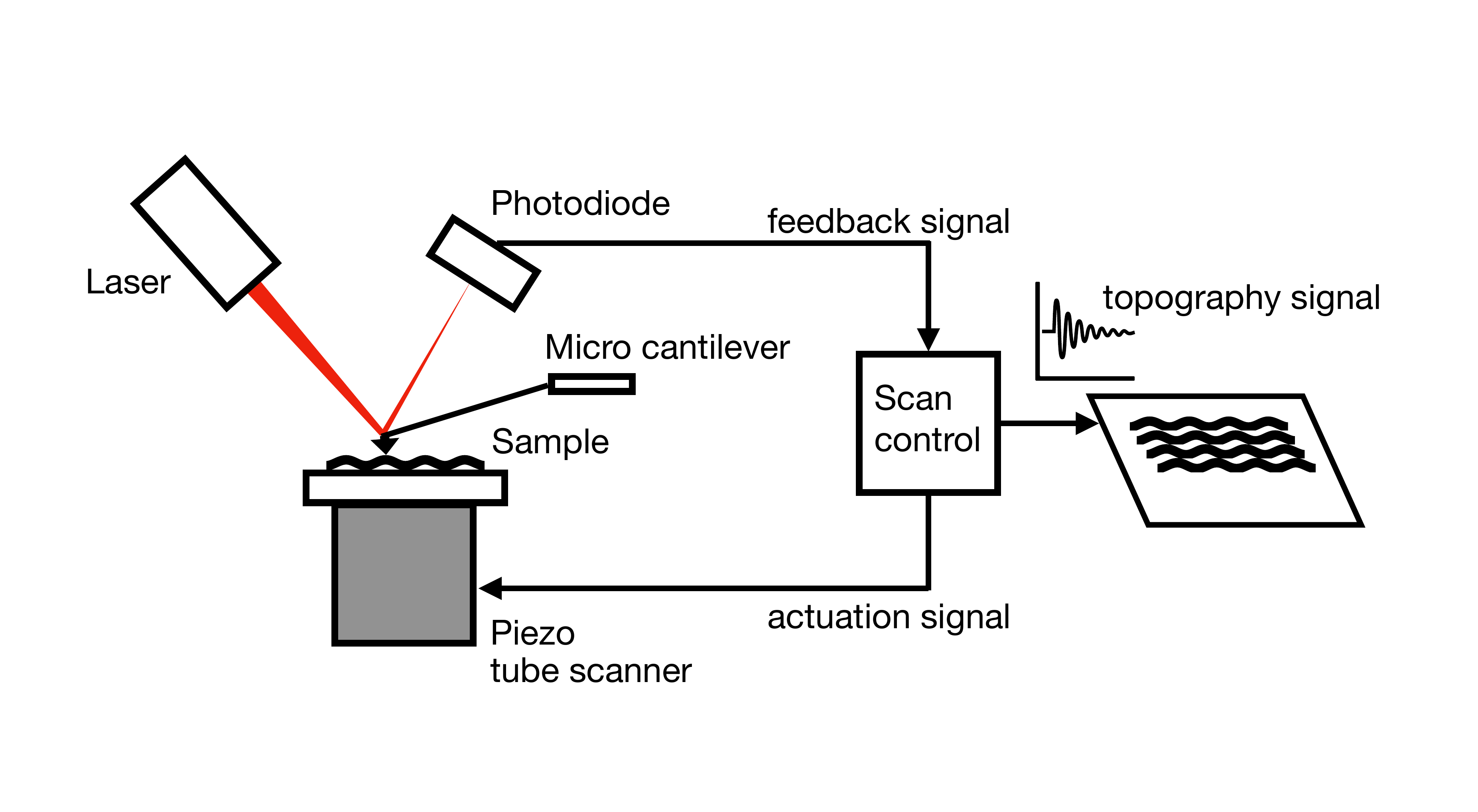} 
\caption{Surface topography: Velocity-dependent feedback control used to reduce thermal noise of a cantilever in Atomic Force Microscopy.} 
\label{SurfTop}
\end{center}\end{figure}

These new applications also pose new {\em physics} questions as the system is driven to a non-equilibrium {\em steady state} \cite {Q,KQ,BCD,PT1}. In \cite{FH}, a suitable {\em efficiency measure} for these diffusion-mediated devices was introduced which involves a class of stochastic control problems. Stochastic oscillators play also an important role in accelerating convergence of stochastic gradient descent for neural networks \cite[Chapter 6]{pavl}, \cite{COOSC}.

In \cite{CGPcooling}, the problem of asymptotically driving system (\ref{OUC}) to a desired {\em steady state} corresponding to reduced thermal noise was considered.
Among the feedback controls achieving the desired asymptotic transfer, it was found  that the least energy one is characterized by  {\em {time}-reversibility}. This problem has its roots in the classical covariance control of Skelton, Grigoriadis and collaborators \cite{HS1,HS2,SkeIke,GS,ZGS}. 

The problem of steering with minimum effort such a system in {\em finite time} to a target steady state distribution was also solved in  \cite{CGPcooling} as a generalized  Schr\"odinger bridge problem.   The system can then be maintained in the desired state state
through the optimal steady-state feedback control. The solution, {in the finite-horizon case,}  involves a space-time harmonic function $\varphi$ satisfying
\begin{equation}\label{sth}\frac{\partial\varphi}{\partial t}+v\cdot\nabla_x\varphi+(-\beta v-\frac{1}{m}\nabla_x V)\cdot\nabla_v\varphi+\frac{\sigma^2}{2}\Delta_{v}\varphi=0.
\end{equation} 
Here $-\log\varphi$ plays the role of an artificial, time-varying potential under which the desired evolution takes place. This two-step control strategy is effectively illustrated by the following simple Gaussian example. The system
\begin{eqnarray*}
dx(t)&=&\phantom{-}v(t)dt\\
dv(t)&=& -v(t)dt +u(t)dt -x(t)dt +dW_t
\end{eqnarray*}
is first optimally steered from time $t=0$  to time $t=1$ between the initial and final
Gaussian
marginals $\mathcal N(0,(1/2)I)$ and $\mathcal N(0,(1/2^4)I)$, respectively. The latter distribution is  then maintained through constant feedback in this terminal desired state; see Figure \ref{fig:1} where the transparent tube represents the $3$-standard deviation region of the state distribution.
\begin{figure}\begin{center}
\includegraphics[width=0.55\textwidth]{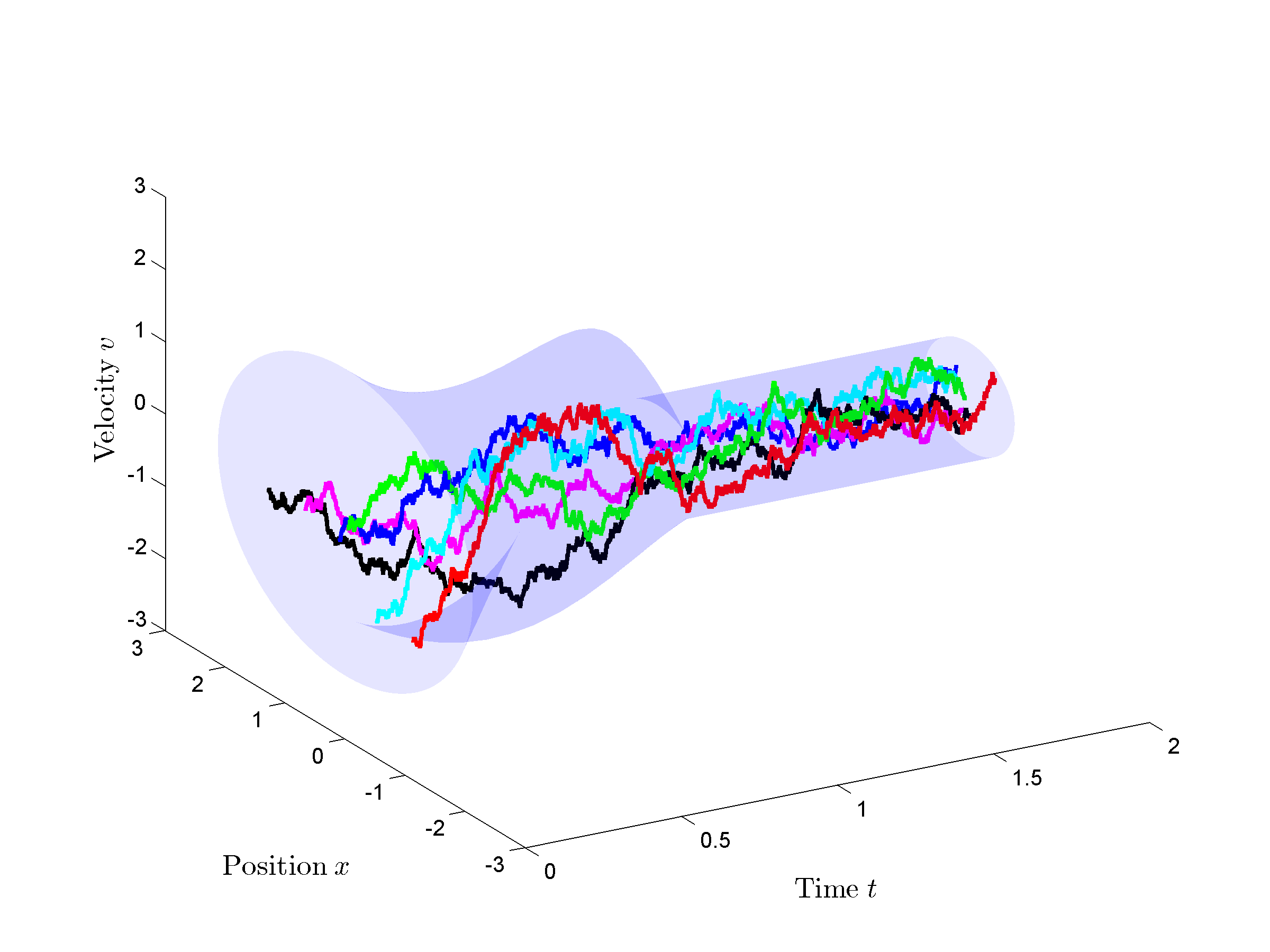}
   \caption{Inertial particles: trajectories in phase space}
   \label{fig:1}
\end{center}\end{figure}

\section{Minimizing the free energy}\label{connections}

We now proceed to clarify how the problems considered by Sinkhorn \cite{Sin64,Sin67} are connected to Schr\"odinger bridges and to thermodynamical free energy. To achieve this in the most transparent way, we turn to the discrete setting.

\subsection{Regularized transport problems}\label{RLPP}
Let us first recall the notion of the simplex of probability distributions on a finite set. Let $V$ be a vector space and $A\subseteq V$. The
{\em convex hull} \cite{rocka} of $A$, written $con A$, is the intersection of all
convex subsets of $V$ containing $A$. The convex hull of $n+1$ {\em affinely independent}\footnote{The points $x_1,x_2,\ldots,x_{n+1}$ are called affinely independent if every point $x$ in their convex hull admits a {\em unique} representation as convex combination of the points.} points of a Euclidean space is called an $n$-{\em simplex}. For example, a $1$-simplex is a line segment, a $2$-simplex is a triangle and a $3$-simplex is a tetrahedron, and so on.
Let ${\cal D}(\mathcal X)$ denote the family of all probability distributions on the sample space  $\mathcal X=\{1,2,\ldots,n\}$. Then ${\cal D}(\mathcal X)$ is an $(n-1)$-simplex whose vertices are the 
distributions $p_i(j)=\delta_{ij}$, where $\delta_{ij}$ is the Kronecker delta.

The discrete OMT problem \cite[Vol. I]{RR} has been popularized in the following form. Suppose there are $n$ mines with mine $i$ producing the fraction $p_i$ of the total production. There are also $n$ factories which need the raw material from the mines. To operate, factory $j$ needs the fraction $q_j$ of the total available supply. Let $C=(c_{ij})_{i,j=1}^n$ be a matrix of ``transportation costs"\footnote{$c_{ij}$ is the cost of transporting one unit of material from mine $i$ to factory $j$.} with nonnegative elements. On ${\cal D}(\mathcal X)$, we can then define a metric in the following way: Given the two probability distributions  $p,q\in \mathcal D(\mathcal X)$, let $\Pi(p,q)$ be the family of probability distributions on $\mathcal X\times\mathcal X$ that are ``couplings" of $p$ and $q$, namely $\pi\in\Pi(p,q)$ has marginals $p$ and $q$, respectively.  Any $\pi\in\Pi(p,q)$ represents a feasible transport plan, the quantity $\pi_{ij}$ representing the amount of the demand of factory $j$ which is satisfied by mine $i$. Then, the discrete OMT problem of minimizing the total cost of transportation while respecting the constraints leads to the {\em optimal transport distance} between $p$ and $q$  defined by
\begin{equation}\label{DW}d_C(p,q):=\min_{\pi\in\Pi(p,q)}\sum_{i,j}c_{ij}\pi_{ij}.
\end{equation}
When $c_{ij}=d(i,j)^2$, where $d(\cdot,\cdot)$ is a distance on $\mathcal X\times\mathcal X$,
\[W_2(p,q):=\left(d_C(p,q)\right)^{1/2}.
\] 
is called {\em earth mover distance (Wasserstein $2$-distance)}. It can be shown \cite[Proposition 2.2]{PC} that $W_2$ is a {\em bona fide} distance on $\mathcal D(\mathcal X)$. This distance has recently found important applications in many diverse fields of science such as economics, physics, engineering and probability, and in particular, in information engineering for problems of imaging (DTI, multimodal, color, etc), robust-efficient transport over networks, spectral analysis, collective dynamics, etc.  A regularized version of (\ref{DW}), which features important algorithmic/computational advantages \cite{Cuturi,PC},  is obtained by subtracting a term proportional to the entropy 
\begin{equation}\label{RDW}\inf_{\pi\in\Pi(p,q)}\left[\sum_{i,j}c_{ij}\pi_{ij}-\epsilon S(\pi)\right], \quad S(\pi)=-\sum_{ij}\pi_{ij}\log(\pi_{ij}),
\end{equation}
for $\epsilon>0$.
Notice, in particular, that the resulting functional
\begin{equation}\label{regfunct}
J(\pi)=\sum_{i,j}c_{ij}\pi_{ij}+\epsilon\sum_{ij}\pi_{ij}\log(\pi_{ij})
\end{equation}
is strictly convex in $\pi$.

\subsection{Thermodynamic systems: Statics}\label{thermostatics}
We consider a physical system with state space ${\cal X}=\{1,2,\ldots,M\}$. We can think of this mesoscopic description as originating from a microscopic description where the {\em phase space}, in Boltzmann's style, has undergone a ``coarse graining" through subdivision into small cells which is what we typically observe. Each of the cells represents a mesoscopic state. While the microscopic states are equally likely, this is no more true for the macroscopic states which correspond to different numbers of microstates.

For each macroscopic state $x$ we consider its {\em energy} $E_x\ge 0$. The function {$H: x\mapsto E_x$ is referred to as the {\em Hamiltonian}}. The thermodynamic states of the system are given by probability distributions on ${\cal Z}$ reflecting how many microscopic states correspond to the macroscopic ones, namely by ${\cal D}({\cal X})$. On  ${\cal D}({\cal X})$, we define the {\em internal energy} as the expected value of the Energy {\em observable} in state $\pi$, namely,
\begin{equation}
U(\pi)=\E_{\pi}\{H\}=\sum_xE_x\pi_x=\langle E,\pi\rangle,
\end{equation}
where $E$ denotes the $n$-dimensional vector with components $E_x$. Let us also introduce the {\em Gibbs entropy}
\begin{equation}
S_G(\pi)=kS(\pi)=-k\sum_x \pi_x\log \pi_x,
\end{equation}
where $k$ is Boltzmann's constant. As is well-known, $S_G$ is nonnegative and strictly concave on ${\cal D}({\cal X})$.
Let $\bar{E}$ be a  constant satisfying
\begin{equation}\label{energybounds}
\min_x E_x\le \bar{E}\le\frac{1}{n}\sum_xE_x.
\end{equation}

We can think of $\bar{E}$ as the energy of the underlying conservative microscopic system (the upper bound $\frac{1}{n}\sum_lE_l$  in (\ref{energybounds}) guarantees existence of a positive multiplier, see below). We now consider the following {\em Maximum Entropy} problem:
\begin{subequations}
\begin{eqnarray}\label{ME1}
\max \quad \{S_G(\pi) \,\mid \, \pi\in{\cal D}({\cal X})\}
\\{\rm subject \;to}\quad U(\pi)=\bar{E}.\label{ME1a}
\end{eqnarray}
\end{subequations}
This is an (important) instance of a class of maximum entropy problems originating with Boltzmann \cite{BOL}, see  \cite{PF} for a survey, where entropy is maximized over probability distributions that give the correct expectation of certain observables in accordance with known macroscopic quantities. 
The {\em Lagrangian function} is then given by
\begin{equation}\label{lagrangianSM}
{\cal L}(\pi,\lambda):=S_G(\pi)+\lambda(\bar{E}-U(\pi)),
\end{equation}
where  the Lagrange multiplier $\lambda$ is positive, corresponding to positive ``absolute temperatures'' $T=\lambda^{-1}$.
The problem is then equivalent to minimizing over  ${\cal D}({\cal X})$ the {\em Helmholtz Free energy} functional
\begin{equation}\label{freeenergy}
F(\pi,T)=U(\pi)-TS_G(\pi)=\sum_xE_x\pi_x+kT\sum_x \pi_x\log \pi_x. 
\end{equation}

Since $F$ is strictly convex on  ${\cal D}({\cal X})$, the first order optimality conditions are sufficient, and determine the unique minimizer in the form of the {\em Boltzmann distribution}\footnote{The letter $Z$ for the {\em partition function} was chosen by Boltzmann to indicate ``zust\"{a}ndige Summe" (relevant sum).}
\begin{equation}\label{Boltdistr}
\pi_B(x)=Z(T)^{-1}\exp\left[-\frac{E_x}{kT}\right], \mbox{ where }Z(T)=\sum_x\exp\left[-\frac{E_x}{kT}\right],
\end{equation}
see e.g. \cite{Ruelle}. Alternatively, it suffices to observe that 
 \begin{equation}\label{eq:F2}
F(\pi,T)=T\D(\pi\|\pi_B)-T\log Z(T)
\end{equation}
and invoke the properties of relative entropy. As is well-known, the Boltzmann distribution (\ref{Boltdistr}) tends to the uniform (maximum entropy) distribution as $T\nearrow+\infty$ and tends to concentrate on the set of minimal energy states as $T\searrow 0$. Hence, for $0<T<+\infty$, the Boltzmann distribution represents in a precise way a {\em compromise} between minimizing energy and maximizing entropy.

\subsection{Schr\"odinger and Sinkhorn,
redux}\label{SCHRSINK}

The fact that the Boltzmann distribution minimizes the free energy may be viewed as an elementary version of what is often called {\em Gibbs' variational principle}. Notice that the minimization of $F$ in (\ref{freeenergy}) is {\em unconstrained}. Nevertheless, we are often interested in minimizing the free energy under additional constraints. This is usually the case with natural evolutions which tend to maximal entropy configurations\footnote{According to Planck,  nature seems to favour high entropy states.} 
while respecting
certain constraints. In particular, we now consider a  constrained version of the minimization of (\ref{freeenergy}). In the notation of Section \ref{RLPP}, let ${\cal Z}=\mathcal X\times\mathcal X$ and consider the problem
\begin{equation}\label{constrainedfree2}
\inf_{\pi\in\Pi(p,q)} F(\pi,T). 
\end{equation}
Then, letting $\epsilon=kT$, comparing (\ref{freeenergy}) with (\ref{regfunct}) shows that (\ref{constrainedfree2}) coincides with the regularized optimal transport problem (\ref{RDW}). In particular, up to constants, {\em the negative Lagrangian} (\ref{lagrangianSM}) {\em for Problem} (\ref{ME1})-(\ref{ME1a}) {\em coincides with the functional} (\ref{regfunct}) {\em to be minimized in regularized optimal transport}.          On the other hand, because of (\ref{eq:F2}), Problem (\ref{RDW}) is equivalent to 
\begin{equation}\label{constrainedfree}
\min_{\pi\in\Pi(p,q)} \D(\pi\|\pi_B), 
\end{equation}
where
\begin{equation}\label{Boltdistr2}
\pi_B(i,j)=Z(T)^{-1}\exp\left[-\frac{c_{ij}}{kT}\right],\quad Z(T)=\sum_{ij}\exp\left[-\frac{c_{ij}}{kT}\right],
\end{equation}
which is a discrete counterpart of Problem \ref{static}. Naturally, this and the other maximum entropy problems of this section, also admit the large deviations interpretation of Section \ref{margedevmaxent}.

Let us now write the joint probability $\pi_B(i,j)$ as
\[\pi_B(i,j)=p_B(i)p_B(i,j),
\]
where $p_B(i,j)$ is the conditional probability. Introducing Lagrange multipliers for the linear constraints
\begin{eqnarray}\label{constraintuniformg1}
\sum_j \pi_{ij}&=&p_i,\quad i=1,2,\ldots,n\\\sum_i \pi_{ij}&=&q_j, \quad j=1,2,\ldots,n,
\label{constraintuniformg2}
\end{eqnarray}
and proceeding precisely as in Section \ref{derivation}, we readily get the following expression for the optimal $\pi$
\begin{equation}\label{optimumrepresentation}
\pi^*_{ij}=\hat{\varphi}(i)p_B(i,j)\varphi(j)
\end{equation}
where the non-negative functions $\hat{\varphi}$ and $\varphi$ satisfy the system
\begin{eqnarray}\label{DSS1}\hat{\varphi}(i)\left[\sum_jp_B(i,j)\varphi(j)\right]&=&p_i;\\
\varphi(j)\left[\sum_ip_B(i,j)\hat{\varphi}(i)\right]&=&q_j.
\label{DSS2}
\end{eqnarray}
Defining $\hat{\varphi}(0,i)=\hat{\varphi}(i)$, $\varphi(1,j)=\varphi(j)$, we see that (\ref{DSS1})-(\ref{DSS2}) can be replaced by the system
\begin{subequations}\label{eq:DSchrsystem}
\begin{eqnarray}\label{eq:DSchrsystemA}
&&\varphi(0,i)=\sum_jp_B(i,j)\varphi(1,j),\\
&&\hat{\varphi}(1,j)=\sum_ip_B(i,j)\hat{\varphi}(0,i)\label{eq:DSchrsystemB},\\&&\varphi(0,i)\cdot\hat{\varphi}(0,i)=p_i,\label{eq:DSchrsystemC}\\&&\varphi(1,j)\cdot\hat{\varphi}(1,j)=q_j.\label{eq:DSchrsystemD}
\end{eqnarray}
\end{subequations}

Let us write
\[\pi^*_{ij}=p_i\cdot p^*(i,j),
\]
and assume $p_i>0$ for all $i$. Dividing both sides of (\ref{optimumrepresentation}) by $p_i$ we get, in view of (\ref{eq:DSchrsystemC}), 
\begin{equation}\label{optimumrepresentation2}
p^*(i,j)=\frac{1}{\varphi(0,i)}p_B(i,j)\varphi(1,j)
\end{equation}
which should be compared to (\ref{multfunct}). It is interesting to write (\ref{optimumrepresentation2}) in matricial form. Let $P^*=(p^*(i,j))$ and $P_B=(p_B(i,j))$. Then (\ref{optimumrepresentation2}) gives
\begin{equation}\label{optimumrepresentation3}
P^*=\diag\left(\frac{1}{\varphi(0,1)},\ldots,\frac{1}{\varphi(0,n)}\right)P_B\diag\left(\varphi(1,1),\ldots,\varphi(1,n)\right).
\end{equation}

System (\ref{eq:DSchrsystem}) represents {\em a discrete counterpart of the Schr\"odinger system} (\ref{Schonestep1})-(\ref{Schonestep2})-(\ref{BConestep}). Existence for the latter, as already observed after (\ref{BConestep}), is extremely challenging with the first solution being provided by Robert Fortet in 1940 \cite{For} through a complex iterative scheme. The same problem  is much simpler in the discrete setting with the first convergence proof for the classical IPF procedure being provided in a special case by Sinkhorn in 1964 \cite{Sin64}. 
Indeed, consider the special case where both marginals are uniform distributions so that $p_i=q_i=1/n,  i=1,2,\ldots,n$. Let $C=(c_{ij})$ be the matrix of transportation costs. Then $\pi\in{\cal D}({\cal X}\times{\cal X})$ belongs  to $\Pi(p,q)$ if and only if it satisfies the constraints
\begin{eqnarray}\label{constraintuniform1}
\sum_j \pi_{ij}&=&\frac{1}{n},\quad i=1,2,\ldots,n\\\sum_i \pi_{ij}&=&\frac{1}{n}, \quad j=1,2,\ldots,n.
\label{constraintuniform2}
\end{eqnarray}
Thus, the matrix $n\pi$ must be {\em doubly stochastic} which was the original Sinkhorn problem\footnote{Sinkhorn: ``It is not the intent of this paper to obtain properties of this estimate..". Sinkhorn appears only concerned in establishing convergence of the iterative method to a doubly stochastic matrix without clearly connecting the latter to an optimization problem.}.

\section{The Fortet-IPF-Sinkhorn algorithm}\label{FIPFS}
\subsection{Continuous case}\label{contIPF}
In 1938-1940 Robert Fortet, a fine French analyst former student of Maurice Frechet, sets out to solve the problem of existence and uniqueness for the Schr\"odinger system
(\ref{Schonestep1})-(\ref{Schonestep2})-(\ref{BConestep}), left open by Schr\"odinger\footnote{Schr\"odinger thought existence and uniqueness should hold since the problem looked to him so natural except, possibly, in the case of very nasty marginals.} as well as by Bernstein in \cite{Ber}.
Fortet's proof in \cite{For0,For} is, to this day, the only {\em algorithmic} one and  in a rather general setting,
establishing convergence of successive approximations. More explicitly, let $g(x,y)$ be a nonnegative, continuous function on $\R\times\R$ bounded from above. Suppose $g(x,y)>0$ except  possibly for a  zero measure set for each fixed value of $x$ or of $y$. Suppose that $\rho_0(x)$ and $\rho_1(y)$ are continuous, nonnegative functions such that 
\[\int\rho_0(x)dx=\int\rho_1(y)dy.
\]
Suppose, moreover, that the integral
\[\int \frac{\rho_1(y)}{\int g(z,y)\rho_0(z)dz}dy
\]
is finite (this is Fortet's crucial hypothesis).
Then the system \cite[Theorem 1]{For}
\begin{eqnarray}
\hat\varphi(x)\int g(x,y)\varphi(y)dy&=&\rho_0(x),\label{OPT1}\\
\varphi(y)\int g(x,y)\hat\varphi(x)dx&=&\rho_1(y) \label{OPT2}
\end{eqnarray}
admits a solution $(\varphi(x),\hat\varphi(y))$ with $\varphi\ge 0$ continuous and $\hat\varphi\ge 0$ measurable. Moreover, $\hat\varphi(x)=0$ only where $\rho_0(x)=0$ and $\varphi(y)=0$ only where $\rho_1(y)=0$.

The result is proven by setting up a complex approximation scheme to show that $\cC(h) = h$ with
\begin{equation}\label{nonlineareq}
\cC(h)(\cdot) : =\int g(\cdot,y)\frac{\rho_1(y)dy}{\int g(z,y)\frac{\rho_0(z)}{h(z)}dz}
\end{equation}
has a fixed point. The map $\cC$ is considered on functions of class (K), namely functions $h:\R\rightarrow (\R\cup +\infty)$ which satisfy the following properties:
\begin{itemize}
\item[i)] $h$ is measurable;
\item[ii)] There exists $\alpha>0$ such that $h(x)\ge \alpha, \forall x\in\R$;
\item[iii)] For almost every $x\in\R$, $h(x) < +\infty$.
\end{itemize}
If $h_0$ and $h_1$ are of class (K) and $h_0\le h_1$ a.e., then $\cC(h_0)\le\cC(h_1)$. Moreover, on class (K) functions, $\cC$ is positively homogeneous of degree one. Unfortunately, the map $\cC$ does not map class (K) functions into class (K) functions as it does not preserve the property of being bounded away from zero. This is a fundamental difficulty of the continuous case which Fortet circumvents through his brilliant but complex approximation scheme involving three sequences of functions. This difficulty can be altogether avoided in the discrete case through a suitable positivity assumption, see Theorem \ref{Fundtheorem} below.   Notice that the one-dimensional heat kernel
\[
g(x,y)=p(0,x,1,y)=\frac{1}{\sqrt{2\pi\gamma}}\exp\left[-\frac{|x-y|^2} {2\gamma}\right].
\]
satisfies all assumptions of Fortet's theorem. Uniqueness, in the sense described after formula (\ref{BConestep}) , namely uniqueness of rays, is much easier to establish. In \cite{EP}, most of Fortet's paper has been revisited filling in all the gaps and  explaining the meaning of the various steps of his elaborate approach. Another recent paper in this direction is \cite{leo3}.

\subsection{Discrete case}\label{ITERATIVEALGORITHM}

In 1940, an  {\em iterative proportional fitting} (IPF)  procedure, was proposed in the statistical literature on contingency tables \cite{DS1940}. Convergence for the IPF algorithm was first established (in a special case) by Richard Sinkhorn in 1964  \cite{Sin64}. The iterates were shortly afterwards shown to converge to a ``minimum discrimination information" \cite{IK,Fien,csiszar0}, namely to a minimum entropy distance. This line of research, usually called {\em Sinkhorn algorithms}, continues to this date, see e.g. \cite{Cuturi,Alt,TCDP}.

We 
now state and, later on,
establish the following fundamental result.

\begin{theorem} \label{Fundtheorem}  Let $\mathcal X=\{1,2,\ldots,n\}$ and $p,q\in{\cal D}(\mathcal X)$. Assume that the $n\times n$ matrix $G=\left(g_{ij}\right)$ has all positive elements.  Then, there exist  vectors
$\varphi(0,\cdot),\,\hat\varphi(1,\cdot)$ with positive entries such that

\begin{subequations}\label{eq:DSchroesystem}
\begin{eqnarray}\label{eq:DSchroesystemA}
&&\varphi(0,i)=\sum_j g_{ij}\varphi(1,j),\\
&&\hat{\varphi}(1,j)=\sum_i g_{ij}\hat{\varphi}(0,i)\label{eq:DSchroesystemB},\\&&\varphi(0,i)\cdot\hat{\varphi}(0,i)=p_i,\label{eq:DSchroesystemC}\\&&\varphi(1,j)\cdot\hat{\varphi}(1,j)=q_j.\label{eq:DSchroesystemD}
\end{eqnarray}
\end{subequations}
The pair $\varphi(0,\cdot),\,\hat\varphi(1,\cdot)$ is unique up to multiplication of $\varphi(0,\cdot)$  by a positive constant $\alpha$ and division of $\hat\varphi(1,\cdot)$ by the same constant $\alpha$.
\end{theorem}

We  first set up a natural iterative scheme (Fortet-IPF-Sinkhorn) for system (\ref{eq:DSchroesystem}).  
We introduce the following linear maps on $\R^n_+=\{x\in\R^n: x_i\ge 0\}$, the positive orthant of $\R^n$:
\begin{eqnarray*}
\cE\;:\; &x\mapsto y&=\sum_jg_{ij}x_j,\\
\cE^\dagger \;:\; &x\mapsto y&=\sum_ig_{ij}x_i.
\end{eqnarray*}
Here and in the sequel $\dagger$ denotes adjoint\footnote{Our use of the adjoint for the map $\cE$ is consistent with the standard notation in diffusion processes where the Fokker-Planck (forward) equation involves the adjoint of the {\em generator} appearing in the backward Kolmogorov equation.}.We also define the following nonlinear maps on the interior of the positive orthant ${\rm int}\;\R^n_+=\{x\in\R^n: x_i> 0\}$
\begin{eqnarray*}
\cD_0\;:\;&x\mapsto y&= \frac{p}{x}\\
\cD_1\;:\;&x\mapsto y&= \frac{q}{x}
\end{eqnarray*}
where division of vectors is {\em componentwise}. On $\R^n_+$, consider also the composition of the four maps
\begin{equation}\label{Composition}\cC:=\cE \circ\cD_1\circ\cE^\dagger\circ\cD_0. 
\end{equation}
This is just the discrete counterpart of  map (\ref{nonlineareq}). Consider the vector iteration
\begin{equation}\label{Iteration}
\varphi_{k+1}(0)=\cC(\varphi_k(0)), \quad \varphi_0(0)=\mathds{1},
\end{equation}
where $\mathds{1}^\dagger=(1,1,\ldots,1)$. Observe first that (\ref{Iteration}) stays in the interior of $\R^n_+$ even when the marginals $p$ and/or $q$ have some zero components. Indeed, since $g_{ij}>0, \forall (i,j)$,  maps $\cE$ and $\cE^\dagger$ map $\R^n_+$ into ${\rm int}\;\R^n_+$. It follows that the componentwise divisions of $\cD_0$ and $\cD_1$ are well defined and $\cC:{\rm int}\;\R^n_+\mapsto{\rm int}\;\R^n_+$.  Next, we want to show that the sequence generated by (\ref{Iteration}) converges to a fixed point of $\cC$, thereby proving Theorem \ref{Fundtheorem}.
Theorem \ref{Fundtheorem} asserts, in particular, that {\em uniqueness} for the Schr\"odinger system (\ref{eq:DSchroesystem}) concerns {\em rays} in the positive orthant. This suggests that contractivity of the map (\ref{Composition}) should be established on the set of rays endowed with a {\em projective metric}. This was accomplished in \cite[Theorem 3]{GP}. It extends the approach of Franklin and Lorenz \cite{FL} dealing with the scaling of nonnegative matrices. We devote the next subsection to some key results on Hilbert's projective metric in which the rays convergence can be proved.


\subsection{Hilbert's projective metric}
This metric was introduced in 1895 \cite{hilbert1895gerade}. In 1957 \cite{birkhoff1957extensions}, Garrett Birkhoff proved  crucial contractivity result in this metric that permits to establish existence of solutions  of linear equations on cones (such as the Perron-Frobenius theorem (Theorem \ref{perrontheorem} below)). This result was extended to certain nonlinear maps by Bushell \cite{bushell1973projective,bushell1973hilbert}. Besides the ergodic theory for Markov chains,  the Birkhoff-Bushell results have been applied  to positive integral operators and to positive definite matrices \cite{bushell1973hilbert, lemmens2013birkhoff}. In recent times, this geometry has proven useful in  problems concerning  communication and computations over networks (see \cite{tsitsiklis1986distributed}. Other significant applications have been developed by  Sepulchre and collaborators \cite{sepulchre2010consensus,sepulchre2011contraction,BFS}. These concern  consensus in non-commutative spaces and metrics for spectral densities. We mention also applications to  in  quantum information theory \cite{reeb2011hilbert}. On the more mathematical side, a  survey on the applications in analysis is \cite{lemmens2013birkhoff}. The use of the Hilbert metric is crucial in the nonlinear Frobenius-Perron theory \cite{lemmens2012nonlinear}. A further extension of the Perron-Frobenius theory beyond linear positive systems and monotone systems has been recently proposed in \cite{FS}.

Applications of the Birkhoff-Bushell contractivity results to the topics of this paper apparently initiated in 1989 with the paper  \cite{FL} which deals with scaling of nonnegative matrices. In \cite{GP}, we showed that the Schr\"odinger bridge for Markov chains and quantum channels can be efficiently obtained from the fixed-point of a map which  contracts the {\em Hilbert metric}.  In \cite{CGP9}, a similar approach was taken in the context of diffusion processes leading to  a new proof of a classical result of Jamison on existence and uniqueness for the Schr\"odinger bridge and  providing as well an efficient computational scheme for both Schr\"odinger Bridges  and OMT. This new computational approach can be effectively employed, for instance, in image interpolation. There are, however, some fundamental difficulties in using this approach in Schr\"odinger's original setting. These are outlined in Remark \ref{HMC} below.

Following \cite{bushell1973hilbert}, we recall some basic concepts and results of this theory.

Let $\cS$ be a real Banach space and let $\cK$ be a closed solid cone in $\cS$, i.e., $\cK$ is closed with nonempty interior ${\rm int}\,\cK$ and is such that $\cK+\cK\subseteq \cK$, $\cK\cap -\cK=\{0\}$ as well as $\lambda \cK\subseteq \cK$ for all $\lambda\geq 0$. Define the partial order
\[
x\preceq y \Leftrightarrow y-x\in\cK,\quad x< y \Leftrightarrow y-x\in{\rm int}\,\cK
\]
and for $x,y\in\cK_0:=\cK\backslash \{0\}$, define
\begin{eqnarray*}
M(x,y)&:=&\inf\, \{\lambda\,\mid x\preceq \lambda y\}\\
m(x,y)&:=&\sup \{\lambda \mid \lambda y\preceq x \}.
\end{eqnarray*}
Then, the Hilbert metric is defined on $\cK_0$ by
\[
d_H(x,y):=\log\left(\frac{M(x,y)}{m(x,y)}\right).
\]
Strictly speaking, it is a {\em projective} metric since it is invariant under scaling by positive constants, i.e.,
$d_H(x,y)=d_H(\lambda x,\mu y)$ for any $\lambda>0, \mu>0$ and $x,y\in{\rm int}\,\cK$. Thus, it is actually a  distance between rays. If $U$ denotes the unit sphere in $\cS$, $\left({\rm int}\,\cK\cap U,d_H\right)$ is a metric space.
\begin{example}\label{ex1} Let $\cK=\R^n_+=\{x\in\R^n: x_i\ge 0\}$ be the positive orthant of $\R^n$. Then, for $x,y\in{\rm int}\R^n_+$, namely with all positive components,
\[d_H(x,y)=\log\max\{x_iy_j/y_ix_j\}.
\]
\end{example}

Another very important example for applications in many diverse areas of statistics, information theory, control, etc., is the cone of Hermitian, positive semidefinite matrices.
\begin{example}\label{ex2}
Let $\cS=\{X=X^\dagger\in\C^{n\times n}\}$, where $\dagger$ denotes here transposition plus conjugation and, more generally, adjoint. Let $\cK=\{X\in\cS: X\ge 0\}$ be the positive semidefinite matrices. Then, for $X,Y\in{\rm int}\,\cK$, namely positive definite, we have
\[d_H(X,Y)=\log\frac{\lambda_{\max}\left(XY^{-1}\right)}{\lambda_{\min}\left(XY^{-1}\right)}=\log\frac{\lambda_{\max}\left(Y^{-1/2}XY^{-1/2}\right)}{\lambda_{\min}\left(Y^{-1/2}XY^{-1/2}\right)}.
\]
It is closely connected to the Riemannian (Fisher-information) metric
\[
 d_R(X,Y)=\|\log\left(Y^{-1/2}XY^{-1/2}\right)\|_{F}=
\sqrt{\sum_{i=1}^n[\log\lambda_i\left(Y^{-1/2}XY^{-1/2}\right)]^2}.
\]
\end{example}

Notice that, in the two examples above, Hilbert's pseudo-metric  puts the boundary of
the cone at  infinite distance from any interior point.

A map $\cE:\cK\rightarrow\cK$ is called {\em non-negative}. It is called {\em positive} if $\cE:{\rm int}\,\cK\rightarrow{\rm int}\,\cK$. If $\cE$ is positive and $\cE(\lambda x)=\lambda^p\cE(x)$ for all $x\in{\rm int}\,\cK$ and positive $\lambda$, $\cE$ is called {\em positively homogeneous of degree $p$} in ${\rm int}\,\cK$.
For a positive map $\cE$, the {\em projective diameter} is befined by
\begin{eqnarray*}
\Delta(\cE):=\sup\{d_H(\cE(x),\cE(y))\mid x,y\in {\rm int}\,\cK\}
\end{eqnarray*}
and the {\em contraction ratio} by
\begin{eqnarray*}
\kappa(\cE):=\inf\{\lambda \mid d_H(\cE(x),\cE(y))\leq \lambda d_H(x,y),\forall x,y\in{\rm int}\,\cK\}.
\end{eqnarray*}
Finally, a map $\cE:{\cS}\rightarrow\cS$ is called {\em monotone increasing} if $x\le y$ implies $\cE(x)\le\cE(y)$.
\begin{theorem}[\cite{bushell1973hilbert}] \label{poshom}Let $\cE$ be a monotone increasing positive mapping which is positive
homogeneous of degree $p$ in ${\rm int}\,\cK$. Then the contraction $\kappa(\cE)$ does not exceed $p$. In particular, if $\cE$ is a positive linear mapping, $\kappa(\cE)\le1$.
\end{theorem}

\begin{theorem}[\cite{{birkhoff1957extensions},bushell1973hilbert}]\label{BBcontraction}
Let $\cE$ be a positive linear map. Then
\begin{equation}\label{condiam}
\kappa(\cE)=\tanh(\frac{1}{4}\Delta(\cE)).
\end{equation}
\end{theorem}

\begin{theorem}[\cite{bushell1973hilbert}]\label{BBcontraction2}
Let $\cE$ be either
\begin{itemize}
\item[\rm (a)] a monotone increasing positive mapping which is positive homogeneous of degree $p (0<p<1)$ in ${\rm int}\,\cK$, or
\item[\rm (b)] a positive linear mapping with finite projective diameter.
\end{itemize}
Suppose the metric space $Y=\left({\rm int}\,\cK\cap U,d_H\right)$ is complete. Then, in case $(a)$ there exists a unique $x\in{\rm int}\,\cK$ such that $\cE(x)=x$, in case $(b)$ there exists a unique positive eigenvector of $\cE$ in $Y$.
\end{theorem}

This result provides a far-reaching generalization of the celebrated Perron-Frobenius theorem  \cite{birkhoff1962Perron-Frobenius} (Theorem \ref{perrontheorem} below). Notice that in both Examples \ref{ex1} and \ref{ex2}, the space $Y=\left({\rm int}\,\cK\cap U,d_H\right)$ is indeed complete \cite{bushell1973hilbert}.
We also note that there are other metrics as well that are contracted by positive monotone maps. For instance, the closely related {\em Thompson metric} \cite{Tho} $d_T(x,y)=\log\max\{M(x,y),m^{-1}(x,y)\}$. The Thompson metric is a {\em bona fide} metric on $\cK$. It has been, for instance, employed in \cite{LW,CPSV,BFS}.

\begin{remark}\label{HMC}
{\em If we try to use a similar approach to prove existence for the Schr\"odinger system (\ref{OPT1})-(\ref{OPT2}), we may expect that in an infinite dimensional setting questions of boundness or integrability might become delicate. The main difficulty, however, lies here with two other issues. To introduce them, let us observe that in the Birkhoff-Bushell theory we have linear or nonlinear iterations which remain in the {\em interior of a cone}. For example, in the application of the Perron-Frobenius theorem to the ergodic theory of Markov chains, the assumption that there exists a power of the transition matrix with all strictly positive entries ensures that the evolution of the probability distribution occurs in the  {\em interior} of the positive orthant (intersected with the simplex). The first difficulty is that the natural function space cones such as $L^1_+$ ($L^2_+$), namely integrable (square integrable) nonnegative functions on $\R^d$, have {\em empty interior}\,!  The second difficulty is that, even if we manage to somehow define a suitable function space cone with nonempty interior\footnote{This was indeed accomplished in \cite{CGP9}.}, the nonlinear map $\Omega$ defined in (\ref{nonlineareq}) cannot map the interior of the cone into itself. Precisely to overcome this difficulty  in \cite{CGP9} the two marginals were assumed to have compact support. We see here once more, from a slightly different angle, how much more challenging the continuous case is.}\hfill$\Box$
\end{remark}

\subsection{Proof of Theorem \ref{Fundtheorem}}
We begin with
three preliminary results.
\begin{lemma}\label{Linear}
Consider the maps $\cE$ and $\cE^\dagger$. We have the following bounds on their contraction ratios:
\begin{equation}\label{strictcon}
\kappa(\cE)=\kappa(\cE^\dagger)=\tanh\left(\frac{1}{4}\Delta(\cE)\right)<1.
\end{equation}
\end{lemma}
\begin{proof}
Observe that $\cE$ is a positive {\em linear} map and its projective diameter is
\begin{eqnarray*}
\Delta(\cE)&=&\sup\{d_H(\cE(x),\cE(y)) \mid x_i>0,\,y_i>0\}\\
&=&\sup\{\log\left(\frac{g_{ij}g_{k\ell}}{g_{i\ell}g_{kj}}\right)\mid1\leq  i,j,k,\ell\leq n\}.
\end{eqnarray*}
It is finite since all entries $g_{ij}$'s are positive. It now follows from Theorem \ref{BBcontraction} that its contraction ratio satisfies (\ref{strictcon}). Similarly for the adjoint map $\cE^\dagger$. \end{proof}
\begin{lemma}\label{inversion}
\[\kappa(\cD_0)\le1,\quad \kappa(\cD_1)\le1.\]
\end{lemma}
\begin{proof} Observe that when 
$p$ and $q$ have positive entries, both $\cD_0$ and $\cD_1$
are isometries in the Hilbert metric. Indeed, for vectors $x,y$ in the interior of $\R^n_+$, inversion and element-wise scaling are both isometries for the Hilbert metric as the following calculations show \begin{eqnarray*}
d_H(x,y)&=&\log\left((\max_i (x_i/y_i))\frac{1}{\min_i (x_i/y_i)}\right)\\
&=&\log\left(\frac{1}{\min_i ((x_i)^{-1}/(y_i)^{-1})}\max_i ((x_i)^{-1}/(y_i)^{-1})\right)\\&=&d_H(x^{-1},y^{-1})
\end{eqnarray*}
where $x^{-1}$ and $y^{-1}$ are  obtained from $x$ and $y$, respectively, through componentwise inversion. Moreover, let $px$ and $py$ be the vectors with components $p_ix_i$ and $p_iy_i$, respectively. Then
\begin{eqnarray*}
d_H(px,py)&=&\log\frac{\max_i ((p_i x_i)/(p_i y_i))}{\min_i ((p_i x_i)/(p_i y_i))}\\
&=&\log\frac{\max_i (x_i/y_i)}{\min_i (x_i/y_i)}=d_H(x,y).
\end{eqnarray*}
If $p$ has  zero entries, then the second equality above needs to be replaced by the inequality ``$\le$''.
\end{proof}
\begin{lemma}\label{compothm} The composition
\begin{equation}\label{composition}\cC=\cE\circ \cD_1\circ \cE^\dagger\circ\cD_0
\end{equation}
contracts the Hilbert metric with contraction ratio $\kappa(\cC)<1$, namely
\[
d_H(\cC(x),\cC(y))<d_H(x,y), \quad \forall x,y\in{\rm int}\,\R^n_+.
\]
\end{lemma}
\begin{proof}The result follows at once from Lemmas \ref{Linear} and \ref{inversion}.
\end{proof}

We now complete the proof of Theorem \ref{Fundtheorem}.
The set of rays in ${\rm int}\R^n_+$  with the Hilbert metric is complete (this can be proven by intersecting ${\rm int}\R^n_+$ with the unit sphere, see \cite[Section 4]{bushell1973hilbert}). By Lemma \ref{compothm}, $\cC$ contracts the Hilbert metric. By the Banach-Caccioppoli Contraction Mapping Theorem, there exists a unique ray in ${\rm int}\R^n_+$ to which the iteration (\ref{Iteration}) (starting from any vector in ${\rm int}\R^n_+$) converges. Next, we prove that, due to (\ref{eq:DSchroesystemD})-(\ref{eq:DSchroesystemD}), the iteration (\ref{Iteration}) actually converges to fixed vector. Since the iteration has a fixed ray, we have that for some positive constant $\alpha$
\[
\alpha\cdot\varphi(0)=\cC(\varphi(0)),
\]
where the composed map is $\cC:=\cE \circ\cD_1\circ\cE^\dagger\circ\cD_0$.
From this we can obtain
\begin{eqnarray*}
\hat\varphi(1)&=&\cE^\dagger(\hat\varphi(0)),\\
\varphi(0)&=&\cE(\varphi(1)),
\end{eqnarray*}
while
\begin{eqnarray*}
\hat\varphi(1)\varphi(1)&=& q\mbox{ and}\\
\alpha\hat\varphi(0)\varphi(0)&=& p, 
\end{eqnarray*}
where, as usual, multiplication is componentwise.
Let $\langle\cdot,\cdot\rangle$ denote the scalar product in $\R^n$. Since $p$ and $q$ are probability distributions, we have
\begin{eqnarray*}
1&=& \alpha\langle \hat\varphi(0),\varphi(0)\rangle\\
&=& \alpha\langle \hat\varphi(0),\cE(\varphi(1))\rangle\\
&=&\alpha\langle \cE^\dagger(\hat\varphi(0)),\varphi(1)\rangle\\
&=&\alpha\langle \hat\varphi(1),\varphi(1)\rangle\\
&=&\alpha.
\end{eqnarray*}
Thus $\alpha=1$, the iteration (\ref{Iteration}) converges to a fixed point (vector in $\R^n_+$) and the four vectors $(\varphi(0), \hat{\varphi}(0), \varphi(1),\hat{\varphi}(1))$ satisfy the Schr\"{o}dinger system (\ref{eq:DSchroesystem}). The vectors $\varphi(0)$ and $\hat\varphi(1)$ have all positive components.                      

\section{Efficient vs.\ robust routing for networks flows}\label{MENF}
While Problem \ref{constrainedfree}, and the corresponding equivalent regularized OMT (\ref{RDW}), are the discrete counterparts of Problem \ref{static}, it is apparent that the discrete counterpart of the ``dynamic" Schr\"odinger Bridge  Problem \ref{bridge} is still missing. Before we turn to dynamic problems with discrete state space, let us mention that there is also work on  discrete time and continuous state-space \cite{XS,Beghi, Bak1, GT1,Bak2,Bak5}. This literature mostly deals with Gaussian distributions,  both, in the finite and infinite horizon case, with and without noise in the dynamics 
and, 
with and without constraints. The case of regularized transport on discrete metric graphs has been studied by L\'eonard in \cite{leoD}.

\subsection{Generalized bridge problems}\label{GBProblems}
 Consider a directed, strongly connected (i.e., with at least one path joining each pair of vertices), aperiodic graph ${\bf G}=(\mathcal X,\mathcal E)$ with vertex set $\mathcal X=\{1,2,\ldots,n\}$ and edge set $\mathcal E\subseteq \mathcal X\times\mathcal X$.  Time is discrete and taken in ${\mathcal T}=\{0,1,\ldots,N\}$.
 
 Let ${\mathcal FP}_0^N\subseteq\mathcal X^{N+1}$ denote the family of feasible paths $x=(x_0,\ldots,x_N)$ of length $N$, namely paths such that $x_ix_{i+1}\in\mathcal E$ for $i=0,1,\ldots,N-1$.  We seek a probability distribution $\fP$ on ${\mathcal {FP}}_0^N$ with prescribed initial and final marginal probability distributions $\nu_0(\cdot)$ and $\nu_N(\cdot)$, respectively, and such that the resulting random evolution
is closest to a ``prior'' measure $\fM$ on ${\mathcal {FP}}_0^N$ in a suitable sense.

The prior law for our problem is induced by the Markovian evolution
 \begin{equation}\label{FP}
\mu_{t+1}(x_{t+1})=\sum_{x_t\in\mathcal X} \mu_t(x_t) m_{x_{t}x_{t+1}}(t)
\end{equation}
with nonnegative distributions $\mu_t(\cdot)$ over $\mathcal X$, $t\in{\mathcal T}$, and weights $m_{ij}(t)\geq 0$ for all indices $i,j\in{\mathcal X}$ and all times. Moreover, to respect the topology of the graph, $m_{ij}(t)=0$ for all $t$ whenever $ij\not\in\mathcal E$.  Often, but not always, the matrix
\begin{equation}\label{eq:matrixM}
M(t)=\left[ m_{ij}(t)\right]_{i,j=1}^n
\end{equation}
may not depend on $t$.
The rows of the transition matrix $M(t)$ do not necessarily sum up to one, so that the ``total transported mass'' is not necessarily preserved.  It occurs, for instance, when $M$ simply encodes the topological structure of the network with $m_{ij}$ being zero or one, depending on whether a certain link exists (i.e., when $M$ represents the {\em adjacency matrix} of the graph).
The evolution \eqref{FP}, together with measure $\mu_0(\cdot)$, which we assume positive on $\mathcal X$, i.e.,
\begin{equation}\label{eq:mupositive}
\mu_0(x)>0\mbox{ for all }x\in\mathcal X,
\end{equation}
 induces
a measure $\fM$ on ${\mathcal {FP}}_0^N$ as follows. It assigns to a path  $x=(x_0,x_1,\ldots,x_N)\in{\mathcal {FP}}_0^N$ the value
\begin{equation}\label{prior}\fM(x_0,x_1,\ldots,x_N)=\mu_0(x_0)m_{x_0x_1}\cdots m_{x_{N-1}x_N},
\end{equation}
and gives rise to a flow
of {\em one-time marginals}
\[\mu_t(x_t) = \sum_{x_{\ell\neq t}}\fM(x_0,x_1,\ldots,x_N), \quad t\in\mathcal T.\]
\begin{definition} We denote by ${\mathcal P}(\nu_0,\nu_N)$ the family of probability distributions on ${\mathcal {FP}}_0^N$ having the prescribed marginals $\nu_0(\cdot)$ and $\nu_N(\cdot)$.
\end{definition}

We seek a distribution in this set which is closest to the prior $\fM$ in {\em relative entropy} where, for $P$ and $Q$ measures on $\mathcal X^{N+1}$,  the relative entropy (divergence, Kullback-Leibler index) $\D(P\|Q)$ is
\begin{equation*}
\D(P\|Q):=\left\{\begin{array}{ll} \sum_{x}P(x)\log\frac{P(x)}{Q(x)}, & \support (P)\subseteq \support (Q),\\
+\infty , & \support (P)\not\subseteq \support (Q),\end{array}\right.
\end{equation*}
Here, by definition,  $0\cdot\log 0=0$.
Naturally, while the value of $\D(P\|Q)$ may turn out negative due to miss-match of scaling (in case $Q=\fM$ is not a probability measure), the relative entropy is always jointly convex. Thus,
we are led to the {\em Schr\"odinger Bridge  Problem} (SBP):

\begin{problem}\label{prob:optimization}
Determine
 \begin{eqnarray}\label{eq:optimization}
\fM^*[\nu_0,\nu_N]:={\rm argmin}\{ \D(P\|\fM) \mid  P\in {\mathcal P}(\nu_0,\nu_N)
\}.
\end{eqnarray}
\end{problem}
\noindent
The following result may be proven in the usual way, cf. Section \ref{SCHRSINK} and  \cite{PT2,CGPT1,CGPT2}.
\begin{theorem}\label{solbridge} Assume that the entries of the matrix product 
\[G:=M(0)M(1) \cdots M(N-2)M(N-1)
\] 
are all positive.
Then there exist nonnegative functions  $\varphi(\cdot)$ and $\hat{\varphi}(\cdot)$ on $\cT\times\mathcal X$ satisfying
\begin{subequations}\label{eq:Schroedingersystem}
\begin{eqnarray}\label{Schroedingersystem1}
\varphi(t,i)&=&\sum_{j}m_{ij}(t)\varphi(t+1,j),
\\\hat{\varphi}(t+1,j)&=&\sum_{i}m_{ij}(t)\hat{\varphi}(t,i),\label{Schroedingersystem2}
\end{eqnarray}
for $t\in\{0,1,\cdots,N-1\}$, along with the (nonlinear) boundary conditions
\begin{eqnarray}\label{bndconditions1}
\varphi(0,x_0)\hat{\varphi}(0,x_0)&=&\nu_0(x_0)\\\label{bndconditions2}
\varphi(N,x_N)\hat{\varphi}(N,x_N)&=&\nu_N(x_N),
\end{eqnarray}
\end{subequations}
for $x_0, x_N\in\mathcal X$.
Moreover, the solution $\fM^*[\nu_0,\nu_N]$ to Problem \ref{prob:optimization} is unique and obtained by
\[
\fM^*(x_0,\ldots,x_N)=\nu_0(x_0)\pi_{x_0x_{1}}(0)\cdots \pi_{x_{N-1}x_{N}}(N-1),
\]
where the one-step transition probabilities
\begin{equation}\label{OPTTRANSITION1}
\pi_{ij}(t):=m_{ij}(t)\frac{\varphi(t+1,j)}{\varphi(t,i)}
\end{equation}
are well defined.
\end{theorem}

As usual, factors $\varphi$ and $\hat{\varphi}$ are unique up to multiplication of $\varphi$ by a positive constant and division of $\hat{\varphi}$ by the same constant.
Let $\varphi(t)$ and $\hat{\varphi}(t)$ denote the column vectors with components $\varphi(t,i)$ and $\hat{\varphi}(t,i)$, respectively, with $i\in\mathcal X$. In matricial form, (\ref{Schroedingersystem1}), (\ref{Schroedingersystem2}
) and (\ref{OPTTRANSITION1}) read
\begin{equation}
\varphi(t)=M(t)\varphi(t+1),\; ~~\hat{\varphi}(t+1)=M(t)^T\hat{\varphi}(t),
\end{equation}
and
\begin{equation}
	\Pi(t):=[\pi_{ij}(t)]=\diag(\varphi(t))^{-1}M(t)\diag(\varphi(t+1)).
\end{equation}
We see that the scheduling of the transport plan amounts to modifying the prior transition mechanism.  Therefore, this brings us to an alternative interpretation of the Schr\"odinger Problem, as a special Markov Decision Processes' problem \cite{Put,Berts}. This can be accomplished, once more, through a generalized multiplicative functional transformation (\ref{multfunct}) even when the prior is not a probability measure.

\subsection{Invariance of most probable paths}

In  \cite[Section 5]{DP}, Dai Pra established an interesting path-space  property of the Schr\"{o}dinger bridge for diffusion processes, namely that the ``most probable path" \cite{DB,TW} of the prior and the solution are the same. Loosely speaking, a most probable path is similar to a {\em mode} for the path space measure $P$.  More precisely, if both drift $b(\cdot,\cdot)$ and diffusion coefficient $\sigma(\cdot,\cdot)$ of the Markov diffusion process
\[dX_t=b(t,X_t)dt+ \sigma(t,X_t)dW_t
\]
are smooth and bounded, with $\sigma(t,x)\sigma(t,x)'>\eta I$, $\eta>0$, and $x(t)$ is a path of class $C^2$, then there exists an asymptotic estimate of the probability $P$ of a small tube around $x(t)$ of radius $\epsilon$. It  follows from this estimate that the most probable path is the minimizer in a deterministic calculus of variations problem where the Lagrangian is an {\em Onsager-Machlup functional}, see \cite[p.\ 532]{IW} for the full story{\footnote {The Onsager-Machlup functional was introduced in {\cite{OM} to develop a theory of fluctuations in equilibrium and nonequilibrium thermodynamics.} }.

The concept of most probable path is, of course, much simpler in our discrete setting. We now define this for general positive measures on paths.

Given a positive measure $\fM$ as in Section \ref{GBProblems} on the feasible paths of our graph $\bf G$, we say that $x=(x_0,\ldots,x_N)\in {\mathcal {FP}}_0^N$ is of {\em maximal mass} if for all other feasible paths $y\in{\mathcal {FP}}_0^N$ we have $\fM(y)\le\fM(x)$. Likewise we consider paths of {\em maximal mass} connecting particular nodes. It is apparent that paths of maximal mass always exist but are, in general, not unique. If $\fM$ is a probability measure, then the maximal mass paths--most probable paths, are simply the modes of the distribution. We establish below that the maximal mass paths joining two given nodes under the solution of a Schr\"odinger Bridge problem as in the previous section are the same as for the prior measure.

\begin{proposition} \label{invariance}
Consider marginals $\nu_0$ and $\nu_N$ in Problem \ref{prob:optimization}. Assume that $\nu_0(x)>0$ on all nodes 
$x\in\mathcal X$ and that the product $M(0)M(1) \cdots M(N-2)M(N-1)$ of transition  matrices of the prior has all positive elements (cf.\ with $M$'s as in \eqref{eq:matrixM}). Let  $x_0$ and $x_N$ be any two nodes. Then, under the solution $\fM^*[\nu_0,\nu_N]$ of the SBP, the family of maximal mass paths joining $x_0$ and $x_N$ in $N$ steps is the same as under the prior measure $\fM$.
\end{proposition}
\proof
{Suppose path $y=(y_0=x_0,y_1,\ldots, y_{N-1},y_N=x_N)$ has maximal mass under the prior $\fM$. In view of (\ref{prior}) and (\ref{OPTTRANSITION1}) and assumption (\ref{eq:mupositive}), we have
\begin{eqnarray}\nonumber\fM^*[\nu_0,\nu_N](y)&=&\nu_0(y_0)\pi_{y_0y_1}(0)\cdots \pi_{y_{N-1}y_N}(N-1)\\
&=&\frac{\nu_0(x_0)}{\mu_0(x_0)}\frac{\varphi(N,x_N)}{\varphi(0,x_0)}\fM(y_0,y_1,\ldots,y_N).\nonumber
\end{eqnarray}
Since the quantity
\[\frac{\nu_0(x_0)}{\mu_0(x_0)}\frac{\varphi(N,x_N)}{\varphi(0,x_0)}
\]
is positive and does not depend on the particular path joining $x_0$ and $x_N$, the conclusion follows.}
\qed

The above calculation establishes in fact the following stronger result. 
\begin{proposition} Let  $x_0$ and $x_N$ be any two nodes in $\mathcal X$.
Then, under the assumptions of Proposition \ref{invariance}, the measures $\fM$ and $\fM^*[\nu_0,\nu_N]$, restricted on the set of paths that begin at $x_0$ at time $0$, and end at $x_N$, at time $N$, are identical.
\end{proposition}

\subsection{Robust network routing}\label{RNR}
We now discuss yet another possible usage and interpretation of Schr\"odinger Bridges, motivated by a concept of robustness of a transportation plan for a given network.

Network robustness is typically understood as the ability of a network to maintain connectivity, or to be insensitive (observables), in the event of node or link failures, or a disturbance.
Maintaining connectivity may be seen as an {\em inverse percolation problem} \cite{Bar1,Bar2}. There exist several other notions of robustness such the one defined through a {\em fluctuation-dissipation relation} involving the {\em topological entropy rate}. This notion captures the behaviour while relaxing back to equilibrium after a perturbation, see \cite{Dem1,Dem2,Sandhu1,Sandhu2}. Also, robust network design to meet demands in a given uncertainty set was studied in \cite{Olver}. Finally, a concept of {\em resilience} of a routing policy in  the presence of {\em cascading failures} introduced and studied  in \cite{SCD}. Following the latter rationale, linking robustness to resilience, equilibrium considerations play no role. Indeed, this rationale motivates maximal utilization of all available options equally, as much as possible, so as to preempt failures. 

Thus, we are led to formulate the following problem: Given times $t=0,1,\ldots,N$ and a directed, strongly connected graph, find a transportation plan from a {\em source node} to a {\em sink node}\footnote{We always have a loop on the sink node to allow part of the mass to arrive there earlier than the planned time horizon.} such that most of the mass arrives by time $N$ even in the presence of  failures (of course, more general initial and final distributions can also be similarly treated). For instance, consider the graph of Figure \ref{fig:graph}, 
with the total mass residing at node $1$ at time $t=0$, and the requirement to transport the total mass to node $9$ in $N=3$, or $N=4$, time steps.
A natural idea is to spread the mass as much as the topology of graph permits, before reassembling the mass at the target end-point distribution (here at node $9$). The property of ``spreading'' of distributions along interpolation between end-point marginals, whether in OMT or Schr\"odinger Bridge problems, is often referred to as ``lazy gas'' \cite{Vil,Vil2}. Thus, it is natural to consider transportation plans that share such a property via utilizing the framework of Schr\"{o}dinger bridges!

However, the current task, to transport provides no ``prior'' measure. 
It is simply a problem in transportation. Can we select a suitable measure? Perhaps, select as prior a prespecified plan that no longer meets desired transportation requirements? Or, modify a transportation plan by adding costs to mediate congestion? All of these directions are possible and can be profitably pursued in engineering problems. Yet, herein we motivate a different prior, one that {\em maximally} spreads mass utilizing available options. Indeed, such a prior must be a sort of {\em uniform distribution} on paths. But what does that exactly mean in this setting? Fortunately, such a notion already exists, and for this we follow the interesting paper  \cite{dellib}. But let us first recall a most famous result in Linear Algebra \cite{Horn}.

\begin{figure}[htbp]
\begin{center}
\includegraphics[width=0.45\textwidth]{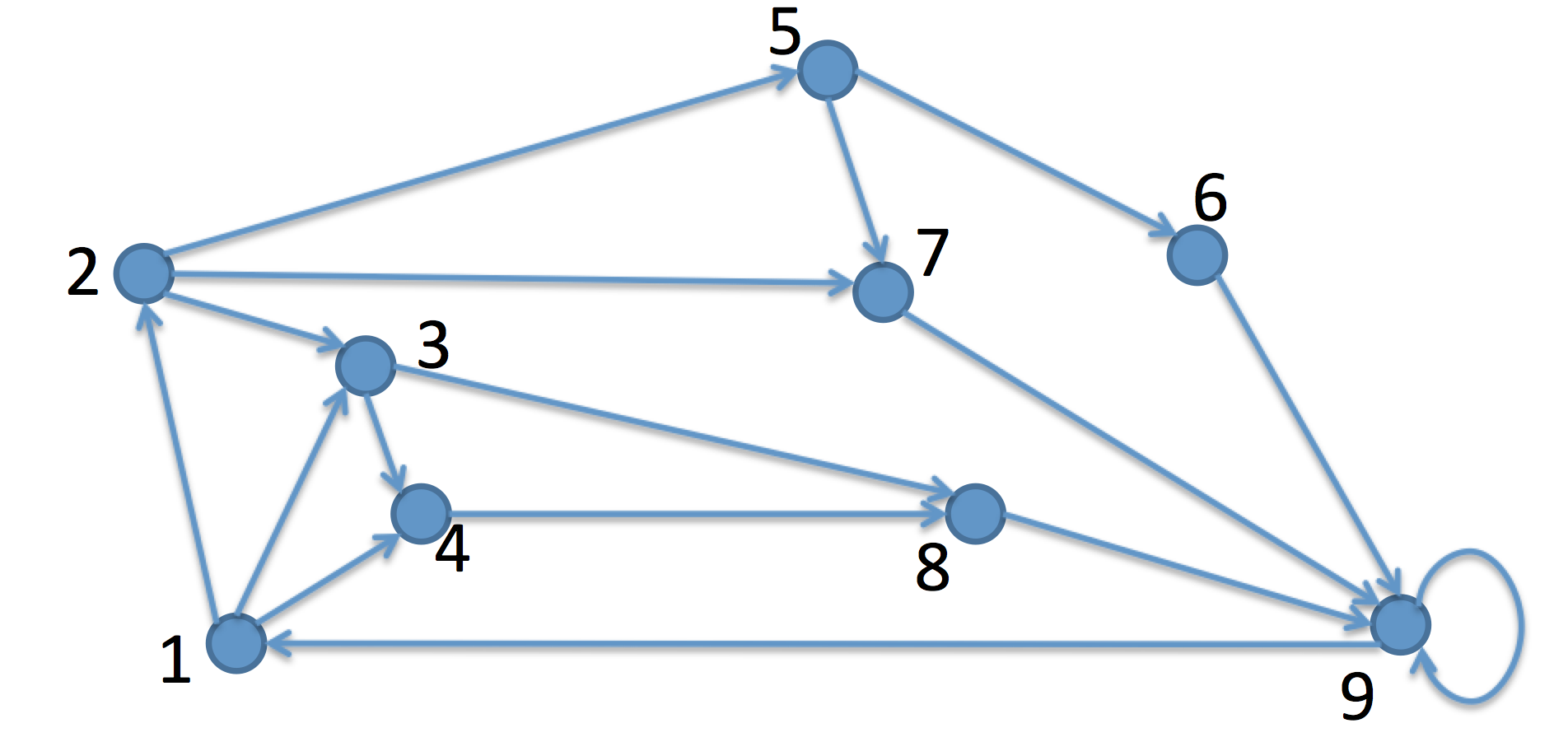}
\caption{Transportation network}\label{fig:graph}
\end{center}
\end{figure}

\begin{theorem}[Perron-Frobenius]\label{perrontheorem}
Let $A=\left(a_{ij}\right)$ be an $n\times n$ matrix with nonnegative entries. Let $\lambda_A=\max\{|\lambda_1|,|\lambda_2|,\ldots,|\lambda_n|\}$ be its spectral radius. Suppose there exists $N$ such that $A^N$ has only positive entries. Then
\begin{enumerate}
\item[i)] $\lambda_A>0$ is an eigenvalue of $A$;
\item[ii)] $\lambda_A$ is a simple eigenvalue;
\item[iii)] there exists an eigenvector $v$ corresponding to $\lambda_A$ with strictly positive entries;
\item[iv)] $v$ is the only non-negative eigenvector of A;
\item[v)] \label{item_v} let $B=[b_{ij}]$ be an $n\times n$ matrix with nonnegative entries. If $a_{ij}\le b_{ij}$, $\forall i,j\leq n$ and $A\neq B$, then $\lambda_A<\lambda_B$. 
\end{enumerate}
\end{theorem}

Returning to graphs we discuss first the notion of topological entropy, namely, the rate by which the cardinality of the set of paths of length $N$ increases, as $N\to \infty$.
To this end, we consider a strongly connected directed graph ${\bf G}=(\mathcal X,\mathcal E)$ as before.
The {\em topological entropy rate} is
\[
H_{\bf G}=\limsup_{N\rightarrow\infty}[\log |\{{\rm paths\; of \;length} \;N\}|/N].
\]
If $A$ denotes the {\em adjacency matrix} of the graph, it is easy to show\footnote{This follows from the fact that the $ij$th entry of $A^N$ (a positive integer) enumerates the number of distinct paths from vertix $i$ to vertix $j$, in $N$ steps.}
that  
\[H_{\bf G}= \log(\lambda_A).
\]
Since the graph is stongly connect, $A^N$ has only positive entries. Let $\hat \varphi$ and $\varphi$ be its left and right eigenvectors with positive entries corresponding to $\lambda_A$ (Theorem \ref{perrontheorem}), so that
\[A^T\hat\varphi =\lambda_A \hat\varphi , \quad A\varphi =\lambda_A \varphi,
\]
and select/scale those so that
\[
\langle \hat\varphi ,\varphi\rangle:=\sum_i \hat\varphi _i\varphi_i=1.
\]
Then
\begin{equation}\label{invariantmeasure1}
\nu_{RB}(i)=\hat\varphi _i\varphi_i  
\end{equation}
defines a probability distribution on ${\cal X}$  which is invariant under the transition matrix
\begin{equation}
R=[r_{ij}], \quad r_{ij}=\frac{1}{\lambda_A }\frac{\varphi_j}{\varphi_i}a_{ij}.\label{Opttransition1}  
\end{equation}
that is,
\[
R^T\nu_{RB}=\nu_{RB}.
\]

The transition matrix $R$ in (\ref{Opttransition1}), together with stationary measure $\nu_{RB}$ in (\ref{invariantmeasure1}), define the {\em Ruelle-Bowen }({\em Markovian}) {\em path measure}
\[
\fM_{\rm RB}(x_0,x_1,\ldots,x_N):=\nu_{RB}(x_0)r_{x_0x_1}\cdots r_{x_{N-1}x_N}.
\]
Equation \eqref{Opttransition1} brings up the unmistakeable links to the structure of Schr\"odinger Bridges that maximize entropy. But here, a deeper fact is at play, in that the Ruelle-Bowen distribution \cite{Parry,Ruelle} represents a uniform distribution on paths, made precise by the following remarkable proposition.
\begin{proposition} \label{uniformpath} The measure $\fM_{\rm RB}$ assigns probability $\lambda_A^{-t}\hat\varphi_i\varphi_j$ to any path of length $t$ from node $i$ to node $j$.
\end{proposition}

\begin{proof} Starting from the stationary distribution \eqref{invariantmeasure1}, and in view of \eqref{Opttransition1}, the probability of a path $ij$ is
\[
\hat \varphi_i\varphi_i \left(\frac{1}{\lambda_A} \varphi_i^{-1}\varphi_j\right)=\frac{1}{\lambda_A}\hat \varphi_i\varphi_j,
\]
assuming that node $j$ is accessible from node $i$ in one step.
Likewise, if node $k$ is accessible from $j$, the probability of the path $ijk$ is
\[
\hat \varphi_i\varphi_i \left(\frac{1}{\lambda_A} \varphi_i^{-1}\varphi_j\right)\left(\frac{1}{\lambda_A} \varphi_j^{-1}\varphi_k\right)=\frac{1}{\lambda_A^2}\hat \varphi_i\varphi_k
\]
independent of the intermediate state $j$, and so on.
\end{proof}

Shannon entropy of paths of length $N$ grows like $N\log\lambda_A$. Thus, since the {\em entropy rate} of this particular distribution is $\log\lambda_A=H_{\bf G}$, it is indeed the maximum possible!

We note that the stationary measure $\nu_{RB}$ of this law was used in \cite{dellib} 
in order to obtain a  {\em centrality measure} (entropy ranking) similar to the Google Page ranking but more robust and discriminating. In the exposition herein, instead, $\fM_{\rm RB}$ is a natural choice as a  prior distribution in the Schr\"{o}dinger bridge problem to achieve the spreading of the mass. Moreover, as just noted earlier, the Ruelle-Bowen measure $\fM_{\rm RB}$ on paths may be itself seen as the solution of a Schr\"{o}dinger bridge problem where the ``prior" transition matrix is the adjacency matrix $A$ and the two marginals are $\nu_0=\nu_N=\nu_{RB}$, see \cite[Section 4]{CGPT1}. 

Returning to the transportation problem, we seek
 \[
\fM^*[\delta_1,\delta_n]={\rm argmin}\{ \D(P\|\fM_{\rm RB}) \mid  P\in {\mathcal P}(\delta_1,\delta_n)\}
\]
By Theorem \ref{solbridge}, the solution is the Markovian evolution starting at $t=0$ with the distribution $\delta_1$ and with transition matrix 
\[
\Pi^*(t)=\diag(\varphi(t))^{-1}R \diag(\varphi(t+1)),
\]
where
\[
	\varphi(t)=R\varphi(t+1), ~\hat{\varphi}(t+1)=R^T\hat{\varphi}(t),
\]
with the boundary conditions
\[
\varphi(0,x)\hat{\varphi}(0,x)=\delta_{1}(x),\;\varphi(N,x)\hat{\varphi}(N,x)=\delta_{n}(x), \quad \forall x\in\mathcal X.
\]
 \noindent
 Thus, the solution $\fM^*[\delta_1,\delta_n]$ is a bridge over $\fM_{\rm RB}$ which is itself a bridge. Namely the solution is a {\em bridge over a bridge}. 
 
 We conclude with a remarkable  {\em iterated bridge property} of the Schr\"{o}dinger bridges (not to be confused with an iterated I-projection property \cite{csiszar0,csiszar1}).  Suppose first the prior is a probability distribution; this defines a {\em reciprocal class} \cite{Jam1,levy1990modeling} of distributions. The solution to the Schr\"odinger Bridge problem is in fact the unique Markovian evolution in the same reciprocal class as the prior (same three times transition probabilities). I.e., if we take the solution as a prior for a new bridge problem, the reciprocal class stays the same. Notice that this is the case even when there is loss/creation of mass in the prior evolution, see \cite[Section IIA]{CGPT2} for the details. Hence, the new distribution on paths $\fM^*[\delta_1,\delta_n]$ can be obtained solving a {\em unique} bridge problem with prior transition the adjacency matrix $A$ and marginals $\delta_1$ and $\delta_n$, see \cite{CGPT1}. The solution may be computed through an iterative algorithm like the one  described in Section \ref{ITERATIVEALGORITHM} where $G=A^N$.  We exemplify the steps and rationale presented with the following academic exercise.

\begin{example}\em 
Consider Figure \ref{fig:graph} and let $\nu_0=\delta_1$, $\nu_N=\delta_9$. Take first $N=3$. The shortest path from node $1$ to $9$ is of length $3$ and there are three such paths, which are
$1-2-7-9$, $1-3-8-9$ and $1-4-8-9$.  Using $\fM_{\rm RB}$ as the prior, then we get a transport plan with {\em equal probabilities} for all these three paths. The evolution of the mass distribution is given by the rows of the following matrix, row $i$ representing the distribution on the nodes at time $t=i, i=0,1,2,3$
 \[ 
       \left[
       \begin{matrix}
       1 & 0 & 0 & 0 & 0 & 0 & 0 & 0 & 0\\
       0 & 1/3 & 1/3 & 1/3 & 0 & 0 & 0 & 0 & 0\\
       0 & 0 & 0 & 0 & 0 & 0 & 1/3 & 2/3 & 0\\
       0 & 0 & 0 & 0 & 0 & 0 & 0 & 0 & 1
       \end{matrix}
       \right].
\]

 For $N=4$, the mass spreads even more before reassembling at node $9$
 \begin{center}
 \[
       \left[
       \begin{matrix}
       1 & 0 & 0 & 0 & 0 & 0 & 0 & 0 & 0\\
       0 & 4/7 & 2/7 & 1/7 & 0 & 0 & 0 & 0 & 0\\
       0 & 0 & 1/7 & 1/7 & 2/7 & 0 & 1/7 & 2/7 & 0\\
       0 & 0 & 0 & 0 & 0 & 1/7 & 1/7 & 2/7 & 3/7\\
       0 & 0 & 0 & 0 & 0 & 0 & 0 & 0 & 1
       \end{matrix}
       \right]
\]
\end{center}
\end{example}

\subsection{Optimal flows on weighted graphs}

With the same notation as in Section \ref{GBProblems}, we now suppose that to each edge $ij$ is associated a length $l_{ij}\ge 0$. If $ij\not\in\mathcal E$, we set $l_{ij}=+\infty$. The length may represent distance, cost of transport, cost of communication, inverse capacity of the link, and so on. As an example, suppose a relief organization, operating in an area where a natural disaster has occurred or an epidemic  or in a war zone, needs to transport resources. At the initial time $t=0$, there is a  distribution  $\nu_0(x)$ of available relief goods in sites $x\in\cal X$. Using the available road network, the goods must reach certain other locations after $N$ units of time to be distributed according to a desired distribution $\nu_N(x)$.  On the one hand, since the feasibility of the various possible routes is uncertain, it is desirable that the goods spread as much as the road network allows before reaching the target nodes. But at the same time, it is also important that shorter paths are used to keep the fuel consumption within the available budget.
In such a scenario it is possible to repeat the construction of Section\ref{RNR}, replacing the adjacency matrix $A$ with a weighted adjacency matrix $B$
\[B=\left[b_{ij}\right]=\left[\exp\left(-l_{ij}\right)\right],
\]
again assuming that $B^N$ has positive entries representing cost, thereby allowing us to employ the Perron-Frobenius theorem.

The measure $\fM_L$ that replaces the Ruelle-Bowen measure is, of course, no longer ``uniform".  Given the development of Section \ref{thermostatics}, we already know what we can expect. Rather than maximizing entropy, it is natural to seek a compromise between the latter goal and that of minimizing length/energy/cost. Indeed, let $\hat\varphi$ and $\varphi$ be left and right eigenvectors with positive entries of the matrix $B$ corresponding to the spectral radius $\lambda_B$ of $B$, so that
$$B^T\hat\varphi=\lambda_B \hat\varphi, \quad B\varphi=\lambda_B \varphi.
$$ 
Suppose once again that $\hat\varphi$ and $\varphi$ are chosen so that $\langle \hat\varphi,\varphi\rangle=\sum_i \hat\varphi_i\varphi_i=1$. Then $\mu_L$ given by
\begin{equation}\label{invariantmeasure}
\mu_L(i)=\hat\varphi_i\varphi_i
\end{equation}
is a probability distribution which is invariant for the transition matrix
\begin{equation}\label{opttransition}R_L=\lambda_B^{-1}\diag(\hat\varphi_i\varphi_i)^{-1}B\diag(\hat\varphi_i\varphi_i), 
\end{equation}
namely
\begin{equation}\label{optevol}
R_L^T\mu_L=\mu_L.
\end{equation}
As expected, the corresponding path measure $\fM_L$ is no longer uniform on paths of equal length joining two specific nodes. Indeed, the probability of the path $(i=x_0,x_1,\ldots, x_{t-1},j=x_t)$ is 
\[
	\lambda_B^{-t}\exp (-\sum_{k=0}^{t-1}l_{x_k x_{k+1}})\hat\varphi_i \varphi_j.
\]  
It is namely the minimum {\em free energy rate} distribution  (topological pressure in thermodynamics) attaining the minimum value,
which is $-\log\lambda_B$, and has therefore the form of a {\em Boltzmann distribution} (\ref{Boltdistr}), see \cite[Section IV]{dellib} for details.  Indeed, for a path $x=(x_0,\ldots,x_N)\in\mathcal X^{N+1}$,  define the length of $x$ to be 
\[l(x)=\sum_{t=0}^{N-1}l_{x_tx_{t+1}},
\]
and for any distribution $P$ on ${\cal X}^{N+1}$, the {\em average path length}
\begin{equation}\label{internal}
L(P)=\sum_{x\in\mathcal X^{N+1}}l(x)P(x).
\end{equation}
This plays the same role as the internal energy in the state $P$  of Section \ref{thermostatics}, corresponding  the length $l(x)$ of $x$ with the energy $E_x$. Clearly, $L(P)$ is finite if and only if $P$ is supported  on actual, existing paths of ${\bf G}$. The Boltzmann distribution (\ref{Boltdistr}) on $X^{N+1}$ is then
\begin{equation}\label{pathBoltdistr}p_B(x)=\fM_L(x)=Z(T)^{-1}\exp\left[-\frac{l(x)}{kT}\right],\mbox{ for }Z(T)=\sum_{x\in \mathcal X}\exp\left[-\frac{l(x)}{kT}\right].
\end{equation}
Note that the the support of the Boltzmann distribution $\support(p_B)$ is contained in ${\mathcal {FP}}_0^N$. 
By statement v) in Theorem \ref{perrontheorem}, we then have that $\log\lambda_A<\log\lambda_B$, namely, the topological entropy increases in a way that is consistent with intuition. 

We can  take $\fM_L$  as the prior distribution in a maximum entropy problem as in Section \ref{RNR} obtaining again through the solution $\fM^*_L[\delta_{1},\delta_{n}]$ a robust-efficient transportation plan from node $1$ to node $n$. 
We are now ready to prove a striking result which generalizes Proposition \ref{uniformpath}.

\begin{theorem}\label{robustOMT} $\fM_L^*[\delta_{1},\delta_{n}](x)$ assigns equal probability to paths $x\in\mathcal X^{N+1}$ of equal cost. In particular, it assigns maximum and equal probability to minimum length paths.
\end{theorem}
\begin{proof}For a path $x=(x_0,\,x_1,\ldots,x_N)$, we have
\begin{eqnarray}\nonumber\fM_L^*[\delta_{1},\delta_{n}](x)=\delta_{1}(x_0)\frac{\varphi_v(N,x_N)}{\varphi_v(0,x_0)}\prod_{t=0}^{N-1} b_{x_t x_{t+1}}\\=\delta_{1}(x_0)\frac{\varphi_v(N,x_N)}{\varphi_v(0,x_0)}\exp[-\sum_{t=0}^{N-1} l_{x_t x_{t+1}}].
\end{eqnarray}
Observe once more that $\delta_{1}(x_0)\frac{\varphi_v(N,x_N)}{\varphi_v(0,x_0)}$ does not depend on the particular path joining $x_0$ and $x_N$. Since $\sum_{t=1}^{N-1} l_{x_t x_{t+1}}=l(x)$ is the total length of the path, the conclusion now follows. 
\end{proof}
\begin{remark}{\em 
In the discrete (OMT) problem \cite[Vol.I]{RR} introduced in Section \ref{RLPP}, one first seeks to identify the shortest path(s) $(x_0,x_1^*,\ldots,x_{N-1}^*,x_N)$ from any starting node $x_0\in\mathcal X$ to any ending node $x_N$,
\begin{equation}\label{minlength}
l_{{\rm min}}(x_0x_N)=\min_{x_1^*,\ldots,x_{N-1}^*} \sum_{t=0}^{N-1}l_{x^*_tx^*_{t+1}}.
\end{equation}
This is a {\em combinatorial problem} but can also be cast as a linear program \cite{bazaraa2011linear}. It is apparent that the computational complexity of such a problem becomes rapidly unbearable as the number of nodes $n$ and the length of the path $N$ increase.
Having a solution to this first problem, the OMT problem can then be recast as the linear program  in (\ref{DW}), where the cost of a path is its length.
Alternatively, the OMT problem can be directly cast as a linear program in as many variables as there are edges  \cite{bazaraa2011linear}. 
The transport provided by Theorem \ref{robustOMT}, which readily generalizes to any two marginals $\nu_0$ and $\nu_N$, provides an attractive alternative to the OMT approach: Minimum length paths all have maximum probability, but some of the mass is also transported on alternative paths thereby ensuring a certain amount of robustness of the transportation plan. Also notice that Theorem \ref{robustOMT} provides an alternative paradigm to find the minimum length paths through simulation! }\hfill$\Box$
\end{remark}

We conclude this subsection with a few observations on the role of the temperature parameter, referring to \cite[Section V]{CGPT2} for the proofs.
\begin{remark}\label{remtem}{\em
Consider the solution $\fM^*_{L,T}[\delta_{x_0},\delta_{x_N}]=:\fM^*_T$ to the maximum entropy problem  \[
\fM_{L,T}^*[\delta_1,\delta_n]={\rm argmin}\{ \D(P\|p_B(x;T)) \mid  P\in {\mathcal P}(\delta_1,\delta_n)\}
\]
with prior  (\ref{pathBoltdistr}) where we have emphasized the dependence on the parameter $T$.
Let $l_{{\rm min}}(x_0x_N)$ be as in (\ref{minlength}).
\begin{itemize}
\item[i)] For $T\searrow 0$, $\fM^*_T$ tends to concentrate itself on the set of feasible, minimum length paths joining $x_0$ and $x_N$ in $N$ steps. Namely, if $y=(y_0=x_0,y_1,\ldots,y_{N-1},y_N=x_N)$ is such that $l(y)>l_{{\rm min}}(x_0x_N)$, then $\fM^*_T(y)\searrow 0$ as $T\searrow 0$.
\item[ii)] For $T\nearrow+\infty$, $\fM^*_T$ tends to the uniform distribution on all feasible paths joining $x_0$ and $x_N$ in $N$ steps.
\end{itemize}
We notice that, as in the diffusion case \cite{Mik, mt, MT,leo,leo2,CGP3,CGP4,CDPS},  when the ``heat bath" temperature $T$ is close to $0$, the solution of the Schr\"{o}dinger bridge problem is close to the solution of the discrete OMT problem. Since for the former an efficient iterative algorithm is available (\ref{Iteration}), we see that also in this discrete setting the SBP provides a valuable computational approach to solving OMT problems. We illustrate this, Theorem \ref{robustOMT} and the invariance of the most probable paths below, in a simple example below.
}
\hfill$\Box$\end{remark}
\begin{example}\em
Consider again Figure \ref{fig:graph} and let $\nu_0=\delta_1$ and $\nu_N=\delta_9$. Let the time horizon for the transport $N=3$ or $N=4$. We first set the length of all edges equal to $1$ except $l_{99}=0$.  The shortest path from node $1$ to $9$ is of length $3$ and there are three such paths, which are
$1-2-7-9$, $1-3-8-9$ and $1-4-8-9$. If we want to transport the mass with a minimum number of  steps, we may end up using one of these three paths. We use the results of Section \ref{RNR} to compute a robust routing policy. Since all the three feasible paths have equal length, we get a transport plan with equal probabilities using all these three paths, regardless of the choice of temperature $T$. The evolution of mass distribution is given by
 \[
           \left[
       \begin{matrix}
       1 & 0 & 0 & 0 & 0 & 0 & 0 & 0 & 0\\
       0 & 1/3 & 1/3 & 1/3 & 0 & 0 & 0 & 0 & 0\\
       0 & 0 & 0 & 0 & 0 & 0 & 1/3 & 2/3 & 0\\
       0 & 0 & 0 & 0 & 0 & 0 & 0 & 0 & 1
       \end{matrix}
       \right],
\]
where the four rows of the matrix show the mass distribution at time step $t=0, 1, 2 ,3$ respectively, while columns correspond to vertices. As we can see, the mass spreads out first and then goes to node $9$. When we allow for more steps $N=4$, we get, for ``temperature'' $T=1$,
 \[
           \left[
       \begin{matrix}
       1 & 0 & 0 & 0 & 0 & 0 & 0 & 0 & 0\\
       0 & 0.4705 & 0.3059 & 0.2236 & 0 & 0 & 0 & 0 & 0\\
       0 & 0 & 0.0823 & 0.0823 & 0.1645 & 0 & 0.2236 & 0.4473 & 0\\
       0 & 0 & 0 & 0 & 0 & 0.0823 & 0.0823 & 0.1645 & 0.6709\\
       0 & 0 & 0 & 0 & 0 & 0 & 0 & 0 & 1
       \end{matrix}
       \right].
\]
There are $7$ feasible paths of length $4$, which are $1-2-7-9-9$, $1-3-8-9-9$, $1-4-8-9-9$, $1-2-5-6-9$, $1-2-5-7-9$, $1-3-4-8-9$ and $1-2-3-8-9$. The amounts of mass traveling along these paths are 
	\[
		0.2236, 0.2236, 0.2236, 0.0823, 0.0823, 0.0823, 0.0823,
	\]
	respectively.
The first three are the most probable paths. This is consistent with Proposition \ref{invariance} since they are the paths with minimum length. If we change the temperature $T$, the flow changes. The set of most probable paths, however, remains invariant. In particular, when $T=0.1$, the flow concentrates on the most probable set (effecting OMT-like transport), as shown below
 \[
       \left[
       \begin{matrix}
       1 & 0 & 0 & 0 & 0 & 0 & 0 & 0 & 0\\
       0 & 0.3334 & 0.3333 & 0.3333 & 0 & 0 & 0 & 0 & 0\\
       0 & 0 & 0 & 0 & 0 & 0 & 0.3334 & 0.6666 & 0\\
       0 & 0 & 0 & 0 & 0 & 0 & 0 & 0 & 1\\
       0 & 0 & 0 & 0 & 0 & 0 & 0 & 0 & 1
       \end{matrix}
       \right].
\]

Now we change the graph by setting the length of edge $(7,\,9)$ as $2$, that is, $l_{79}=2$. When $N=3$ steps are allowed to transport a unit mass from node $1$ to node $9$, the evolution of mass distribution for the optimal transport plan, for $T=1$, is given by
 \[
          \left[
       \begin{matrix}
       1 & 0 & 0 & 0 & 0 & 0 & 0 & 0 & 0\\
       0 & 0.1554 & 0.4223 & 0.4223 & 0 & 0 & 0 & 0 & 0\\
       0 & 0 & 0 & 0 & 0 & 0 & 0.1554 & 0.8446 & 0\\
       0 & 0 & 0 & 0 & 0 & 0 & 0 & 0 & 1
       \end{matrix}
       \right].
\]
The mass is transported through paths $1-2-7-9$, $1-3-8-9$ and $1-4-8-9$, but unlike the first case, the transport plan doesn't equilize probability for these three paths. Since the length of the edge $(7,\,9)$ is larger, the probability that the mass takes this path becomes smaller. The plan does, however, assign equal probability to the two paths $1-3-8-9$ and $1-4-8-9$ with minimum length; that is, these are the most probable paths. The evolutions of mass for $T=0.1$ and $T=100$ are 
\[
       \left[
       \begin{matrix}
       1 & 0 & 0 & 0 & 0 & 0 & 0 & 0 & 0\\
       0 & 0 & 1/2 & 1/2 & 0 & 0 & 0 & 0 & 0\\
       0 & 0 & 0 & 0 & 0 & 0 & 0 & 1 & 0\\
       0 & 0 & 0 & 0 & 0 & 0 & 0 & 0 & 1
       \end{matrix}
       \right]
\]
and
\[
       \left[
       \begin{matrix}
       1 & 0 & 0 & 0 & 0 & 0 & 0 & 0 & 0\\
       0 & 0.3311 & 0.3344 & 0.3344 & 0 & 0 & 0 & 0 & 0\\
       0 & 0 & 0 & 0 & 0 & 0 & 0.3311 & 0.6689 & 0\\
       0 & 0 & 0 & 0 & 0 & 0 & 0 & 0 & 1
       \end{matrix}
       \right],
\]
respectively. We observe that, when $T=100$ the flow assigns almost equal mass to the three available paths, while, when $T=0.1$ (OMT-like transport), the flow concentrates on the most probable paths $1-3-8-9$ and $1-4-8-9$. This is clearly a consequence of the properties seen in Remark \ref{remtem}.
\end{example}
\begin{remark}{\em  As in the continuous case, it is possible to transform the dynamic problems considered in this section into static ones using a decomposition for the relative entropy similar to (\ref{decomposition}). Indeed, let $P$ and $Q$ be two probability distributions on $\mathcal X^{N+1}$. For $x=(x_0,x_1,\ldots,x_N)\in\mathcal X^{N+1}$, consider the multiplicative decomposition
$$P(x)=P_{x_0,x_N}(x)p_{0N}(x_0,x_N),
$$
where
\[P_{\bar{x}_0,\bar{x}_N}(x)=P(x|x_0=\bar{x}_0,x_n=\bar{x}_N)\]
and we have assumed that the joint initial-final distribution $p_{0N}$ is everywhere positive on $\mathcal X\times \mathcal X$, and similarly  for $Q$.
We get
\begin{eqnarray}\nonumber
\D(P\|Q)&=&\sum_{x_0x_N}p_{0N}(x_0,x_N)\log \frac{p_{0N}(x_0,x_N)}{q_{0N}(x_0,x_N)}\\&+&\sum_{x\in{\cal X}^{N+1}}P_{x_0,x_N}(x)\log \frac{P_{x_0,x_N}(x)}{Q_{x_0,x_N}(x)} p_{0N}(x_0,x_N).\nonumber
\end{eqnarray}
This is the sum of two nonnegative quantities. The second becomes zero if and only if $P_{x_0,x_N}(x)=Q_{x_0,x_N}(x)$ for all $x\in{\cal X}^{N+1}$.}
\hfill$\Box$\end{remark}

Thus, expanding on the above remark, for instance, the problem 
\begin{equation}\label{Ddynamic}{\rm min}\{ \D(P\|p_B(x;T)) \mid  P\in {\mathcal P}(\nu_0,\nu_N)\}
\end{equation}
 can be reduced to
\begin{problem}\label{discretestatic}
 \begin{equation}
 {\rm minimize}\;J(p_{0N}):=\D(p_{0N}\|p_{B;0N})
 \end{equation}
 over 
 \begin{equation}\label{distrfamily}\{p_{0N} \;{\rm probability\; distribution\; on}\;\mathcal X\times\mathcal X:  \sum_{x_N}p_{0N}(\cdot,x_N)=\nu_0(\cdot), \sum_{x_0}p_{0N}(x_0\cdot)=\nu_N(\cdot)\}.
\end{equation}
\end{problem}

If $p^*_{0N}$ solves the above problem, then
\begin{equation}\label{optsolution}P^*(x)=p_{B;x_0,x_N}(x)p^*_{0N}(x_0,x_N)
\end{equation}
solves Problem (\ref{Ddynamic}). As in the continuous setting, the solution lies in the same reciprocal class of the prior. Observe, however, that, contrary to (\ref{OPTTRANSITION1}), the solution in the form (\ref{optsolution}) does not yield immediate by-product information on the new transition probabilities and on what paths the optimal mass flow selects (cf. Remark \ref{Cdynvsstatic} in the continuous setting). It is therefore less suited for many network routing applications.


\section{Closing comments}
We have reviewed some of the essential theoretical features of the Kantorovich relaxed formulation of OMT with quadratic cost and of the theory of the Schr\"odinger bridge problem, SBP.
We have tried to do this from a somewhat unusual angle, namely that of stochastic control.
We have explained that such a viewpoint opens the way to natural generalizations and important applications, particularly in physics (such as cooling), aerospace engineering (such as guidance for spacecrafts), robotics (such as controlling swarms), and many other areas.
Both problems, OMT and SBP, in their Eulerian formulation, lead to stochastic control steering problems for probability distributions. In this respect, we have highlighted the non-equivalence between Schr\"odinger's original problem and the related problems of steering for one-time densities, cf. Remark \ref{deterministicflow}. We have also attempted to clarify the relation between Yasue's action, Carlen's problem, and the fluid dynamic Problem \ref{fluidsymm} introduced in \cite{CGP4} that involves a Fisher Information Functional, see Remark \ref{carlenfisher}.  

In order to keep the paper within a reasonable length, we have avoided touching here on a large number of related topics such as non-commutative OMT, gradient flows on Wasserstein spaces, geometry of displacement and entropic interpolation, multimarginal transport, functional inequalities, mismatched transport, barycenters, connection to mean-field games, to information theory,  to stochastic mechanics, to differential games, to name but a few.

In the discrete setting, we have shown in some detail the connection between the relaxed OMT, Schr\"odinger bridges and the statistical mechanical free energy. It turns out (Section \ref{connections}) that the regularized OMT is just a discrete Schr\"odinger bridge problem equivalent to a contrained minimization of the free energy for distributions on path space. The latter may be viewed as a constrained {\em Gibbs' variational principle} recovering Schr\"odinger's  original large deviations motivation.  Moreover, the IPF-Sinkhorn algorithm is just a discrete counterpart of Fortet's algorithm (Section \ref{FIPFS}) for which Fortet had proven convergence under rather general assumptions in the much more challenging continuous case already in 1940. In Section \ref{FIPFS}, we have also proven convergence of the algorithm as a consequence of convergence in a projective metric of rays in the positive orthant of $\R^n$.

This paper has also the ambition to act as a {\em navigation chart} for the topics it discusses. We felt that this was needed as, on the one hand, probabilists considered the theory of Schr\"odinger Bridges pretty much complete by the early 1990's (see Wakolbinger's survey \cite{Wak} that was published in 1992).
Yet, as we argued at the beginning of Section \ref{anisotropic}, this is far from being true. On the other hand, some very fine analysts who contributed to the fantastic development of OMT of the past twenty or so years, have very little interest in the probabilistic motivation and the statistical physics underlying the Schr\"odinger Bridge theory in spite of the fine work of Mikami, Thieullen and L\'eonard \cite{Mik, mt, MT,leo,leo2}. Moreover, many first-class scientists working in the field have a more computational background and interest, having been driven to these problems by the effectiveness of the {\em earth mover distance} in several modern applications such as those (we quote from \cite{PC}) ``in imaging sciences (such as color or texture processing), computer
vision and graphics (for shape manipulation) or machine learning (for regression,
classification and density fitting)".  Finally, control engineers were interested from 1985 on in a special bridge problem on an infinite horizon problem, called covariance control, starting with the seminal work of Skelton and collaborators \cite{HS1,HS2,SkeIke,GS,ZGS}. The field was considered exhausted more than twenty years ago, but this turned out to be far from the case; it is currently experiencing another phase of fast development. 

On top of all of these reasons for incomprehensions and difficult communication, there is, of course, a Babelian confusion of Tongues as scientists working in this field have had a plethora of backgrounds such as in various areas of pure and applied mathematics, statistical physics, statistics, computer graphics, control, mechanical and aerospace engineering, numerical analysis, machine learning, etc.  
Nevertheless, considering the spectacular flourishing of OMT, Schr\"odinger Bridges and relative applications, one feels tempted to end the paper with a Chinese quote: ``Great is the confusion under the sky. The situation is therefore excellent.''


\end{document}